\documentclass[reqno,11pt,a4paper]{amsart}
\usepackage{cases}
\usepackage{amsmath}
\usepackage{graphicx}%%%%%%%%%%%%%% ¨¬¨ª?¨®¨ª??? %%%%%%%%%%%%%%%%%%%%%%%5
 \usepackage{mathrsfs}
\usepackage{bbm}
\usepackage{amsfonts}
\usepackage{amssymb}
\usepackage{color}
\usepackage{multirow}
\usepackage{enumerate}
\usepackage{epstopdf}
\usepackage{cite}
\usepackage{tikz}
\usepackage{algorithmic}
\usetikzlibrary{arrows.meta}
\usepackage{hyperref}
\usepackage{enumitem}
\usepackage{epsfig,subfigure,epstopdf}
\numberwithin{figure}{section}
\usepackage[margin=1in]{geometry}
\usepackage{array}
 \usepackage{appendix}
\usepackage{booktabs}

\allowdisplaybreaks

\renewcommand\abstractname{\textbf{Abstract}}

\usepackage{titlesec}
\titleformat{\section}{\vskip10pt\large\bfseries}{\thesection.}{0.5em}{\centering\vspace{5pt}}
\titleformat{\subsection}{\vskip10pt\normalsize\bfseries}{\thesubsection.}{0.5em}{}
\def\opn#1#2{\def#1{\operatorname{#2}}} % to make operators
\opn\chara{char} \opn\length{\ell}
\opn\projdim{proj\,dim} \opn\injdim{inj\,dim} \opn\rank{rank}
\opn\depth{depth} \opn\grade{grade} \opn\height{height}
\opn\embdim{emb\,dim} \opn\codim{codim}

\opn\Tr{Tr} \opn\bigrank{big\,rank}
\opn\superheight{superheight}\opn\lcm{lcm}
\opn\trdeg{tr\,deg}%
\opn\reg{reg} \opn\lreg{lreg}
\opn\Ker{Ker} \opn\Coker{Coker} \opn\Im{Im} \opn\Hom{Hom}
\opn\Tor{Tor} \opn\Ext{Ext} \opn\End{End} \opn\Aut{Aut} \opn\id{id}

\opn\nat{nat}
\opn\pff{pf}%   \pf exists already
\opn\Pf{Pf} \opn\GL{GL} \opn\SL{SL} \opn\mod{mod} \opn\ord{ord}
\let\to=\rightarrow

\def\Implies{\ifmmode\Longrightarrow \else
     \unskip${}\Longrightarrow{}$\ignorespaces\fi}
\def\implies{\ifmmode\Rightarrow \else
     \unskip${}\Rightarrow{}$\ignorespaces\fi}
\def\iff{\ifmmode\Longleftrightarrow \else
     \unskip${}\Longleftrightarrow{}$\ignorespaces\fi}

\let\:=\colon

\renewcommand{\theequation}{\thesection.\arabic{equation}}
\newtheorem{theorem}{Theorem}[section]
\newtheorem{lemma}[theorem]{Lemma}
\newtheorem{corollary}[theorem]{Corollary}

\newtheorem{remark}[theorem]{Remark}

\newtheorem{Example}[theorem]{Example}

\theoremstyle{definition}

\def\d{{\mathrm d}}

\def\C{{\mathbb C}}
\def\R {{\mathbb R}}

\def\le{\leqslant}
\def\ge{\geqslant}
\def\Omega{\varOmega}
\def\Delta{\varDelta}

\graphicspath{{figs/}}

\title{Numerical analysis of 2D Navier--Stokes equations\\ with nonsmooth initial value in the critical space}
\date{\today}
\author[]{Buyang Li$^*$ \and\, Qiqi Rao$^*$ \and\, Hui Zhang$^{*}$\and\, Zhi Zhou$^{*}$}
\date{}

\thanks{$^*$Department of Applied Mathematics, The Hong Kong Polytechnic University, Hong
Kong. Email address: buyang.li@polyu.edu.hk, qi-qi.rao@connect.polyu.hk, hui1203.zhang@polyu.edu.hk, zhizhou@polyu.edu.hk}

\begin{document}
\maketitle

\textbf{Abstract.}
This paper addresses the numerical solution of the two-dimensional Navier--Stokes (NS) equations with nonsmooth initial data in the $L^2$ space, which is the critical space for the two-dimensional NS equations to be well-posed. In this case, the solutions of the NS equations exhibit certain singularities at $t=0$, e.g., the $H^s$ norm of the solution blows up as $t\rightarrow 0$ when $s>0$. To date, the best convergence result proved in the literature are first-order accuracy in both time and space for the semi-implicit Euler time-stepping scheme and divergence-free finite elements (even high-order finite elements are used), while numerical results demonstrate that second-order convergence in time and space may be achieved. Therefore, there is still a gap between numerical analysis and numerical computation for the NS equations with $L^2$ initial data. The primary challenge to realizing high-order convergence is the insufficient regularity in the solutions due to the rough initial condition and the nonlinearity of the equations. In this work, we propose a fully discrete numerical scheme that utilizes the Taylor--Hood or Stokes-MINI finite element method for spatial discretization and an implicit-explicit Runge--Kutta time-stepping method in conjunction with graded stepsizes.  By employing discrete semigroup techniques, sharp regularity estimates, negative norm estimates and the $L^2$ projection onto the divergence-free Raviart--Thomas element space, we prove that the proposed scheme attains second-order convergence in both space and time. Numerical examples are presented to support the theoretical analysis. In particular, the convergence in space is at most second order even higher-order finite elements are used. This shows the sharpness of the convergence order proved in this article.
\\

\\


%\vspace{-15pt}
\setlength\abovedisplayskip{4pt}
\setlength\belowdisplayskip{4pt}

\section{Introduction}\label{section:intro}
We denote by $\Omega$ a convex polygonal domain in $\R^2$ and consider the Navier--Stokes (NS) equations on $\Omega$ with the no-slip boundary condition up to a given time $T>0$, i.e.,
\begin{equation}\label{system-1}
    \left\{\begin{array}{rclll}
        \partial_t u + u\cdot \nabla u - \Delta u + \nabla p & \!\!\!\!= & \!\!\!\!0 & \mbox{in }& \Omega\times (0,T],\\
	\nabla \cdot u  &\!\!\!\! = & \!\!\! \! 0 & \mbox{in} & \Omega \times (0,T],\\
        u & \!\!\!\! = & \!\!\!\! 0 & \mbox{on} & \partial\Omega\times(0,T],\\
        u & \!\!\!\! = & \!\!\!\! u_0 & \mbox{on} & \Omega\times\{0\} ,
    \end{array}\right.
\end{equation}
where $\partial \Omega$ denotes the boundary of domain $\Omega$.
In particular, we assume that the initial value $u_0$ belongs to $ \dot L^{2}(\Omega) $, which is defined as
\begin{equation}
    \dot L^{2} (\Omega) = \{v \in L^{2}(\Omega)^{2}: \nabla \cdot v = 0\; \text{in} \; \Omega, v \cdot \nu = 0  \;\text{on}\;\partial\Omega\},
\end{equation}
where $ \nu $ denotes the unit outward normal vector on $ \partial \Omega $. It is known that problem \eqref{system-1} has a unique weak solution $ u \in L^{2}(0,T;\dot H_{0}^{1}(\Omega)) \cap H^{1}(0,T;\dot H^{-1}(\Omega)) \hookrightarrow C ([0,T]; \dot L^{2}(\Omega)) $, where $ \dot H_{0}^{1}(\Omega) = \{v \in H_{0}^{1}(\Omega)^{2}: \nabla \cdot v = 0\} $ and $ \dot H^{-1}(\Omega) $ is the dual space of $ \dot H_{0}^{1}(\Omega) $; see \cite{temam1977navie} for a rigorous proof of this result. The uniqueness of solution $p$ can be guaranteed by requiring $p \in L_{0}^{2}(\Omega) := \{v \in L^{2}(\Omega): \int_{\Omega} v \, \d x = 0\} $.

The NS equations are the fundamental partial differential equations describing the motion of incompressible viscous fluids. They are widely used in fluid dynamics to model water and blood flows, air flow around a wing, and ocean currents. As exact solutions are unknown for most practical applications, numerical solutions of the NS equations are of paramount importance. Error estimates can be obtained based on the regularity assumptions of the solution and the initial data.
Optimal error estimates for high-order methods can be proved when the solutions to the Navier-Stokes equations are sufficiently regular, meaning they are sufficiently smooth and adhere to the compatibility conditions.
%; see e.g.,\cite{Baker1976,Baker1982,Burman2007,Girault1979,Notsu2016,Wang2012}.
%More refined analyses are required when these high regularity conditions are not met.
For example, if the initial values are sufficiently smooth, i.e. $u_0\in \dot{H}_0^1(\Omega)\cap H^2(\Omega)^2$ or above, then optimal-order convergence of temporal and spatial discretizations of the NS equations have been proved for various methods in \cite{Badia2007,Bermejo2012,Bosco2021,Heywood1990,Labovsky2009,Ross2013,Rebholz2007,Tone2004,heyinnian2007,Shen1992,Shen2,MR4656794},
%proved the numerical solution to be convergent with optimal order.
where the finite element and spectral Galerkin methods were usually used for spatial discretization, and the time-stepping schemes include varies of the Crank--Nicolson method, Euler method, two-step backward difference formula, projection methods, fractional step methods and so on.
However, the error estimates discussed in the aforementioned articles are not applicable to nonsmooth initial data.

When the initial value $u_0$ belongs to the space $\dot{H}_0^1(\Omega)$, a number of numerical analyses for the Navier-Stokes equations are available. The analysis in \cite{MR1424303} essentially proves almost first-order convergence in time of the Runge--Kutta method for the two-dimensional NS equations when the initial value is in $\dot{H}_0^1(\Omega)$. In \cite{Hill2000}, Hill and S\"uli proved second-order convergence of the semidiscrete finite element method.
%In \cite{Heyinnian2005}, an error bound of $O(\lambda_m^{-1}+\tau)$ was established for the implicit-explicit Euler scheme with spectral Galerkin space discretization, under the CFL condition $\tau\ln(\lambda_m/\lambda_1)\le C$, where $\tau$ is the temporal stepsize and $\lambda_m$ is the maximal eigenvalue of the Stokes operator used in the spectral method.
For the implicit-explicit finite element method, first-order convergence in time and second-order convergence in space were proved under condition $\tau |\log h| \le C$ in \cite{Heyinnian2008}, where $\tau$ and $h$ are the temporal stepsize and spatial mesh size, respectively.
Additionally, the error of semi-discretization in time by the Crank--Nicolson/Adams--Bashforth implicit-explicit scheme with a uniform stepsize was shown to be $O(\tau^{\frac32})$ in \cite{Heyinnian2012}. This convergence rate is sharp with respect to the empirical numerical results. Second-order convergence in time and space was proved for a linearly extrapolated Crank--Nicolson scheme and a two-step backward differentiation formula by utilizing graded stepsizes locally refined towards $t=0$; see \cite{Chu2023,Na2021}.

Discussions concerning the case that $u_0\in \dot{L}^2(\Omega)$ are less prevalent in the literature.
It has been known that $\dot{L}^2(\Omega)$ is  a critical space for the well-posedness of the two-dimensional NS equations \cite{Gallagher2018}.
%, i.e. $u_0\in \dot{L}^2(\Omega)$, which is known to be a critical space for the 2D NS equations, a maximal Sobolev space on which well-posedness of the 2D NS equations was proved in \cite{Gallagher2018}.
The error analysis in this case turns out to be significantly more challenging than for cases with smoother initial data, and the literature offers only a limited number of relevant results. Under the CFL condition, $\tau \leq C\lambda_m^{-1}$, it was shown in \cite{Heyinnian22008} that the implicit-explicit Euler spectral Galerkin method has an error bound of $O(\lambda_m^{-1/2} + \tau^{1/2})$ over a bounded time interval. For the implicit-explicit Euler scheme with finite element  spatial discretization, several stability results were proved in \cite{Hesun2007} without error estimates. In more recent developments, first-order convergence in both time and space was shown in \cite{2DNS2022Li} for high-order divergence-free finite elements. To our knowledge, this represents the most advanced convergence result obtained to date. However, there is still a gap between the numerical analysis and the numerical results, which demonstrate the possibility of achieving second-order convergence in space by using the Taylor--Hood finite elements. Proving second-order convergence of any numerical method for the NS equations remains an open and challenging task. Furthermore, the employed time-stepping scheme is of low order. developing higher-order schemes (with rigorous proof of the convergence rates) presents additional challenges due to limited smoothness of the solution and the nonlinearity of the NS equations.
Recently, the construction and analysis of low-regularity integrators for nonlinear dispersive equations and NS equations based on energy techniques as well as harmonic analysis techniques become an active research area; see \cite{MR4471050,MR4269650,MR4166458,MR4497822,MR4269650}. The analyses in these articles generally require discovering and utilizing certain cancellation structures in the equations. An application of the general framework in \cite{MR4275500} to the NS equations was shown in \cite{MR4471050}. Since this approach does not use the smoothing property of the NS equations (thus the results are independent of the viscosity of fluid), it requires the initial value to be in $\dot H^1_0(\Omega)\cap H^2(\Omega)$ for the numerical solution to have first-order convergence in time.

In this paper, we consider a fully discrete implicit-explicit Runge--Kutta finite element scheme for the NS equations with $L^2$ initial data by utilizing an $L^{2}$ projection $P_{h}^{\text{RT}}$ onto the divergence-free subspace of the Raviart--Thomas element space in the numerical scheme. The linear term is discretized using the Runge--Kutta Lobatto IIIC scheme, while the nonlinear term is handled through an extrapolation approximation. To address the solution's singularity near $t=0$, we employ graded stepsizes that provide enhanced resolution where needed.
We prove the a nearly optimal error estimate. More specifically, let $u_h^{n+1}$ be the numerical solution of the fully discrete scheme at time level $t = t_{n+1}$.
Theorems \ref{thsemi} and \ref{thfull} show that, for arbitrarily small $\varepsilon>0$,
%(Theorems \ref{thsemi} and \ref{thfull})
\begin{equation*}
\| u (t_{n+1})  - u_h^{n+1} \|_{L^{2}(\Omega)} \le C_\varepsilon (h^{2-2\varepsilon} t_{n+1}^{\varepsilon-1} + t_{n+1}^{-2-\varepsilon}\tau_{n+1}^{2}),
\end{equation*}
where $\tau_{n+1}$ and $h$ denote the temporal stepsize of the $(n+1)$th step and spatial mesh size, respectively. A crucial element in our error analysis is the utilization of the $L^{2}$ projection $P_{h}^{\text{RT}}$, which plays a key role in achieving second-order convergence in space and in deriving discrete energy decay, as detailed in Lemma \ref{leelle3.1}. Our analysis also employs the discrete semigroup technique and the estimate of numerical solution in $H^1$ norm (Lemma \ref{H1-estimate-numerical-solution}), as well as some negative norm error estimates (Lemma  \ref{H_negative_estimate_error}). The choice of the Lobatto IIIC scheme is also critical for our analysis due to its distinctive property that the second internal stage coincides with the endpoint of the time interval. This property is extensively used in the stability estimates, e.g., Lemma \ref{H1-estimate-numerical-solution}.
Numerical examples are provided to support the theoretical analysis, which show that the numerical solutions of the NS equations with $L^2$ initial data achieve second-order convergence in both time and space. This is consistent with our theoretical analysis. Moreover, the convergence in space is at most second order even higher-order finite elements are used. This shows the sharpness of the convergence order proved in this paper.

% Furthermore, we present a practical numerical example, which is detailed in several references \cite{ferreira2005,Josserand2007,mao2012,orlandi2007}.
%This example depicts the merging of two-dimensional co-rotating vortices and highlights the effect of vortex interaction. Compared to previous literature, our example here involves an $L^2$ rough initial value. The evolution of the vorticity $ w = \text{curl}\, u $ for the co-rotating vortices is illustrated at various time instances. Our numerical simulations demonstrate the gradual merger of the two co-rotating vortices over time.

The rest of this paper is organized as follows. In Section \ref{sec:fem}, we describe the finite element method for the spatial discretization using Taylor--Hood or Stokes-MINI element, and present the error analysis of the semi-discrete scheme. The fully discrete scheme is developed and analyzed in Section \ref{sec:fully}. Some numerical experiments are shown in Section \ref{sec:numerics} to support and complement our theoretical analysis. Finally, the conclusion is given in Section \ref{conclusion}.

\section{Spatial semi-discretization by finite element method}%
\label{sec:fem}
For $s\ge 0$ and $1\le p \le \infty$, we denote by $W^{s,p}(\Omega)$ the conventional Sobolev spaces of functions defined on $\Omega$, with abbreviation $H^s(\Omega) = W^{s,2}(\Omega)$ and $L^p(\Omega)= W^{0,p}(\Omega)$. % and we denote by $W^{s,p}(\Omega)'$ the dual space of $W^{s,p}(\Omega)$.
For the simplicity of notation, we denote by $\|\cdot\|_{W^{s,p}(\Omega)}$ the norm of the spaces $W^{s,p}(\Omega)$, $W^{s,p}(\Omega)^2$ and $W^{s,p}(\Omega)^{2\times2}$, omitting the dependence on dimension.

Let $ \dot H_{0}^{1}(\Omega) = \{v \in H_{0}^{1}(\Omega)^{2}: \nabla \cdot v = 0\} $ and let $\dot H^s_0(\Omega)=(\dot L^2(\Omega), \dot H^1_0(\Omega))_{[s]}$ be the complex interpolation space between $\dot L^2(\Omega)$ and $ \dot H^1_0(\Omega)$. The dual space of $\dot H^s_0(\Omega)$ is denoted by $\dot H^{-s}(\Omega)$.

% Moreover, we define the divergence-free Sobolev space $\dot H^{1}_0(\Omega) = H_0^1(\Omega)^2 \cap \dot L^{2}(\Omega)$ for $s\in[0,2]$ and denote by $\dot H^{-1} (\Omega)$ the dual space of $ \dot H_0^{1}(\Omega) $.

%The natural function spaces associated to incompressible flow are the divergence-free subspaces of $ L^{2}(\Omega)^2 $ and $ H_{0}^{1}(\Omega)^{2} $, denoted by
%\begin{align*}
%    X = \dot L^{2}(\Omega) \quad \text{and} \quad V = \dot H_0^1 (\Omega),
%\end{align*}
%respectively, as defined in the Introduction section.

\subsection{Weak solution}%
\label{sub:Weak solution}

Let $ P_{X} $ be the $ L^{2} $-orthogonal projection from $L^{2}(\Omega)^2 $ to $ \dot L^{2}(\Omega) $. Then any function $ v \in L^{2}(\Omega)^2$ has a decomposition
\begin{equation}
    v = P_{X} v + \nabla \eta,
\end{equation}
where $ \eta \in H^{1}(\Omega) \cap L_0^2 (\Omega) $ satisfies the following elliptic equation with Neumann boundary condition
\begin{equation*}
    \left\{\begin{array}{rcc}
	\Delta \eta  =  \nabla   \cdot v  & \text{in} & \Omega,\\
	\frac{\partial \eta}{\partial \nu} = v \cdot \nu & \text{on} & \partial \Omega.
    \end{array}\right.
\end{equation*}
\iffalse
{\color{blue}The solution $\eta$ of the Poisson equation has up to $H^2$ regularity in a convex polygon, %(see \cite[Chap. 4]{grisvard2011elliptic}) i.e.,
$$
\|\eta\|_{H^{s+1}(\Omega)}\le C\|v\|_{H^{s}(\Omega)}  \quad  \textrm{for} \ s\in[0,1] .
$$
%Let $W^{s,p}(\Omega)$ denote the complex interpolation space $(L^p(\Omega),W^{1,p}(\Omega))_{[s]}$ for $s\in[0,1]$.
Then the $L^2$ projection $P_X$ extends to a bounded operator on $H^{s}(\Omega)^2$, i.e.,
\begin{equation} \label{ee22}
\|P_Xv\|_{H^{s}(\Omega)}\le C\|v\|_{H^{s}(\Omega)} \quad \textrm{for}\,\,\, v\in H^{s}(\Omega)^2 ,\,\,\, s\in[0,1].
\end{equation}
Since the $L^2$ projection operator $P_X$ is self-adjoint, it follows that (by a duality argument) $P_X$ extends to a bounded operator on $H^{-s}(\Omega)^2$, i.e.,
\begin{equation}
\label{ee32}
\|P_Xv\|_{H^{-s}(\Omega)}\le C\|v\|_{H^{-s}(\Omega)}  \quad  \textrm{for}\,\,\,
v\in H^{-s}(\Omega)^2 ,\,\,\, s\in[0,1] .
\end{equation} }
%where $\dot H^{-s} (\Omega)$ is the dual space of $ \dot H^{s}(\Omega) $.
%
\fi
Since $ \nabla p $ is orthogonal to  $ \dot L^{2}(\Omega) $ for any function $ p \in H^{1}(\Omega) $, it follows that $ P_{X} \nabla p = 0 $.
%\begin{equation*}
% (\nabla p, v)_{\Omega} = (p,\nabla \cdot v)_{\Omega} + (p, v \cdot \nu)_{\partial \Omega}= 0 \quad \text{for all}  ~~v \in \dot L^{2}(\Omega)
%\end{equation*}
%where $ (\cdot,\cdot)_{\Omega} $ denotes the $ L^{2} $ inner product in $ \Omega $. Therefore $ P_{X} \nabla p = 0 $.

We denote by $ A := P_{X} \Delta$ the Stokes operator on $ \dot L^{2}(\Omega) $ with domain $D(A)=\dot H_0^1 (\Omega)\cap H^2(\Omega)^2$, which is a self-adjoint operator and negative-definite. The Stokes operator has an extension to a bounded operator $A:\dot H_0^1 (\Omega) \rightarrow \dot H^{-1}(\Omega)$ defined by
\begin{align}\label{project11-equation}
(Au,v)=-\int_{\Omega}\nabla u\cdot \nabla v\d x\ \ \forall u,v \in \dot H_0^{1}(\Omega).
\end{align}
By applying $P_X$ to the first equation in \eqref{system-1}, we obtain the following abstract parabolic equation in terms of the Stokes operator $A$:
\begin{align}\label{project-equation}
	\partial_t u - A u = - P_{X} (u \cdot \nabla u)\; \;\text{in}\; \;\Omega \times (0,T].
\end{align}
The the weak solution of \eqref{project-equation} can be expressed as
\begin{align}\label{mild-solution}
    u (\cdot, t) = e^{t A} u_{0} - \int_{0}^{t} e^{(t-s)A} P_{X} (u (s)\cdot \nabla u (s)) \d s.
\end{align}
The properties of operator $ A $ are similar to the Laplacian operator $ \Delta $. For example, for any functions $ v, w \in \dot H_{0}^{1}(\Omega) $, $ (A v,w) = -(\nabla v, \nabla w) $.

We recall the following regularity estimate of the solution proved in \cite[Lemma 3.2]{2DNS2022Li}.
\begin{lemma} \label{regularity-lemma-u}
For any given initial value $u_0\in\dot L^{2}(\Omega)$, the exact solution $ u $ of problem \eqref{system-1} satisfy the following regularity result.
\begin{align}\label{regularity-result-u}
    \|\partial_t^mu (\cdot,t)\|_{H^{s}(\Omega)} \le C t^{-\frac{s}{2}-m }\;\text{for}\;\  0 \le s \le 2,\ m=0,1,2,\ldots
\end{align}
%where $ \dot H^{s}(\Omega) = H^s(\Omega)^2 \cap \dot L^{2}(\Omega)$.
\end{lemma}

The exponential operator $ e^{tA}  $ plays a crucial role in the error analysis. By taking Laplace transform and inverse Laplace transform, we have
\begin{align*}
    e^{tA} = \int_{| z | = \sigma} e^{zt} (z-A)^{-1}\d z,
\end{align*}
for some constant $ \sigma > 0 $. Due to the analyticity of $ e^{zt} (z-A)^{-1} $ in the sector $\{z \in \C: | \text{arg}(z) | \le \pi \} $, the straight line $ | z | = \sigma$ in the complex plane can be deformed to a contour $ \Gamma_{\delta,\kappa} $
\begin{align*}
    \Gamma_{\delta,\kappa} = \{\kappa e^{i \theta}: - \delta \le \theta \le \delta \}\cup \{\rho e^{\pm i \delta}: \kappa \le \rho < \infty \}.
\end{align*}
Hence, the operator $ e^{tA} $ has the form
\begin{align}\label{exponential-operator}
    e^{tA} = \int_{\Gamma_{\delta,\kappa}}e^{zt} (z-A)^{-1}\d z.
\end{align}
The stability estimate of the operator $ e^{tA}  $ then follows from the estimate of the resolvent operator $ (z-A)^{-1} $.
\begin{lemma}\label{A-stability}
    The operator $ e^{tA} P_X $ satisfies the following stability estimates.
    \begin{align}
	    \|e^{tA}P_X f\|_{L^{2}(\Omega)} \le & \|f\|_{L^{2}(\Omega)}, \label{A-stability-1}\\
	    \|e^{tA}P_X f\|_{L^{2}(\Omega)} \le & C t^{-\frac{s}{2} }\|f\|_{H^{-s}(\Omega)}\,\,\,\textrm{for}\,\,\, 0 \le s \le 2, \label{A-stability-2}\\
	    \|e^{tA}P_X f\|_{L^{2}(\Omega)} \le & t^{-\frac{1}{r} }\|f\|_{W^{-1,r}(\Omega)}\,\,\, \textrm{for}\,\,\, 1 < r \le 2. \label{A-stability-3}
    \end{align}
\end{lemma}
\begin{proof}
	The first inequality follows from the relation \eqref{exponential-operator}  and the standard resolvent estimate (see \cite[Theorem 3.7.11]{arendt2011vector})
\begin{align}
	\|(z-A)^{-1}P_X f\|_{L^{2}} \le & C | z |^{-1} \| f \|_{L^{2}(\Omega)} \;\text{for}\;  z \in \Gamma_{\delta,\kappa}.\label{resolvent-l2-l2}
	%\|(z-A)^{-1}\|_{H^{1}\to H^{1}} \le & C | z |^{-1}\;\text{for}\; z \in \Gamma_{\delta,\kappa}.\label{resolvent-h1-h1}
\end{align}

To show the second estimate, we let $w = (z-A)^{-1}P_X f $,
then according to \ref{resolvent-l2-l2} we have $ \|w\|_{L^{2}(\Omega)} \le C | z |^{-1} \|f\|_{L^{2}(\Omega)} $.
This together with the elliptic regularity estimate implies
\begin{equation*}
 \|w\|_{H^{2}(\Omega)} \le \| A w \|_{L^2(\Omega)} \le \|z w - P_X f\|_{L^{2}(\Omega)} \le C \|f\|_{L^{2}(\Omega)},
\end{equation*}
and hence
\begin{equation*}
 \|(z-A)^{-1}P_X f\|_{H^{2}(\Omega)} \le C \| f \|_{L^{2}(\Omega)}.
\end{equation*}
Then by means of interpolation there holds
\begin{align}\label{L2-Hk-resolvent}
  \|(z-A)^{-1}P_X f\|_{H^{s}(\Omega)} \le C | z |^{-1+\frac{s}{2} } \| f \|_{L^{2}(\Omega)} \,\,\,\text{for}\,\,\,0 \le s \le 2.
\end{align}
Since the resolvent operator $ (z-A)^{-1} P_X : L^{2} \to  \dot H_0^{s}$ is self-adjoint, we have
\begin{align}\label{L2-H-k-resolvent}
 \|(z-A)^{-1} P_X f \|_{L^{2}} \le C | z |^{-1+\frac{s}{2} } \| f \|_{H^{-s}(\Omega)}\,\,\,\text{for}\,\,\, 0 \le s \le 2.
\end{align}
Then Substituting \eqref{L2-H-k-resolvent} into \eqref{exponential-operator} and evaluating the integral leads to \eqref{A-stability-2}.

To prove \eqref{A-stability-3}, we apply the following embedding estimate in two dimension that
\begin{align}\label{embedding-w-1-p}
	W^{-1,r}(\Omega) \hookrightarrow H^{-2/r}(\Omega)\;\text{for}\;1 < r \le 2.
\end{align}
This, together with \eqref{A-stability-2} with $s = 2/r$, leads to the estimate   \eqref{A-stability-3}.
%
%Choosing $s = 2/r $ in \eqref{A-stability-2}, we have
%\begin{align*}
%	    \|e^{tA} f\|_{L^{2}(\Omega)} \le t^{-\frac{1}{r} } \|f\|_{\dot H^{-\frac{2}{r} }(\Omega)}.
%\end{align*}
%	The last inequality \eqref{A-stability-3} then follows from the imbedding \eqref{embedding-w-1-p}.
\end{proof}

\subsection{Spatial semi-discretization}%
\label{sub:Spatial semi-discretization}
Let $ \mathcal T_{h} $ denote a shape-regular and quasi-uniform triangulation of mesh size $ h $. We define $ \text{RT}^{1}(\mathcal T_{h}) $ to be the $ \text{H}(\text{div},\Omega) $-conforming Raviart-Thomas finite element space:
\begin{align*}
    \text{RT}^{1}(\mathcal T_{h}) := \{w \in \text{H}(\text{div},\Omega): w|_{K} \in  P_{1}(K)^{2} + x P_{1}(K), \;\forall \; K \in \mathcal T_{h}\}.
\end{align*}
Furthermore, we let $ \text{RT}_{0}^{1}(\Omega) $ be a subspace of $ \text{RT}^{1}(\Omega) $ such that
\begin{align*}
	\text{RT}_{0}^{1}(\mathcal T_{h}):=\{v_{h}\in \text{RT}^{1}(\mathcal T_{h}): \nabla \cdot v_{h} = 0 \; \text{in} \; \Omega \; \text{and}\;v_{h}\cdot \nu = 0\;\text{on}\;\partial \Omega\}.
\end{align*}
Define the $ L^{2} $ projection $ P_{h}^{\text{RT}} $ from $ \dot L^{2} $ to $ \text{RT}_{0}^{1} $, that satisfies
\begin{align}
    (v - P_{h}^{\text{RT}} v, \chi_{h}) = 0\;\text{for}\;\text{any}\;v \in \dot L^{2}(\Omega)\;\text{and}\;\chi_{h} \in \text{RT}_{0}^{1}(\mathcal T_{h}).
\end{align}
The projection $ P_{h}^{\text{RT}} $ satisfies the following estimate for $ v \in X $ (cf. \cite[Eq. (3.4)]{li2022convergent}):
\begin{align}\label{projectionRT_error}
    \| P_{h}^{\text{RT}}v - v  \|_{L^{2}( \Omega )} \le C h^{l} \| v \|_{H^{l}( \Omega )},\;l = 1,2.
\end{align}
%The projection operator $ P_h^{\text{RT}} $ is $ L^{2} $ and $ H_0^{1} $ stable. By using the interpolation inequalities, we can derive the $ L^{p} $ stability of $ P_h^{\text{RT}} $, i.e.,
%\begin{align}\label{stability_Lp_RTprojection}
%    \| P_h^{\text{RT}}v \|_{L^{p}(\Omega)} \le C \| v \|_{L^{p}(\Omega)} \quad \text{for} \quad 2 \le p < \infty.
%\end{align}
%This also implies the $H^{1} $ stability of the projection $ P_{h}^{\text{RT}} $. Then the following $ \dot H^{-1} $ stability of $ P_{h}^{\text{RT}} $ can be proved by a duality argument:
%\begin{align}
   %& \| P_{h}^{\text{RT}}v - v \|_{\dot H^{-1}( \Omega )} \le C h^{l+1}\| v \|_{\dot H^{l}( \Omega )},\; l = 1,2,\label{feddd}\\
%   \| P_{h}^{\text{RT}}v\|_{\dot H^{-1}( \Omega )} \le C \| v \|_{\dot H^{-1}( \Omega )}.\; \label{eddd}
%\end{align}

Let the pair $ (V_{h}, Q_{h}) \subset ( H_{0}^{1}(\Omega), L^{2}_{0}(\Omega))$ denote the Taylor--Hood element spaces or Stokes-MINI element space, which have the following approximation properties (see \cite{ABF-1984,Verf1984Error,Raviart1986Finite}):
\begin{align}\label{FE-space-Error}
\inf_{v_h\in V_h} \|v-v_h\|_{H^{s}(\varOmega)}
+ \inf_{q_h\in Q_h} \|q-q_h\|_{H^{s-1}(\varOmega)}
\le Ch^{m-s}\|v\|_{H^{m}(\varOmega)},\quad 0\le s\le 1,\quad 1\le m\le 2.
%\quad\mbox{\color{red}(This needs to be verified)}
\end{align}
Both the Taylor--Hood and Stokes-MINI finite element spaces satisfy the discrete inf-sup condition, i.e., there is a generic constant $\kappa >0 $ such that
\begin{align}
    \sup_{v_{h}\in V_{h},\nabla v_{h}\neq 0} \frac{( q_{h}, \nabla v_{h} )}{\| \nabla v_{h} \|_{L^{2}( \Omega )}} \ge \kappa \| q_{h} \|_{L^{2}( \Omega )}\quad \forall \; q_{h} \in Q_{h}.
\end{align}
%at least second-order convergence in approximating the linear Stokes equations.
%For any functions $ (u,p) \in (\dot H^{1}_{0}(\Omega),L^{2}_{0}(\Omega)) $, we define its Stokes--Ritz projection $ (R_{h}u, R_{h}p) $ by
%\begin{align}
%	(\nabla(u - R_{h}u), \nabla v_{h}) + (p - R_{h}p, v_{h}) & = 0 \quad \forall \;v_{h} \in V_{h},\\
%	(\nabla \cdot (u - R_{h}u), q_{h}) & = 0 \quad \forall \; q_{h} \in  Q_{h}.
%\end{align}

We denote by $ X_{h} := \{v_{h} \in V_{h}: (\nabla \cdot v_{h}, q_{h}) = 0 ~ \forall \, q_{h} \in Q_{h}\} $ the discrete divergence-free subspace of $ V_{h} $,
%\begin{align}
%    X_{h} := \{v_{h} \in V_{h}: (\nabla \cdot v_{h}, q_{h}) = 0 ~ \forall \, q_{h} \in Q_{h}\} ,
%    % = \{v_{h} \in V_{h}: \nabla \cdot v_{h} = 0\}.
%\end{align}
and define the $ L^{2} $ projection $ P_{X_{h}} $ from $ \dot L^{2}(\Omega) $ onto $ X_{h} $ by the following relation:
\begin{align}
    (v - P_{X_{h}}v, w_{h}) = 0\quad \forall \;w_{h} \in X_{h}.
\end{align}
The semi-discrete scheme for the NS equations in \eqref{system-1} reads: Find  $ (u_{h},p_{h}) \in (V_{h},Q_{h})$ such that
\begin{subequations}\label{semi-scheme}
\begin{align}
	(\partial_{t} u_{h},v_{h}) + (P_{h}^{\text{RT}}u_{h}\cdot \nabla u_{h}, v_{h}) + (\nabla u_{h}, \nabla v_{h}) - (p_{h}, \nabla  \cdot v_{h}) & = 0 \quad \forall \; v_{h} \in V_{h},\label{semi-scheme-a}\\
	(\nabla \cdot u_{h}, q_{h}) & = 0 \quad  \forall \; q_{h} \in Q_{h}.\label{semi-scheme-b}
\end{align}
\end{subequations}
%There exists a linear projection operator $\Pi_h:H_0^1(\Omega)\rightarrow V_h$ such that
%\begin{enumerate}
%  \item [(i)] $\nabla\cdot \Pi_hv=P_{Q_h}\nabla\cdot v$ for $v\in H_0^1(\Omega)$, where $P_{Q_h}:L_0^2(\Omega)\rightarrow Q_h$ denotes the $L^2-$ orthogonal projection.
%  \item [(ii)] The following approximation property holds for $v \in H_0^1(\Omega)\cap H^m(\Omega)$:
%  \begin{align}
%\|v-\Pi_hv\|_{H^s(\Omega)}\le Ch^{m-s}\|v\|_{H^m(\Omega)},\quad 0\le s \le1,\quad 1\le m \le 2.\label{conv1}
%\end{align}
%\item [(iii)] Approximation of $X_h$ to $\dot{H}^1_0(\Omega):$ for $m = 1, 2$, there holds for $v_h=\Pi_hv$
% \begin{align}
%\underset{v_h\in X_h}{\inf}(\|v-v_h\|_{L^2(\Omega)}+h \|v-v_h\|_{H^1(\Omega)}  )\le Ch^{m}\|v\|_{H^m(\Omega)},\quad \forall v\in \dot{H}^1_0(\Omega)\cap H^m(\Omega).\label{conv2}
%\end{align}
%\end{enumerate}
%
Let $ A_{h}: X_h \rightarrow X_h $ be the discrete Stokes operator defined by
\begin{align*}
    (A_{h} v_{h},w_{h}) =  - (\nabla v_{h}, \nabla w_{h})\quad \forall \;v_{h},w_{h} \in X_{h}.
\end{align*}
%Note that all finite element functions satisfy the following inverse inequality %\cite{Brenner2008},
%\begin{align}\label{inverse}
%\|u_h\|_{W^{s_1,p_1}(\Omega)}\le Ch^{s_2-s_1+\frac{d}{p_1}-\frac{d}{p_2}}\|u_h\|_{W^{s_2,p_2}(\Omega)}\ \ \textrm{for}\ 0\le s_2\le s_1\le 1,\ 1\le p_2\le p_1\le \infty,
%\end{align} where the constant $C$ depending on the finite element space of $u_h$ (but independent of $h$).
%
Then, by applying projection operator $ P_{X_{h}} $ to \eqref{semi-scheme}, the semi-discrete scheme in \eqref{semi-scheme} can be rewritten as
\begin{align}\label{FEM-num-solution1}
\partial_t u_h(\cdot,t) - A_h u_h(\cdot,t) = - P_{X_h} (P_{h}^{\text{RT}} u_{h}(s) \cdot \nabla u_{h}(s)),
\end{align}
with initial value $ u_{h}(\cdot,0) = u_{h}^{0} := P_{X_{h}}u_{0} $.
By using Duhamel's formula, the solution to the semidiscrete problem \eqref{FEM-num-solution1} can be written as
\begin{align}\label{FEM-num-solution}
    u_{h}(\cdot,t) = e^{t A_{h}} u_{h}^{0} - \int_{0}^{t} e^{(t-s)A_{h}}P_{X_{h}}(P_{h}^{\text{RT}} u_{h}(s)\cdot \nabla u_{h}(s)) \d s.
\end{align}

\begin{remark}\upshape
If $\varphi_h\in X_h$ and $\varphi\in \dot H^1_0(\Omega)^2\cap H^2(\Omega)^2$ satisfies the following relation:
\begin{align}\label{Afi=Ahfih}
	    A \varphi = A_h \varphi_h.
	\end{align}
then there exist $q\in L^2(\Omega)$ and $q_h\in Q_h$ such that $(\varphi_h,q_h)$ is the Stokes-Ritz projection of $(\varphi,q)$, i.e., Ritz projection associated to the linear Stokes equations. This can be shown as follows: Let $q\in L^2_0(\Omega)$ and $q_h\in Q_h$ be the unique functions (determined via the continuous and discrete inf-sup conditions) such that
\begin{align*}
\begin{aligned}
-(A\varphi,v) &= (\nabla \varphi,\nabla v) - (q,\nabla\cdot v)
&& \forall\, v\in H^1_0(\Omega)^2 , \\
-(A_h\varphi_h,v_h) &= (\nabla \varphi_h,\nabla v_h) - (q_h,\nabla\cdot v_h)
&& \forall\, v_h\in V_h .
\end{aligned}
\end{align*}
Then testing equation $-A \varphi = -A_h\varphi_h$ by $v_h\in V_h$ yields
$$
(\nabla \varphi,\nabla v_h) - (q,\nabla\cdot v_h) = (\nabla \varphi_h,\nabla v_h) - (q_h,\nabla\cdot v_h)
\quad \forall\, v_h\in V_h .
$$
This shows that $(\varphi_h,q_h)$ is the Ritz projection of $(\varphi,q)$ associated to the linear Stokes equations. Moreover, via integration by parts we derive $\nabla q=\Delta\varphi-A\varphi$, which implies that
$$
\|q\|_{H^{l-1}(\Omega)} \le C\|\varphi\|_{H^{l}(\Omega)} \quad\mbox{for}\,\,\, l=1, 2.
$$
Therefore, the standard $L^2$ and $H^1$ error estimates for the Stokes-Ritz projection (see \cite{Raviart1986Finite}) imply the following result:
\begin{align}
	\| \varphi_h - \varphi \|_{L^{2}(\Omega)} + h \| \varphi_h - \varphi \|_{H^{1}(\Omega)}
	&\le C h^l (\| \varphi \|_{H^{l}(\Omega)} + \| q \|_{H^{l-1}(\Omega)}) \notag \\
	&\le C h^l \| \varphi \|_{H^{l}(\Omega)} \quad \text{for}\quad l = 1,2. \label{7ff}
\end{align}
Let $v\in \dot H^1_0(\Omega)^2\cap H^2(\Omega)$ be the solution of the PDE problem $A v = \varphi$, and let $v_h\in X_h$ be the Stokes-Ritz projection of $v$ defined by $A_hv_h=Av =\varphi$. Then testing equation $-A\varphi=-A_h\varphi_h$ yields
\begin{align*}
\|\varphi\|_{L^2(\Omega)}^2
= (-A_h\varphi_h, v-v_h) - (\varphi_h, A_hv_h)
&\le C\|A_h\varphi_h\|_{L^2(\Omega)} \| v-v_h \|_{L^2(\Omega)}
    + C\|\varphi_h\|_{L^2(\Omega)} \| A_hv_h \|_{L^2(\Omega)} \\
&\le \|A_h\varphi_h\|_{L^2(\Omega)} Ch^2 \| v \|_{H^2(\Omega)}
    + C\|\varphi_h\|_{L^2(\Omega)} \| \varphi \|_{L^2(\Omega)} \\
&\le C\|\varphi_h\|_{L^2(\Omega)} \| \varphi \|_{L^2(\Omega)}
    + C\|\varphi_h\|_{L^2(\Omega)} \| \varphi \|_{L^2(\Omega)} ,
\end{align*}
which implies the following $L^2$ stability result:
\begin{align} \label{L2-varphi-varphi_h}
\|\varphi\|_{L^2(\Omega)} \le C\|\varphi_h\|_{L^2(\Omega)} .
\end{align}
By testing equation $-A\varphi=-A_h\varphi_h$ with $\varphi$ we also obtain the following $H^1$ stability result:
\begin{align} \label{H1-varphi-varphi_h}
	\|\varphi\|_{H^1(\Omega)} \le C\|\varphi_h\|_{H^1(\Omega)}  .
\end{align}

\end{remark}

The $L^p$ stability of the projection operator $P_h^{\text{RT}}$ plays a pivotal role in the ensuing error analysis. The following lemma presents a fundamental result crucial for our investigations:
%By using inverse inequality, $ L^{2} $ stability of $ P_h^{\text{RT}} $ and the error estimate \eqref{projectionRT_error}, we can derive the $ L^{p} $ estimate of the projection operator $ P_h^{\text{RT}} $ for $ 2 \le p \le \infty $ and $ \varphi_h \in X_h $
\begin{lemma}
    Let $ \varphi_{h} \in X_h $, and $ 2 \le p \le \infty $, the following inequality holds:
		\begin{align}\label{lp_estimate_RTFEM}
			\| P_{h}^{\text{RT}}\varphi_h \|_{L^{p}(\Omega)}
			\le \| \varphi_h \|_{L^{p}(\Omega)} + C \| \varphi_h \|_{L^{2}(\Omega)}^{\frac{2}{p}} \| \varphi_h \|_{H^{1}(\Omega)}^{1-\frac{2}{p}}.
		\end{align}
\end{lemma}
\begin{proof}
For a function $\varphi_h\in X_h$, we let $ \varphi $ be the solution to the elliptic PDE problem in \eqref{Afi=Ahfih}. Thus $\varphi_h$ is the Stokes-Ritz projection of $\varphi$, satisfying the estimates in \eqref{7ff}--\eqref{H1-varphi-varphi_h}. Next, we proceed to estimate the $L^{p}$ norm of $P_{h}^{\text{RT}}\varphi_h$ as follows:
\begin{align*}
			\| P_{h}^{\text{RT}}\varphi_h \|_{L^{p}(\Omega)} \le & \| \varphi_h \|_{L^{p}(\Omega)} + \| P_h^{\text{RT}}\varphi_h - \varphi_h \|_{L^{p}(\Omega)} \\
			\le & \| \varphi_h \|_{L^{p}(\Omega)} + C h^{\frac{2}{p} - 1} \| P_{h}^{\text{RT}}\varphi_h - \varphi_h \|_{L^{2}(\Omega)}  \\
			\le & \| \varphi_h \|_{L^{p}(\Omega)}
			  + C h^{\frac{2}{p} - 1} \big(\| P_{h}^{\text{RT}}(\varphi_h - \varphi) \|_{L^{2}(\Omega)}+\| P_{h}^{\text{RT}}\varphi - \varphi \|_{L^{2}(\Omega)} + \| \varphi - \varphi_h \|_{L^{2}(\Omega)}\big).
\end{align*}
By incorporating the error estimates \eqref{projectionRT_error}, \eqref{7ff}, the stability estimate in  \eqref{H1-varphi-varphi_h}, and the $L^{2}$ stability of $P_{h}^{\text{RT}}$, we obtain
\begin{align*}
\| P_{h}^{\text{RT}}\varphi_h \|_{L^{p}(\Omega)}
\le & \| \varphi_h \|_{L^{p}(\Omega)} + C h^{\frac{2}{p} } \| \varphi_h \|_{H^1(\Omega)}
\le  \| \varphi_h \|_{L^{p}(\Omega)} + C \| \varphi_h \|_{L^{2}(\Omega)}^{\frac{2}{p}} \| \varphi_h \|_{H^{1}(\Omega)}^{1-\frac{2}{p}} ,
\end{align*}
where we have used the inverse inequality of finite element functions. This proves the result in \eqref{lp_estimate_RTFEM}.
\end{proof}
		
When $p < \infty$, leveraging the interpolation inequality allows us to eliminate the first term on the right-hand side of \eqref{lp_estimate_RTFEM}. However, in the case when $p = \infty$, we encounter the task of estimating the $L^{\infty}$ norm of a finite element function in $X_h$. To address this, we present the following lemma.

%	When $ p < \infty $, by utilizing the interpolation inequality, the first term on the right-hand side of \eqref{lp_estimate_RTFEM} can be removed. When $ p = \infty $, we need to estimate the $ L^{\infty} $ norm of a finite element function in $ X_h $, and we have the following lemma.

\begin{lemma}\label{discreteinter}
The following inequality holds:
\begin{equation}\label{7eeee}
\| \varphi_h\|_{L^\infty(\Omega)}\le C
\| \varphi_h\|_{L^2(\Omega)}^{\frac{1}{2}}
\|A_h\varphi_h\|_{L^2(\Omega)}^{\frac{1}{2}},\quad \forall \varphi_h\in X_h.
\end{equation}
\end{lemma}
\begin{proof}
Let $\varphi$ be the solution of equation \eqref{Afi=Ahfih}.
% of the following elliptic PDE problem:
%\begin{align}
%A\phi=A_h\varphi_h,\label{7fe}
%\end{align}
Then the following standard regularity result hold:
\begin{align}
&\|\varphi\|_{H^2(\Omega)}\le C\|A_h\varphi_h\|_{L^2(\Omega)}.\label{8ff}
\end{align}
Therefore, the Sobolev interpolation inequality in \cite[Theorem 5.9]{Adams2003} implies that
\begin{align}
\|\varphi\|_{L^\infty(\Omega)}
&\le C\|\varphi\|_{L^2(\Omega)}^{\frac{1}{2}}  \|\varphi\|_{H^2(\Omega)}^{\frac{1}{2}}
\le C\|\varphi_h\|_{L^2(\Omega)}^{\frac{1}{2}}  \|A_h\varphi_h\|_{L^2(\Omega)}^{\frac{1}{2}} .
\label{9gg}
\end{align}
%Since $\varphi_h$ is the Stokes-Ritz projection of $\phi$, the standard $H^1$ and $L^2$ error estimates for the Stokes-Ritz projection are given by
%\begin{align}
%&\|\nabla(\varphi_h-\phi)\|_{L^2(\Omega)}\le
%Ch\|\phi\|_{H^2(\Omega)},\ \ \|\varphi_h-\phi\|_{L^2(\Omega)}\le
%Ch^2\|\phi\|_{H^2(\Omega)}.\label{7ff}
%\end{align}
Using the inverse inequality and the error estimate \eqref{7ff}, we have
\begin{align}
&\|I_h \varphi-\varphi_h\|_{L^\infty(\Omega)}
\le
Ch^{-1}\|I_h \varphi-\varphi_h\|_{L^2(\Omega)}\le Ch\|\varphi\|_{H^2(\Omega)}.
\label{1aa}
\end{align}
  Using this result and the triangle inequality, we can bound  $\|\varphi_h\|_{L^\infty(\Omega)}$ by
\begin{align}
\|\varphi_h\|_{L^\infty(\Omega)}
\le& \|I_h \varphi\|_{L^\infty(\Omega)}+ \|I_h \varphi-\varphi_h\|_{L^\infty(\Omega)} \notag\\
\le& C\|\varphi\|_{L^\infty(\Omega)}+ Ch\|\varphi\|_{H^2(\Omega)}
&&\mbox{($L^\infty$-stability of $I_h$)} \notag\\
\le &C\|\varphi_h\|_{L^2(\Omega)}^{\frac{1}{2}}  \|A_h\varphi_h\|_{L^2(\Omega)}^{\frac{1}{2}} +Ch \|A_h\varphi_h\|_{L^2(\Omega)}
&&\mbox{(here \eqref{8ff} and \eqref{9gg} are used)}\notag\\
\le& C\|\varphi_h\|_{L^2(\Omega)}^{\frac{1}{2}}  \|A_h\varphi_h\|_{L^2(\Omega)}^{\frac{1}{2}}
&&\mbox{(inverse inequality)} .
\label{2aa}
\end{align}
This proves the result of Lemma \ref{discreteinter}.
\end{proof}

The discrete operator $ A_{h} $ has similar property to $ A $, we can obtain the regularity result for the semi-discrete numerical solution $ u_{h} $ in the following lemma. The proof is similar to that of Lemma \ref{regularity-lemma-u}.
\begin{lemma}\label{regularity-numerical-u}
   The semi-discrete solution $ u_{h}$ to problem \eqref{FEM-num-solution1}
   is a function of $ L^{2}(0,T;\dot H_{0}^{1}(\Omega))$ and satisfies
    \begin{align}
        \|\partial_t^mu_{h}(\cdot,t)\|_{H^{s}(\Omega)} \le C t^{-\frac{s}{2}-m }\;\text{for}\ \ 0 \le s \le 1,\ m=0,1,2,\ldots
    \end{align}
\end{lemma}

 According to \cite[Eq. (3.5)]{2DNS2022Li}, the projection operator $ P_{X_h} $ is $ H_0^{1}$ stable. By using a duality argument, we can derive that $ P_{X_h} $ is $ H^{-1} $ stable.
 The following corollary present some \text{a priori} estimates for the  semi-discrete solution $u_h$ in negative norms.
\begin{corollary} \label{11regularity-numerical-u}
This is the extension of Lemma \ref{regularity-numerical-u}. The semi-discrete numerical solution $ u_{h} $ is a function of $ L^{2}(0,T;\dot H_{0}^{1}(\Omega)) $ satisfying that
    \begin{align}
        \|\partial_t^mu_{h}(\cdot,t)\|_{H^{-s}(\Omega)} \le C t^{-m+\frac{s}{2}}\;\text{for}\ \ 0\le s \le 1,\ \ m=1,2,\ldots \label{uuh7}
    \end{align}
\end{corollary}
\begin{proof}
	By the equation \eqref{FEM-num-solution1}, the $ H^{-1} $ stability of $ P_{X_h} $, and the inequality \eqref{lp_estimate_RTFEM}, we have
	%{\color{red}(need the $ H^{-1} $ stability of $ P_{X_h} $, $ L^{2} $ stability of $ P_h^{RT} $, and $ L^{4} $ stability of $ P_h^{RT} $)}
\begin{align*}
\| \partial_{t}u_{h}(\cdot,t) \|_{H^{-1}(\Omega)} \le & C\| u_{h}(\cdot,t) \|_{H^{1}(\Omega)} + C\| [P_{h}^{\text{RT}}u_{h}\cdot \nabla u_{h} ](\cdot,t)\|_{H^{-1}(\Omega)} \\
 \le & C \| u_{h}(\cdot,t) \|_{H^{1}(\Omega)}+ C \| [P_{h}^{\text{RT}}u_{h}\otimes u_{h}](\cdot,t) \|_{L^{2}(\Omega)}\\
 \le & C \| u_h (\cdot,t) \|_{H^{1}(\Omega)} + C \| P_{h}^{\text{RT}}u_h (\cdot,t) \|_{L^{4}(\Omega)} \| u_h (\cdot,t) \|_{L^{4}(\Omega)}\\
 \le & C \| u_{h}(\cdot,t) \|_{H^{1}(\Omega)} + C \| u_{h}(\cdot,t) \|_{H^{1}(\Omega)}\| u_{h}(\cdot,t) \|_{L^{2}(\Omega)} \le C t^{-\frac{1}{2}}.
\end{align*}
%Based on Lemma \ref{regularity-numerical-u}, we have
%\begin{align}\label{re1}
%    \|\partial_t^m u_h (\cdot,t)\|_{\dot H^{1}(\Omega)} \le C t^{-\frac{1}{2}-m }\;\text{for}\;\ m=0,1,2,\ldots
%\end{align}
%Testing \eqref{project-equation} with $v\in \dot H_0^{1}(\Omega)$, we have
%\begin{align}\label{re2}
%    (\partial_t u_h,v_h) - (A_h u_h,v_h) = - P_{X_h} (P_{h}^{\text{RT}} u_{h} \cdot \nabla u_h,v_h).
%\end{align}
%Using the H\"older's inequality, one has
%\begin{align}\label{re3}
%\left|- P_{X_h} (P_{h}^{\text{RT}} u_{h} \cdot \nabla u_h,v_h)\right|
%=&\left|(P_{h}^{\text{RT}} u_{h}\cdot \nabla v_h,  u_h)  \right|\notag \\ \le&C \| u_{h}\|^2_{L^4(\Omega)}\|v_h\|_{H^1(\Omega)}\notag \\
% \le& C \| u_{h}\|_{H^1(\Omega)}\|v_h\|_{H^1(\Omega)}.
%\end{align}
%Therefore,
%\begin{align}\label{re4}
%\left|(\partial_t u_h,v_h)\right|\le C \| u_{h}\|_{H^1(\Omega)}\|v_h\|_{H^1(\Omega)}.
%\end{align}
%Furthermore, it yields
%\begin{align*}%\label{re5}
%\|\partial_t u_h\|_{H^{-1}(\Omega)}\le C \| u_{h}\|_{H^1(\Omega)}\le C t^{\frac{-1}{2}}.
%\end{align*}
We denote $u_h^{(m-1)}=\partial_t^{m-1} u_h$, $m\ge2$, and differentiate \eqref{FEM-num-solution1} $m-1$ times, we obtain
\begin{align*}%\label{re5}
	\partial_t u_h^{(m-1)}- A_h u_h^{(m-1)} = - P_{X_h} \sum_{j=0}^{m-1}\binom{m-1}{j}(P_{h}^{\text{RT}} u_{h}^{(j)} \cdot \nabla u_{h}^{(m-1-j)}).
\end{align*}
Similar to the above process, we derive that
\begin{align*}%\label{re5}
\|\partial_t u^{(m-1)}_h\|_{H^{-1}(\Omega)}\le C t^{-m+\frac{1}{2}}.
\end{align*}
Using the interpolation inequality, \eqref{uuh7} is verified.
%We complete this proof.
\end{proof} 

The next lemma provides error bounds between $e^{tA} P_{X}$ and $e^{tA_{h}} P_{X_h}$.
\begin{lemma}\cite[Lemma 4.5]{li2022optimal}\label{err-operator}
	The error between exact operator $ e^{tA} P_{X} $ and $ e^{tA_{h}} P_{X_h} $ is presented as follows %{\color{red}(need clear proof)}
    \begin{align}
	    \|e^{tA} P_{X} - e^{tA_{h}} P_{X_h}\|_{L^{2}\to  L^{2}} & \le C t^{-1}h^{2},\label{err-operator-1}\\
	    \|e^{tA} P_{X} - e^{tA_{h}} P_{X_h}\|_{ L^{2}\to  L^{2}} & \le C t^{-\frac{1}{2} }h,\label{err-operator-2}\\
	    %\|e^{tA} P_{X} - e^{tA_{h}} P_{X_h}\|_{ H^{-1}\to  L^{2}} & \le C t^{-1 } h,\label{err-operator-3}\\
     \|e^{tA} P_{X} - e^{tA_{h}} P_{X_h}\|_{ H^{-1}\to  L^{2}} &  \le C t^{-\frac{1}{2}} .\label{err-operator-4}
    \end{align}
\end{lemma} 

Next, we present an optimal error estimate for the semi-discrete scheme \eqref{FEM-num-solution1}. Here we only consider the short-time error estimate, i.e., $T\le T_0$ with $T_0$ sufficiently small. This case is more tricky since the  $H^2$ norm of the solution exhibits singularity near $t=0$. For large time estimate with $t >T_0$, the standard argument for  the case that $u_0\in \dot{H}_0^1(\Omega)\cap [H^2(\Omega)]^2$ works directly.
%According to \cite[Lemma 3.5]{2DNS2022Li}, for any small $\sigma>0$, there exists $T_\sigma>0$ such that
%\begin{equation*}
%\|u\|_{L^{2}(0,T;\dot H_{0}^{1}(\Omega))} + \|u_{h}\|_{L^{2}(0,T;\dot H_{0}^{1}(\Omega))}\le \sigma \quad  \forall t\in(0,T_\sigma].
%\end{equation*}
%Then we assume that $T$ is small enough, for a given $\varepsilon > 0$

\begin{theorem}
\label{thsemi}
Suppose $ u $ is the mild solution of \eqref{system-1} defined by \eqref{mild-solution}, $ u_{h} $ is the numerical solution defined by \eqref{FEM-num-solution}.
Then the error $ e (t) := u (t) - u_{h} (t) $ satisfies
    \begin{align}\label{error-h-ep-log}
	\|e (t)\|_{L^{2}(\Omega)} \le C t^{-1+\varepsilon} h^{2-2\varepsilon}\quad \forall ~~ t\in(0,T]
    \end{align}
for arbitrarily small $\varepsilon>0$ and sufficiently small $T$. %Here the constant $ C $ is independent of $ t $ and  $ h $.
\end{theorem}
\begin{proof}
By using the equations \eqref{mild-solution} and \eqref{FEM-num-solution}, the error $ \| e (t) \|_{L^{2}(\Omega)} $ can be decomposed as:
		\begin{align*}
			\| e (t) \|_{L^{2}(\Omega)} \le & \Big\| (e^{t A}P_X - e^{t A_h}P_{X_h} )u_0 \Big\|_{L^{2}(\Omega)}\\
											& + \Big\| \int_{0}^t e^{(t-s)A} P_X \Big[u (s)\cdot \nabla u (s) - P_h^{\text{RT}} u_h (s) \cdot \nabla u_h (s)\Big]\d s \Big\|_{L^{2}(\Omega)}\\
											& + \Big\| \int_{0}^{t} \Big[e^{(t-s)A}P_{X} - e^{(t-s)A_h}P_{X_h}\Big] (P_h^{\text{RT}} u_h (s)\cdot \nabla u_h (s))\d s \Big\|_{L^{2}(\Omega)} \\
			=: & \mathcal{E}_1 (t) + \mathcal{E}_2 (t) + \mathcal{E}_3 (t).
		\end{align*}
		The error $ \mathcal{E}_1 (t) $ follows from \eqref{err-operator-1} and the $ L^{2} $ stability of $ e^{tA} $ and $ e^{t A_h} $ such that
		\begin{align}\label{estimate_E1_t}
		    \mathcal{E}_1 (t) \le C t^{-1 + \varepsilon} h^{2-2\varepsilon} \| u_0 \|_{L^{2}(\Omega)}.
		\end{align}

For the estimate of $ \mathcal{E}_2 (t) $, since $ u $ and $ P_h^{\text{RT}}u_h $ are both divergence free, by using \eqref{A-stability-3} and choosing $ r = 1/(1-\frac{\varepsilon}{2}) $, we have
		\begin{align*}
			\mathcal{E}_2 (t) \le & C \int_{0}^{t} (t-s)^{-1 + \frac{\varepsilon}{2}} \| u (s) \otimes u (s) - P_h^{\text{RT}}u_h (s)\otimes u_h (s) \|_{L^{1/(1-\varepsilon/2)} (\Omega)} \d s \\
			\le & C \int_{0}^{t} (t-s)^{-1+\frac{\varepsilon}{2}}\| P_h^{\text{RT}}e (s) \otimes u (s) + P_h^{\text{RT}} u_h (s) \otimes e (s) \|_{L^{1/(1-\varepsilon/2)} (\Omega)}\d s \\
				& +C \int_0^t (t-s)^{-1+\frac{\varepsilon}{2}} \| (u (s) - P_h^{\text{RT}} u (s))\otimes u (s) \|_{L^{1/(1-\varepsilon/2)} (\Omega)} \d s \\
			\le & C \int_{0}^{t} (t-s)^{-1+\frac{\varepsilon}{2}}\Big( \| P_h^{\text{RT}} e (t) \|_{L^{2}(\Omega)} \| u (s) \|_{L^{\frac{2}{1-\varepsilon}}(\Omega)} + \| e (s) \|_{L^{2}(\Omega)} \| P_h^{\text{RT}} u_h (s) \|_{L^{\frac{2}{1-\varepsilon}}(\Omega)}\Big)\d s\\
				& + C \int_{0}^t (t-s)^{-1+\frac{\varepsilon}{2}}\| u (s) - P_h^{\text{RT}}u (s) \|_{L^{2}(\Omega)} \| u (s) \|_{L^{\frac{2}{1-\varepsilon}}(\Omega)} \d s
		\end{align*}
		 By using Lemma \ref{regularity-lemma-u}, Lemma \ref{regularity-numerical-u}, the error estimate \eqref{projectionRT_error}, the $ L^{2}(\Omega) $ stability of $ P_h^{\text{RT}} $, the estimate \eqref{lp_estimate_RTFEM} for $ p = 2/(1-\varepsilon) $, and the interpolation inequality, we have
		\begin{align}
			\mathcal{E}_2 (t) \le & C \int_{0}^{t} (t-s)^{-1+\frac{\varepsilon}{2}}s^{-\frac{\varepsilon}{4}} \Big(\| u (s) \|_{H^{1}(\Omega)}^\frac{\varepsilon}{2} + \| u_h (s) \|_{H^{1}(\Omega)}^{\frac{\varepsilon}{2}} \Big)\| e (t) \|_{L^{2}(\Omega)}\d s \notag\\
			& + C h^{2-2\varepsilon} \int_{0}^{t} (t-s)^{-1+\frac{\varepsilon}{2}} \| u (s) \|_{H^{2}(\Omega)}^{1-\varepsilon} \| u (s) \|_{H^{1}(\Omega)}^{\varepsilon} \d s \notag\\
			\le & C t^{-1+ \varepsilon} h^{2-2\varepsilon} + C \int_{0}^{t} (t-s)^{-1+\frac{\varepsilon}{2}}s^{-\frac{\varepsilon}{4}} \Big(\| u (s) \|_{H^{1}(\Omega)}^\frac{\varepsilon}{2} + \| u_h (s) \|_{H^{1}(\Omega)}^{\frac{\varepsilon}{2}} \Big)\| e (t) \|_{L^{2}(\Omega)}\d s.\label{estimate_E2_t}
		\end{align}

		For the estimate of $ \mathcal{E}_3 (t) $, by using Lemma \ref{err-operator}, we have
		\begin{align}
			\mathcal{E}_3 (t) \le & C h^{2-2\varepsilon} \int_{0}^t (t-s)^{-1+\frac{\varepsilon}{2}} \| P_h^{\text{RT}} u_h (s) \cdot \nabla u_h (s) \|_{L^{2}(\Omega)}^{1-\varepsilon} \| P_h^{\text{RT}} u_h (s) \cdot \nabla u_h (s) \|_{H^{-1}(\Omega)}^{\varepsilon} \d s \notag \\
			\le & C h^{2-2\varepsilon} \int_{0}^{t} (t-s)^{-1+\frac{\varepsilon}{2}} \| P_{h}^{\text{RT}}u_h (s) \|_{L^{\infty}(\Omega)}^{1-\varepsilon} \| \nabla u_h (s) \|_{L^{2}(\Omega)}^{1-\varepsilon} \| P_{h}^{\text{RT}} u_h (s) \otimes u_h (s) \|_{L^{2}(\Omega)}^{\varepsilon}\d s \notag\\
			\le & C h^{2-2\varepsilon} \int_{0}^t (t-s)^{-1+\frac{\varepsilon}{2}} \| P_h^{\text{RT}} u_h (s) \|_{L^{\infty}(\Omega)} \| \nabla u_h (s) \|_{L^{2}(\Omega)}^{1-\varepsilon} \| u_h (s) \|_{L^{2}(\Omega)}^{\varepsilon} \d s \notag \\
			\le & C h^{2-2\varepsilon} \int_{0}^{t} (t-s)^{-1+\frac{\varepsilon}{2}} \Big(\| u_h (s) \|_{L^{\infty}(\Omega)} + \| u_h (s) \|_{H^{1}(\Omega)}\Big) \| \nabla u_h (s) \|_{L^{2}(\Omega)}^{1-\varepsilon}\d s, \label{estimate_E3_t}
		\end{align}
		where the last inequality follows from \eqref{lp_estimate_RTFEM}. By using Lemma \ref{discreteinter}, we have
		\begin{align}\label{linf_estimate_uhs}
		    \| u_h (s) \|_{L^{\infty}(\Omega)} \le C \| u_h (s) \|_{L^{2}(\Omega)}^{\frac{1}{2}} \| A_h u_h (s) \|_{L^{2}(\Omega)}^{\frac{1}{2}}.
		\end{align}
		From the equation \eqref{FEM-num-solution1}, we can estimate $ \| A_h u_h (s) \|_{L^{2}(\Omega)} $ as follows by using the $ L^{2} $ stability of $ P_{X_h} $, \eqref{lp_estimate_RTFEM} and Lemma \ref{regularity-numerical-u}
		\begin{align}
			\| A_h u_h (s) \|_{L^{2}(\Omega)} \le & \| \partial_t u_h (s) \|_{L^{2}(\Omega)} + \| P_{h}^{\text{RT}} u_h (s) \cdot \nabla u_h (s) \|_{L^{2}(\Omega)} \notag\\
			\le & C s^{-1} + C \| P_h^{\text{RT}} u_h (s) \|_{L^{\infty}(\Omega)} \| \nabla u_h (s) \|_{L^{2}(\Omega)} \notag\\
			\le & C s^{-1} + C \Big( \| u_h (s) \|_{L^{\infty} (\Omega)} + \| u_h \|_{H^{1}(\Omega)} \Big) \| \nabla u_h (s) \|_{L^{2}(\Omega)} \notag\\
			\le & C s^{-1} + C s^{-\frac{1}{2}} \| u_h (s) \|_{L^{\infty} (\Omega)}. \label{l2estimate_Ahuh}
		\end{align}
		Substituting \eqref{l2estimate_Ahuh} into \eqref{linf_estimate_uhs} and using Young's inequality, we obtain
		\begin{align}\label{linf_estimate_shalf}
		    \| u_h (s) \|_{L^{\infty} (\Omega)} \le C s^{-\frac{1}{2}}.
		\end{align}
		Substituting \eqref{linf_estimate_shalf} into \eqref{estimate_E3_t} and using Lemma \ref{regularity-numerical-u}, we have
		\begin{align}\label{estimate_E3_finial}
		    \mathcal{E}_3 (t) \le C h^{2-2\varepsilon} \int_{0}^t (t-s)^{-1+\frac{\varepsilon}{2}} s^{-1+\frac{\varepsilon}{2}}\d s \le C t^{-1+\varepsilon} h^{2-2\varepsilon}.
		\end{align}
		Combining the estimates \eqref{estimate_E1_t}, \eqref{estimate_E2_t} and \eqref{estimate_E3_finial}, we obtain the estimate for $ e (t) $
	\begin{align*}
	   \|e (t)\|_{L^{2}(\Omega)}
	   \le & C \int_{0}^{t}(t-s)^{-1+\frac\varepsilon2} s^{-\frac\varepsilon4}\Big(\| u_h ( s ) \|_{H^{1}( \Omega )}^{\frac\varepsilon2}+\| u ( s ) \|_{H^{1}( \Omega )}^{\frac\varepsilon2}\Big)\|e (s)\|_{L^{2}(\Omega)}\d s \\
  & +C t^{-1+\varepsilon}h^{2-2\varepsilon}.
%					   + C \int_{0}^{t}(t-s)^{-1 + \varepsilon} s^{-\varepsilon/2 }(\| u_h ( s ) \|_{H^{1}( \Omega )}^{\varepsilon}+\| u ( s ) \|_{H^{1}( \Omega )}^{\varepsilon})\|e(s)\|_{L^{2}(\Omega)} \d s.% \Big ( \|u (s)\|_{H^{1}(\Omega)}^{\varepsilon}+\|u_{h} (s)\|_{H^{1}(\Omega)}^{\varepsilon} \Big )\d s.
    \end{align*}
    %The result \eqref{error-h-ep-log} then follows from Gr\"onwall's inequality.
    Multiplying $t^{1-\varepsilon} $ on both sides derives that
    \begin{equation*}
    \begin{aligned}
&t^{1-\varepsilon} \|e (t)\|_{L^{2}(\Omega)}\\
\le &C t^{1-\varepsilon}\int_{0}^{t}(t-s)^{-1+\frac\varepsilon2} s^{-1+\frac{3\varepsilon}{4} } (\| u_h ( s ) \|_{H^{1}( \Omega )}^{\frac\varepsilon2}+\| u ( s ) \|_{H^{1}( \Omega )}^{\frac\varepsilon2}) s^{1-\varepsilon} \|e (s)\|_{L^{2}(\Omega)}\d s  + C h^{2-2\varepsilon}.
    \end{aligned}
    \end{equation*}
    By H\"older's inequality, we have
    \begin{align*}
	    \int_{0}^{t}(t-s)^{-1+\frac\varepsilon2} s^{-1+\frac{3\varepsilon}{4} }\|u (s)\|_{H^{1}(\Omega)}^{\frac\varepsilon2}\d s \le & \|u\|_{L^{2}(0,t;\dot H^{1}_{0}(\Omega))}^{\frac\varepsilon2}\Big(\int_{0}^{t}\Big[(t-s)^{-1+\frac\varepsilon2} s^{-1+\frac{3\varepsilon}{4} }\Big]^{\frac{4}{4-\varepsilon} } \,\d s\Big)^{\frac{4-\varepsilon}{4} }\\
	\le & C t^{-1+\varepsilon} \|u\|_{L^{2}(0,t;\dot H_{0}^{1}(\Omega))}^{\frac\varepsilon2}.
    \end{align*}
%      \begin{align*}
%	    \int_{0}^{t}(t-s)^{-1+\varepsilon} s^{-1+\frac{3\varepsilon}{2} }\|u_h (s)\|_{H^{1}(\Omega)}^{\varepsilon}\d s
%	\le & C t^{-1+2\varepsilon} \|u_h\|_{L^{2}(0,T;\dot H_{0}^{1}(\Omega))}^{\varepsilon}.
%    \end{align*}
    Combining the above inequalities above, we have
    \begin{align*}
	    t^{1-\varepsilon} \|e (t)\|_{L^{2}(\Omega)} \le & C h^{2-2 \varepsilon} + C  \Big (\|u\|_{L^{2}(0,t;\dot H_{0}^{1}(\Omega))}^{\frac\varepsilon2}
	    +\|u_h\|_{L^{2}(0,t;\dot H_{0}^{1}(\Omega))}^{\frac\varepsilon2}\Big ) \sup_{0 < s \le t} s^{1- \varepsilon}\|e (s)\|_{L^{2}(\Omega)}.
    \end{align*}
    Taking the supremum with respect to $ t $ on both sides deduce that
    \begin{align*}
	    \sup_{0 < t \le T} t^{1-\varepsilon} \|e (t)\|_{L^{2}(\Omega)} \le & \; C h^{2-2\varepsilon} \\
	  + C  \Big ( \|u\|_{L^{2}(0,T;\dot H_{0}^{1}(\Omega))}^{\frac\varepsilon2}&  + \|u_{h}\|_{L^{2}(0,T;\dot H_{0}^{1}(\Omega))}^{\frac\varepsilon2} \Big ) \sup_{0 < t \le T} t^{1-\varepsilon}\|e (t)\|_{L^{2}(\Omega)}.
    \end{align*}
According to \cite[Lemma 3.5]{2DNS2022Li}, for any small $\sigma>0$, there exists $T_\sigma>0$ such that
\begin{equation*}
\|u\|_{L^{2}(0,T;\dot H_{0}^{1}(\Omega))} + \|u_{h}\|_{L^{2}(0,T;\dot H_{0}^{1}(\Omega))}\le \sigma \quad  \forall T\in(0,T_\sigma].
\end{equation*}
%
%when the energy of the numerical solution decays, the following estimate can be obtained
%\begin{align*}
%	\sum_{n=1}^m\tau_n\|\nabla u_h^n\|_{L^2(\Omega)}\le \varepsilon \quad  \forall t\in(0,T_\varepsilon],
%    \end{align*}
%in which $\tau\rightarrow0$, we derive
%\begin{align*}
%\|u_{h}\|_{L^{2}(0,T;\dot H_{0}^{1}(\Omega))}\le \varepsilon \quad  \forall t\in(0,T_\varepsilon].
%\end{align*}
%    Since $ u, u_{h} \in L^{2}(0,T;\dot H_{0}^{1}(\Omega)) $, $ \|u\|_{L^{2}(0,T;\dot H_{0}^{1}(\Omega))} $ and $ \|u_{h}\|_{L^{2}(0,T;\dot H_{0}^{1}(\Omega))} $ tend to $ 0 $ as $ T \to 0 $.
If $ T $ satisfies $ C \Big(\|u\|_{L^{2}(0,T;\dot H_{0}^{1}(\Omega))}^{\varepsilon} + \|u_{h}\|_{L^{2}(0,T;\dot H_{0}^{1}(\Omega))}^{\varepsilon}\Big) < 1 $, then
we have
%the last term on the right-hand side can be absorbed by the left-hand side, which implies that
    \begin{align*}
        \sup_{0 < t \le T} t^{1-\varepsilon}\|e (t)\|_{L^{2}(\Omega)} \le C h^{2-2\varepsilon},
    \end{align*}
   and complete the proof of theorem.
    \end{proof}

\section{Fully discretization}%
\label{sec:fully}
In this section, we propose and analyze a fully discrete scheme by using a second-order implicit-explicit Runge--Kutta method.

\subsection{Runge-Kutta method and error equations} \label{section_3_1}
Let $0=t_0<t_1<...<t_N=T$ be a partition of the time interval $[0,T]$ with stepsize
\begin{equation}
\label{time1}
\tau_1 = \tau^{\frac{1}{1-\alpha}}\ \textrm{and} \ \tau_n=t_n-t_{n-1}\sim (t_{n-1}/T)^\alpha\tau\ \textrm{for}\ 2\le n\le N,
\end{equation}
where $\tau$ is the maximal stepsize, and $"\thicksim"$ means equivalent magnitude (up to a constant multiple). The parameter $\alpha \in (0,1)$ determines how fast the temporal grids are refined towards $t = 0$. The stepsizes defined in this way have the following properties:
\begin{enumerate}
  \item $\tau_n\sim\tau_{n-1}$ for two consecutive stepsizes.
  \item For any fixed integer $ M_0 $, $\tau_1 \sim \tau_{2} \sim \cdots \sim \tau_{M_0} \sim \tau^{\frac{1}{1-\alpha}}$, the equivalence depends on $ M_0 $, but is independent on $ \tau $ and $ n $. Hence, the starting stepsize is much smaller than the maximal stepsize.
  This resolves the solution's singularity near $t = 0$.
  \item The total number of time levels is $O(T/\tau)$. Therefore, the total computational cost is equivalent to using a uniform stepsize $\tau$.
\end{enumerate}

%With the nonuniform stepsizes defined above,
%we use the implicit Euler method for the first time step,
%\begin{equation}
%\label{atime22}
%\frac{u^1-u^{0}}{\tau_1}-Au^1+P_X(u^1\cdot \nabla u^1)=0,\ 1\le n\le N.\\
%\end{equation}
%The time discrete method in \eqref{atime22} can be expressed by the following equivalent form
%\begin{equation}
%\label{time2}
%u^1=r(\tau_1A)u^{0}-\tau_1r(\tau_1A)P_X(u^1\cdot \nabla u^1),\\
%\end{equation}
%where $r(\tau_1A)$ denotes the discrete semigroup in the time discretization defined by
%\begin{equation}
%\label{ti3}
%r(\tau_1A)=(1-\tau_1A)^{-1}.
%\end{equation}
%When $n\ge1$,
Next, we introduce an implicit Runge--Kutta method with $q$ stages for the time discretization of the evolution equation \eqref{project-equation}.
The coefficients of the method are given by the Butcher tableau
\begin{equation*}
 \begin{array}{c|c}
		%\begin{tabular}{c|c}
  a_{11} \ \cdots \ a_{1q}  & c_1 \\
    \vdots \quad \quad\quad \ \vdots   & \vdots\\
      a_{q1} \ \cdots \ a_{qq}  & c_q\\
\hline
  b_{1} \ \cdots \ b_{q}   &  \\
	\end{array}
\end{equation*}
with $c_1,\ldots,c_q\in(0,1]$. Here the quadrature points $c_i, 1 \le i \le q$, are distinct numbers in $[0,1]$ and the coefficients $a_{ij}$ and $b_j$ are associated with the quadrature formulas
\begin{align}
\label{formula}
\int_{0}^{1}\varphi\d t \approx \sum_{j=1}^{q}b_{j}\varphi(c_j),\qquad
\int_{0}^{c_i}\varphi\d t \approx\sum_{j=1}^{q}a_{ij}\varphi(c_j),\
i = 1,\ldots,q.
\end{align}
We assume that \eqref{formula} are exact for polynomials of degree $p-1$ and $p-2$, respectively. It implies that the method is accurate of order $p$.
Now we introduce error functionals for the quadrature formulae \eqref{formula} for the interval $(t_n,t_{n+1})$ as
\begin{equation}\label{eqn:kutta}
\begin{aligned}
&Q_{n,i}(\varphi)=\int_{t_n}^{t_{n,i}}\varphi\d s -\tau_{n+1}\sum_{j=1}^q a_{ij}\varphi(t_{n,j}),\quad i=1,\cdots,q, \\
&Q_{n+1}(\varphi)=\int_{t_n}^{t_{n+1}}\varphi\d s -\tau_{n+1}\sum_{i=1}^q b_{l}\varphi(t_{n,i}).
\end{aligned}
\end{equation}
Recall the assumption that the quadrature formulae \eqref{formula} are exact for polynomials of degree $p-1$ and $p-2$, respectively
(this means that the time discretization scheme is strictly accurate of $p$). As a result, we have \cite{Michel1987}
\begin{equation}\label{eqn:kutta-err}
\begin{aligned}
&\|Q_{n,i}(\varphi)\| \le C\tau_{n+1}^{l+1}\underset{t_n< s < t_{n+1}}{\sup}\|\varphi^{(l)}(s)\|  \quad \textrm{for} \ l\le p-1,\ i = 1,2. \\
&\|Q_{n+1}(\varphi)\| \le C\tau_{n+1}^{l+1}\underset{t_n< s < t_{n+1}}{\sup}\|\varphi^{(l)}(s)\| \quad \textrm{for} \ l\le p.
\end{aligned}
\end{equation}
where $ \| \cdot \| $ can be $ L^{2}(\Omega) $ norm or $ H^{1}(\Omega) $ norm.

%described by the Butcher tableau
% \begin{table}[!h]
%		\label{tab23333}
%		\centering
%		\begin{tabular}{c|c}
%   $a_{11}$ \ $\cdots$ \ $a_{1s}$  & $c_1$ \\
%    $\vdots$ \quad $\quad$\quad \ $\vdots$   & $\vdots$\\
%      $a_{s1}$ \ $\cdots$ \ $a_{ss}$  & $c_s$\\
%\hline
%  $b_{1}$ \ $\cdots$ \ $b_{s}$   &  \\
%\end{tabular}
%	\end{table}\\
%with $c_1,\ldots,c_s\in(0,1]$.
Taking $\mathcal{O}=(a_{ij})$, the vectors $b=(b_j)$ and $c=(c_i)^T$ for $i,j=1,\cdots,q$. Here we use the two-stage Lobatto IIIC  scheme, with $p=q=2$, namely
\[\mathcal{O}=
\begin{pmatrix}
\frac{1}{2}& -\frac{1}{2}\\
\frac{1}{2}& \frac{1}{2}
\end{pmatrix},\quad
b=\begin{pmatrix}
\frac{1}{2}& \frac{1}{2}
\end{pmatrix},
\quad
c=
\begin{pmatrix}
0\\
1
\end{pmatrix}.
\]
It is well-known that the method is implicit and algebraically stable \cite{Hairer2010}. In the numerical scheme, we linearize the nonlinear term in the Navier--Stokes equation. For a sequence of finite element functions $ \{v_{h}^{n,i}\} $ for $ n = 0,1,\cdots $ and $ i = 1,2 $, we define the extrapolation operator $ \hat{I}_{h} $ as follows:
\begin{equation}
\hat{I}_{h}v_{h}^{n,i}=
    \left\{\begin{array}{ll}
		    v_{h}^{0}, & n = 0,\\
		    v_{h}^{n} + c_{i}\frac{\tau_{n+1}}{\tau_{n}} (v_{h}^{n} - v_{h}^{n-1}), & n \ge 1.
    \end{array}\right.
\end{equation}
Then for a function $ f  $, we have the following error estimate for the extrapolation operator $ \hat{I}_{h} $ for $ n \ge 1 $:
\begin{align}\label{extrapolation_error}
    \| \hat{I}_{h}f (t_{n,i}) - f (t_{n,i}) \| \le C \tau_{n+1}^{2}\sup_{t_{n-1} < s < t_{n+1}} \| \partial_{t}^{2}f (\cdot,s) \| \quad \text{for}\quad i = 1,2,
\end{align}
where $ \| \cdot \| $ can be $ L^{2}(\Omega) $ norm or $ H^{1}(\Omega) $ norm.

For given numerical solutions $u_{h}^{n-1},u^n_h \in X_h$, we compute $u^{n+1}_h \in X_h$ by
\begin{subequations}\label{eqn:app}
\begin{align}
&u_h^{n,i}=u_h^n+\tau_{n+1}\sum_{j=1}^{2}a_{ij}\left[A_hu_h^{n,j}-P_{X_h} (P_{h}^{\text{RT}}\hat{I}_{h}u_h^{n,j}\cdot \nabla u_h^{n,j}) \right],\ i = 1,2, \label{eqn:app_a}\\
&u_h^{n+1}=u_h^n+\tau_{n+1}\sum_{i=1}^{2}b_{i}\left[A_hu_h^{n,i}-P_{X_h} (P_{h}^{\text{RT}}\hat{I}_{h}u_h^{n,i}\cdot \nabla u_h^{n,i}) \right], \label{eqn:app_b}
\end{align}
\end{subequations}
%together with the Dirichlet boundary conditions $u_h^{n,i}= 0$ on $\partial\Omega$.
%These equations are to be solved subsequently for $n =  1, 2,\ldots$
Here $u_h^{n,i}$ are approximations to $u_h(t_{n,i})$ for $i = 1,2$, with $t_{n,i}=t_n+c_i\tau_{n+1}$ being the internal Runge-Kutta nodes.
%We shall always assume that the former is exact for constants, so that
%\begin{equation}\begin{aligned}
%\label{formula}
%\sum_{j=1}^{q}b_{j}=1.
%\end{aligned}\end{equation}

\iffalse
The stability function of the Runge-Kutta method is the rational function
\begin{equation}%\label{rk3}
R(z) = 1 + zb^T(I-z\mathcal{O})^{-1}\amalg,
\end{equation}
where $\amalg = (1,1)^T\in \mathbb{R}^2$. The stability function is a rational approximation to the exponential function, i.e., $R(z) = e^z + O(z^{r+1})$ for $z\rightarrow0$, where $r \ge 0$ is less than or equal to the order of the Runge-Kutta method. Moreover, we set
\begin{equation}\label{rk4}
\begin{aligned}
&\theta(z)=\left(\theta_{ij}(z)\right)=\left(I-z\mathcal{O}\right)^{-1},\\
&\bar{s}(z)=\left(s_1(z),s_2(z)\right)^T=\theta(z)\amalg,\\
&S(z)=\left(s_{ij}(z)\right)=\theta(z)\mathcal{O},\\
&\sigma(z)^T=\left( \sigma_1(z),\sigma_2(z) \right)=b^T\theta(z).
\end{aligned}
\end{equation}
All these functions are bounded for $z\le0$. With these notations, the scheme \eqref{eqn:app} can be written as the following form,
\begin{align}
&u_h^{n,i}=s_i(\tau_{n+1}A_h)u_h^n-\tau_{n+1}\sum_{j=1}^{2}s_{ij}(\tau_{n+1}A_h)P_{X_h} (P_{h}^{\text{RT}}\hat{I}_{h}u_h^{n,j}\cdot \nabla u_h^{n,j}),\ i = 1,2,\label{rkgg3}\\
&u_h^{n+1}=R(\tau_{n+1}A_h)u_h^n-\tau_{n+1}\sum_{i=1}^{2}\sigma_{i}(\tau_{n+1}A_h)P_{X_h} (P_{h}^{\text{RT}}\hat{I}_{h}u_h^{n,i}\cdot \nabla u_h^{n,i}).\label{rk3}
\end{align}
%where the rational functions of $\tau_{n+1}A_h$ are defined by spectral representation. %Thus, \eqref{rk4} are all bounded linear operators.
\fi
Recalling the truncation errors $ Q_{n,i}(\partial_{t} u_{h}) $ and $ Q_{n}(\partial_{t}u_{h}) $,
we write the the semi-discrete solution $u_{h}$  as
\begin{align}
u_h(t_{n,i})=&u_h(t_{n})+\tau_{n+1}\sum_{j=1}^{2}a_{ij}\left[A_hu_h(t_{n,j})-P_{X_h} (P_{h}^{\text{RT}}\hat{I}_{h}u_h(t_{n,j})\cdot \nabla u_h(t_{n,j})) \right]\notag\\
	     & + \tau_{n+1}\sum_{j=1}^{2}a_{ij}\left[ P_{X_{h}}(P_{h}^{\text{RT}}\hat{I}_{h}u_{h}(t_{n,j})\cdot \nabla u_{h}(t_{n,j})) - P_{X_{h}}(P_{h}^{\text{RT}}u_{h}(t_{n,j})\cdot \nabla u_{h}(t_{n,j})) \right]\notag\\
	     &+Q_{n,i}(\partial_tu_h),\quad i = 1,2,\label{kutta5} \\
u_h(t_{n+1})
=&u_h(t_n)+\tau_{n+1}\sum_{i=1}^{2}b_{i}\left[A_hu_h(t_{n,i})-P_{X_h} (P_{h}^{\text{RT}}\hat{I}_{h}u_h(t_{n,i})\cdot \nabla u_h(t_{n,i})) \right]\notag\\
 & + \tau_{n+1}\sum_{i=1}^{2}b_{i}\left[ P_{X_{h}}(P_{h}^{\text{RT}}\hat{I}_{h}u_{h}(t_{n,i})\cdot \nabla u_{h}(t_{n,i})) - P_{X_{h}}(P_{h}^{\text{RT}}u_{h}(t_{n,i})\cdot \nabla u_{h}(t_{n,i})) \right]\notag\\
 &+Q_{n+1}(\partial_tu_h). \label{kutta6}
\end{align}
Now we define
\begin{align*}
    \mathcal G^{n,i} &= P_{X_{h}}(P_{h}^{\text{RT}}\hat{I}_{h}u_{h}(t_{n,i})\cdot \nabla u_{h}(t_{n,i})) - P_{X_{h}}(P_{h}^{\text{RT}}u_{h}(t_{n,i})\cdot \nabla u_{h}(t_{n,i})),\\
     T^{n,i} &=  - P_{X_{h}}(P_{h}^{\text{RT}}\hat{I}_{h}u_{h}^{n,i}\cdot \nabla u_{h}^{n,i}) + P_{X_{h}}(P_{h}^{\text{RT}}\hat{I}_{h}u_{h}(t_{n,i})\cdot \nabla u_{h}(t_{n,i})).
\end{align*}
%By \eqref{kutta5}-\eqref{kutta6}, it follows from similar calculations to \eqref{rkgg3}--\eqref{rk3} that
%\begin{align}
%u_h(t_{n,i})=&s_i(\tau_{n+1}A_h)u_h(t_n)-\tau_{n+1}\sum_{j=1}^{2}s_{ij}(\tau_{n+1}A_h)P_{X_h} (P_{h}^{\text{RT}}u_h(t_{n,j})\cdot \nabla u_h(t_{n,j}))\notag\\
%&+\sum_{j=1}^{2}\theta_{ij}(\tau_{n+1}A_h)Q_{n,j}(\partial_tu_h),\ i = 1,2,\label{kutta7}\\
%u_h(t_{n+1})=&R(\tau_{n+1}A_h)u_h(t_n)-\tau_{n+1}\sum_{i=1}^{2}\sigma_{i}(\tau_{n+1}A_h)P_{X_h} (P_{h}^{\text{RT}}u_h(t_{n,i})\cdot \nabla u_h(t_{n,i}))\notag\\
%&-\sum_{i=1}^{2}\tau_{n+1}A_h\sigma_{i}(\tau_{n+1}A_h)Q_{n,i}(\partial_tu_h)+Q_{n}(\partial_tu_h).\label{kutt8}
%\end{align}
Then the errors $\eta^{n+1}=u_h^{n+1}-u_h(t_{n+1})$ and $\eta^{n,i}=u_h^{n,i}-u_h(t_{n,i})$ satisfy% the following equations:
\begin{equation}\label{eqn:error-eq}
\begin{aligned}
	& \dot \eta^{n,i} = A_{h} \eta^{n,i} + T^{n,i},\\
&\eta^{n,i}=\eta^{n} + \tau_{n+1}\sum_{j=1}^{2}a_{ij}\dot\eta^{n,j} - \tau_{n+1}\sum_{j=1}^{2}a_{ij}\mathcal G^{n,j} - Q_{n,i}(\partial_{t}u_{h})\quad i = 1,2,\\
&\eta^{n+1}=\eta^{n} + \tau_{n+1}\sum_{i=1}^{2}b_{i}\dot \eta^{n,i}- \tau_{n+1}\sum_{i=1}^{2}b_{i}\mathcal G^{n,i} - Q_{n+1}(\partial_{t}u_{h}).
\end{aligned}
\end{equation}
%where $ T^{n,i} =  - P_{X_{h}}(P_{h}^{\text{RT}}\hat{I}_{h}u_{h}^{n,i}\cdot \nabla u_{h}^{n,i}) + P_{X_{h}}(P_{h}^{\text{RT}}\hat{I}_{h}u_{h}(t_{n,i})\cdot \nabla u_{h}(t_{n,i})) $.

In order to estimate the extrapolation error $ \mathcal G^{n,i} $, we first derive an \text{a priori}
estimate for $ A_{h}u_{h}(t_{n,i}) $. In combination with \eqref{l2estimate_Ahuh} and \eqref{linf_estimate_shalf}, we have
%By the equation \eqref{FEM-num-solution1}, we have
%\begin{align*}
%    \| A_{h}u_{h}(t_{n,i}) \|_{L^{2}(\Omega)}
%     &\le \| \partial_{t}u_{h}(t_{n,i}) \|_{L^{2}(\Omega)} + C \| P_h^{RT} \nabla u_{h}(t_{n,i}) \cdot u_{h}(t_{n,i}) \|_{L^2(\Omega)}\\
%     &\le \| \partial_{t}u_{h}(t_{n,i}) \|_{L^{2}(\Omega)} + C \| P_h^{RT} u_{h}(t_{n,i})\|_{L^4(\Omega)}  \|  \nabla  u_{h}(t_{n,i}) \|_{L^4(\Omega)}\\
%     &\le  \| \partial_{t}u_{h}(t_{n,i}) \|_{L^{2}(\Omega)} + C \| P_h^{RT} u_{h}(t_{n,i})\|_{L^4(\Omega)}
%     \|   \nabla u_{h}(t_{n,i}) \|_{L^2(\Omega)}^\frac12   \| A_h  u_{h}(t_{n,i}) \|_{L^2(\Omega)}^\frac12,
  %  &\le \| \partial_{t}u_{h}(t_{n,i}) \|_{L^{2}(\Omega)} + C \| u_{h}(t_{n,i}) \|_{L^{2}(\Omega)}^{1/2}\| u_{h}(t_{n,i}) \|_{H^{1}(\Omega)}\| A_{h}u_{h}(t_{n,i}) \|_{L^{2}(\Omega)}^{1/2}.
%\end{align*}
%where in the last inequality, we use the estimate  \cite[Lemma 3.1]{2DNS2022Li}
%\begin{align}\label{eqn:interp-fem}
% \| \nabla v_h\|_{L^4(\Omega)} \le    C \|  \nabla v_h \|_{L^2(\Omega)}^\frac12   \| A_h v_h \|_{L^2(\Omega)}^\frac12 \quad \forall~ v_h \in X_h.
% \end{align}
% Then, using the interpolation inequality and stability of $P_h^{RT}$, we arrive at
%\begin{align*}
% \| A_{h}u_{h}(t_{n,i}) \|_{L^{2}(\Omega)}
%     \le \| \partial_{t}u_{h}(t_{n,i}) \|_{L^{2}(\Omega)} + C \| u_{h}(t_{n,i}) \|_{L^{2}(\Omega)}^{1/2}\| u_{h}(t_{n,i}) \|_{H^{1}(\Omega)}\| A_{h}u_{h}(t_{n,i}) \|_{L^{2}(\Omega)}^{1/2}.
% \end{align*}
%Then we apply Yong's inequality and Lemma \ref{regularity-numerical-u} to obtain
\begin{equation}\label{discrete_H2_estimate_uh}
    \| A_{h}u_{h}(t_{n,i}) \|_{L^{2}(\Omega)} \le C t_{n,i}^{-1}.
\end{equation}

According to \eqref{extrapolation_error} and \eqref{discrete_H2_estimate_uh}, $ \mathcal G^{n,i} $ satisfies
%{\color{red}need the $ L^{4} $ stability of $ P_h^{\text{RT}} $} %the following error estimates:
\begin{equation}\label{G_estimate_L2}
\begin{aligned}
	    \| \mathcal G^{n,i} \|_{L^{2}(\Omega)} \le & C \| \hat{I}_{h}u_{h}(t_{n,i}) - u_{h}(t_{n,i}) \|_{L^{4}(\Omega)}\| \nabla u_{h}(t_{n,i}) \|_{L^{4}(\Omega)}\\
	    \le & C \| \hat{I}_{h}u_{h}(t_{n,i}) - u_{h}(t_{n,i}) \|_{L^{2}(\Omega)}^{1/2}\| \hat{I}_{h}u_{h}(t_{n,i}) - u_{h}(t_{n,i}) \|_{H^{1}(\Omega)}^{1/2}\\
		& \cdot \| u_{h}(t_{n,i}) \|_{H^{1}(\Omega)}^{1/2}\| A_{h}u_{h}(t_{n,i}) \|_{L^{2}(\Omega)}^{1/2}\\
	    \le & \tau_{n+1}^{2}t_{n+1}^{-3}.
\end{aligned}
\end{equation}
Similarly, we have the estimate in $H^{-1}$ norm
%{\color{red}need the $ H^{-1} $ stability of $ P_{X_h} $}
\begin{equation}\label{G_estimate_H_negative}
\begin{aligned}
\| \mathcal G^{n,i} \|_{H^{-1}(\Omega)} \le C \| P_h^{\text{RT}} (\hat{I}_{h}u_{h}(t_{n,i})- u_{h}(t_{n,i}))\otimes u_{h}(t_{n,i}) \|_{L^{2}(\Omega)}
	    \le C t_{n+1}^{-5/2}\tau_{n+1}^{2}.
\end{aligned}
\end{equation}
%By using \eqref{discrete_H2_estimate_uh}, we can also deduce the estimate of $ \| u_{h}(t_{n,i}) \|_{L^{\infty}(\Omega)} $
%\begin{align}\label{Linf_estimate_uhtni}
%	\| u_{h}(t_{n,i}) \|_{L^{\infty}(\Omega)}\le C \| u_{h}(t_{n,i}) \|_{L^{2}(\Omega)}^{1/2} \| A_{h}u_{h}(t_{n,i}) \|_{L^{2}(\Omega)}^{1/2} \le C t_{n+1}^{-1/2}.
%\end{align}
%where the first inequality follows from the discrete version of interpolation inequality, \textcolor{blue}{see Lemma \ref{}.}{\color{red}given reference or prove it}.

%where
%\begin{align}
%\phi^{n,i}=&-\tau_{n+1}\sum_{j=1}^{2}s_{ij}(\tau_{n+1}A_h)P_{X_h} (P_{h}^{\text{RT}}u_h^{n,j}\cdot \nabla u_h^{n,j}-P_{h}^{\text{RT}}u_h(t_{n,j})\cdot \nabla u_h(t_{n,j}))\notag\\
%&-\sum_{j=1}^{2}\theta_{ij}(\tau_{n+1}A_h)Q_{n,j}(\partial_tu_h)
%=:\phi_1^{n,i}+\phi_2^{n,i},\label{rutta11}\\
%\varphi^{n+1}=&
%-\tau_{n+1}\sum_{i=1}^{2}\sigma_{i}(\tau_{n+1}A_h)P_{X_h} (P_{h}^{\text{RT}}u_h^{n,i}\cdot \nabla u_h^{n,i}-P_{h}^{\text{RT}}u_h(t_{n,i})\cdot \nabla u_h(t_{n,i}))\notag\\
%&+\sum_{i=1}^{2}\tau_{n+1}A_h\sigma_{i}(\tau_{n+1}A_h)Q_{n,i}(\partial_tu_h)-Q_{n}(\partial_tu_h)
%=:\varphi_1^{n+1}+\varphi_2^{n+1}+\varphi_3^{n+1}.\label{rutta12}
%\end{align}
%By iterating \eqref{eqn:kutta0} with respect to $n$, it can be rewritten as
%\begin{equation}
%\label{eu3}
%\eta^{n+1}=\sum_{j=1}^{n+1}E_{j}^{n+1}(A_{h}) \varphi^j,
%\end{equation}
%where $E_{j}^{n+1}(A_{h})$ is defined as
%\begin{equation}
%\label{eu33}
%E_{j}^{n+1}(A_{h})=\left\{\begin{aligned}
%&1,\qquad\qquad\qquad\ j=n+1,\\
%&\prod_{l=j}^{n}R(\tau_{l+1}A_h), \ \ 1\le j \le n.\\
%\end{aligned}\right.
%\end{equation}

\subsection{Regularities of numerical solutions and estimates for operators}
In this subsection, we prove %some estimates for the operator $ (I-z \mathcal O)^{-1} $. And we prove
$ L^{2}(\Omega)^2 $ boundedness, $ L^{2}(0,T;H^{1}_{0}(\Omega)^2) $ boundedness and $ H^{1}(\Omega)^2 $ estimate of the fully discrete solution in \eqref{eqn:app} by using energy estimates.

%To this end, we first present some resolvent estimate for the discrete operator $A_h$.
%We omit the proof of the following lemma, whose proof follows from the standard resolvent estimate,
%duality argument and interpolation inequality.

%\begin{Lemma}\label{operatorle}
%    The discrete resolvent operator $ (I-\kappa A_h)^{-1}:X_{h} \to X_{h}  $ satisfies the following stability estimates for $ \kappa > 0 $.
%    \begin{align}
	    %\|(I-kA_h)^{-1} \|_{L^{2}\to L^{2}} & \le  C , \label{ope1}\\
%	    \|(I-\kappa A_h)^{-1} \|_{H^{-r_1}\to L^{2}} & \le C \kappa^{-\frac{r_1}{2} },\quad   0 \le r_1 \le 2, \label{ope2}\\
%	    \|(I-\kappa A_h)^{-1} \|_{W^{-1,r_2}\to L^{2}} & \le C \kappa^{-\frac{1}{r_2} }.\ \quad 1 < r_2 \le 2, \label{ope3}\\
%	    \| (I - \kappa A_{h})^{-1} \|_{H^{-2}\to H^{-1}} & \le C \kappa^{-\frac{1}{2}}.
%    \end{align}
%    where $ C $ is a constant independent of $ \kappa $ and mesh size $ h $.
%\end{Lemma}
%\begin{Remark}\label{rmk-estimate-operator}
%The entries of the matrix $(I - z \mathcal{O})^{-1}$ are rational functions where the degree of the denominator exceeds the degree of the numerator by one.
%Therefore the absolute value of each entry of the matrix $ (I - z \mathcal O)^{-1} $ is bounded by $ C / (1-z) $ for $ z < 0 $.
%Therefore, the operators $s_i(\kappa A_h)$ and $s_{ij}(\kappa A_h)$ defined in \eqref{operatorle} also satisfy the bounds in Lemma \ref{operatorle}.
%{\color{red}Then the results of Lemma \ref{operatorle} are valid for the operators in \eqref{rk4}. }
%\end{Remark}

The $ L^{2}(\Omega)^2 $ and $ L^{2} (0,T; H_0^1 (\Omega)^2) $ boundedness of the solution of the fully discrete scheme \eqref{eqn:app} is presented in the following lemma.
\begin{lemma} \label{leelle3.1} (Discrete energy decay for the NS equation)
Assume that $u_h^{n}\in X_{h}$ is given. %and the matrix $ \mathcal O $ is positive definite.
Then, the solutions $u_h^{n,i}\in X_{h}$, $i = 1,2$ and $ u_{h}^{n+1} \in X_{h} $
of fully discrete scheme \eqref{eqn:app} satisfy the following estimate:
\begin{align}
	\|u_h^{n+1}\|_{L^2(\Omega)}^2 \le & \|u_h^n\|_{L^2(\Omega)}^2-2\tau_{n+1}\sum_{i=1}^{2} b_i\|\nabla u_h^{n,i}\|_{L^2(\Omega)}^2,\quad \text{for}\quad n \ge 0,\label{numer-solution-l2-estimate}\\
	\sum_{i=1}^{2}\| u_{h}^{n,i} \|_{L^{2}(\Omega)}^2 \le & C \| u_h^n \|_{L^{2}(\Omega)}^{2} + C \tau_{n+1}\sum_{i=1}^2 \| \nabla u_h^{n,i} \|_{L^{2}(\Omega)}^{2} \notag\\
													& + C \tau_{n+1}^2 \sum_{i=1}^{2} \| \hat{I}_h u_h^{n,i} \|_{L^{2}(\Omega)}^{2}\| \hat{I}_h u_{h}^{n,i} \|_{H^{1}(\Omega)}^{2} \| u_h^{n,i} \|_{H^{1}(\Omega)}^{2} ,\quad \text{for}\quad n \ge 0.\label{l2-numer-stability}
\end{align}
\end{lemma}
\begin{proof}
First, we rewrite the numerical scheme  \eqref{eqn:app}    as
\begin{subequations}
\begin{align}
&\dot{u}_h^{n,i}=A_hu_h^{n,i}-P_{X_h} (P_{h}^{\text{RT}}\hat{I}_{h}u_h^{n,i}\cdot \nabla u_h^{n,i}),\ i = 1,2, \label{rkii3}\\
&u_h^{n,i}=u_h^n+\tau_{n+1}\sum_{j=1}^{2}a_{ij}\dot{u}_h^{n,j},\ i = 1,2, \label{rkii4} \\
&u_h^{n+1}=u_h^n+\tau_{n+1}\sum_{i=1}^{2}b_{i}\dot{u}_h^{n,i}.\label{rkii5}
\end{align}
\end{subequations}
%Radau IIA method {\cite[Section IV.5]{Hairer2010}} is an important class of Runge-Kutta method that are A-stable, have an invertible matrix $\mathcal{O}$ and have $R(\infty)=0$ for an arbitrary number of stages $q\ge1$. This method is the collocation method at the Radau nodes (with right-most node $c_q = 1$),  $b_j = a_{qj}$, so that
%$t_{n+1}=t_{n,q}$.
%We first prove \eqref{rkii4} for any solution of internal stages $u_h^{n,i}\in H_0^1(\Omega)^2$, $i = 1,2$. In this case, the nodal value $u_h^{n+1}$ defined by \eqref{rkii5} is in $H_0^1(\Omega)^2$.
According to the \eqref{rkii5}, we conclude
\begin{align}
\|u_h^{n+1}\|_{L^2(\Omega)}^2= &\left( u_h^n+\tau_{n+1}\sum_{i=1}^{2}b_{i}\dot{u}_h^{n,i},u_h^n+\tau_{n+1}\sum_{i=1}^{2}b_{i}\dot{u}_h^{n,i}\right)\notag\\
=&\|u_h^{n}\|_{L^2(\Omega)}^2 +2 \tau_{n+1}\sum_{i=1}^{2} b_i (\dot{u}_h^{n,i},u_h^n)+\tau_{n+1}^2\sum_{i,j=1}^{2} b_i b_j(\dot{u}_h^{n,i},\dot{u}_h^{n,j}).\label{u_h_n_L2}
\end{align}
Substituting \eqref{rkii4} into the second term on the right-hand side of \eqref{u_h_n_L2}, we obtain
\begin{align*}
	\|u_h^{n+1}\|_{L^2(\Omega)}^2= &
\|u_h^{n}\|_{L^2(\Omega)}^2 +2 \tau_{n+1}\sum_{i=1}^{2} b_i \Big(\dot{u}_h^{n,i},u_h^{n,i}-\tau_{n+1}\sum_{j=1}^{2}a_{ij}\dot{u}_h^{n,j}\Big)\\
& +\tau_{n+1}^2
\sum_{i,j=1}^{2} b_i b_j(\dot{u}_h^{n,i},\dot{u}_h^{n,j}).
\end{align*}
Hence
\begin{align*}
\|u_h^{n+1}\|_{L^2(\Omega)}^2= \|u_h^{n}\|_{L^2(\Omega)}^2 +2 \tau_{n+1}\sum_{i=1}^{2} b_i (\dot{u}_h^{n,i},u_h^{n,i})-\tau_{n+1}^2
\sum_{i,j=1}^{2} d_{ij}(\dot{u}_h^{n,i},\dot{u}_h^{n,j}),
\end{align*}
with $d_{ij}=b_ia_{ij}+b_ja_{ji}-b_ib_j$, $i,j=1,2$.
The scheme is algebraic stable, i.e.  the symmetric matrix  $(d_{ij})$
%stability means, besides the positivity of the weights $b_1,b_2$, that the symmetric matrix $ D\in R^{2\times 2}$ with entries $d_{ij}= b_ia_{ij} + b_ja_{ji}-b_ib_j$, $i, j=1,2$,
is positive semidefinite. Therefore,
%Using the positive semidefiniteness of the matrix $\mathcal{O}$ we infer that
\begin{align}
\label{energy2}
&\|u_h^{n+1}\|_{L^2(\Omega)}^2\le \|u_h^{n}\|_{L^2(\Omega)}^2 + 2\tau_{n+1}\sum_{i=1}^{2} b_i(\dot{u}_h^{n,i},u_h^{n,i}).
\end{align}
Testing \eqref{rkii3} with $u_h^{n,i}$ yields
\begin{align}
\label{energy3}
(\dot{u}_h^{n,i},u_h^{n,i})=-\|\nabla u_h^{n,i}\|_{L^2(\Omega)}^2- (P_{X_h} (P_{h}^{\text{RT}}\hat{I}_{h}u_h^{n,i}\cdot \nabla u_h^{n,i}),u_h^{n,i}),\ i = 1,2.
\end{align}
Note that $P_{h}^{\text{RT}}\hat{I}_{h}u_h^{n,i}$ is divergence free. Then we have
%It follows from integration by parts that
\begin{align*}
(P_{h}^{\text{RT}}\hat{I}_{h}u_h^{n,i}\cdot \nabla u_h^{n,i},u_h^{n,i})=\left(P_{h}^{\text{RT}}\hat{I}_{h}u_h^{n,i}, \nabla\frac{1}{2}|u_h^{n,i}|^2 \right)=-\left(\nabla \cdot P_{h}^{\text{RT}}\hat{I}_{h}u_h^{n,i},\frac{1}{2}|u_h^{n,i}|^2\right)=0.
\end{align*}
As a result, we  obtain the inequality \eqref{numer-solution-l2-estimate} by
substituting \eqref{energy3}   into \eqref{energy2}. %, we then obtain the inequality \eqref{numer-solution-l2-estimate}.

 To prove the $ L^{2} $ boundedness of $ u_{h}^{n,i} $, we test the equation \eqref{eqn:app_a} with $ u_h^{n,i} $ and obtain
	\begin{align}
		\| u_h^{n,i} \|_{L^{2}(\Omega)}^{2} = & (u_h^n , u_h^{n,i}) - \tau_{n+1} \sum_{j=1}^2 a_{ij} \Big[ (\nabla u_h^{n,j}, \nabla u_h^{n,i}) + (P_h^{\text{RT}} \hat{I}_h u_h^{n,j}\cdot \nabla u_h^{n,j}, u_h^{n,i})\Big] \notag \\
		\le & \frac{1}{2}\| u_h^{n,i} \|_{L^{2}(\Omega)}^{2} + \frac{1}{2}\| u_h^n \|_{L^{2}(\Omega)}^{2} + C \tau_{n+1} \sum_{j=1}^2 \| \nabla u_h^{n,j} \|_{L^{2}(\Omega)}^{2} \notag \\
			& + \tau_{n+1}\| u_{h}^{n,i} \|_{H^{1}(\Omega)} \sum_{j=1}^2 \| P_h^{\text{RT}} \hat{I}_h u_h^{n,j} \cdot \nabla u_h^{n,j} \|_{H^{-1}(\Omega)}\\
		\le & \frac{1}{2}\| u_h^{n,i} \|_{L^{2}(\Omega)}^{2} + \frac{1}{2}\| u_h^n \|_{L^{2}(\Omega)}^{2} + C \tau_{n+1} \sum_{j=1}^2 \| \nabla u_h^{n,j} \|_{L^{2}(\Omega)}^{2} \notag\\
			& +C \tau_{n+1}\sum_{j=1}^2 \| P_h^{\text{RT}} \hat{I}_h u_h^{n,j} \otimes u_h^{n,j} \|_{L^{2}(\Omega)}^{2}. \label{l2_estimate_uhni_mid}
	\end{align}
	By using H\"older's inequality, the estimate \eqref{lp_estimate_RTFEM}, we have
	\begin{align}
		\| P_h^{\text{RT}} \hat{I}_h u_h^{n,j} \otimes u_h^{n,j} \|_{L^{2}(\Omega)}^2 \le & \| P_{h}^{\text{RT}} \hat{I}_h u_h^{n,j} \|_{L^{4}(\Omega)}^{2} \| u_{h}^{n,j} \|_{L^{4}(\Omega)}^{2}\notag\\
		\le & C \| \hat{I}_h u_{h}^{n,j} \|_{L^{2}(\Omega)} \| \hat{I}_h u_{h}^{n,j} \|_{H^{1}(\Omega)} \| u_{h}^{n,j} \|_{L^{2}(\Omega)} \| u_h^{n,j} \|_{H^{1}(\Omega)}\label{l2estimate_PhRT_Ih_uhnj}
	\end{align}
	Substituting \eqref{l2estimate_PhRT_Ih_uhnj} into \eqref{l2_estimate_uhni_mid}, summing up the obtained inequality with respect to $ i $ from $ i=1 $ to $ i=2 $, and using Young's inequality, we obtain the desired result \eqref{l2-numer-stability}.
\end{proof}

Then next lemma gives an \text{a priori} estimate  for $\| \nabla u_{h}^{n,i} \|_{L^{2}(\Omega)}$.
\begin{lemma}\label{H1-estimate-numerical-solution}
	If $ u_{0} \in \dot L^{2}(\Omega) $, then the fully discrete scheme \eqref{eqn:app} satisfy% the following estimate
    \begin{align}
        \sum_{i=1}^{2}\| \nabla u_{h}^{n,i} \|_{L^{2}(\Omega)} \le C t_{n+1}^{-1/2},\; \; \; \; \text{for} \; \; \; \; n \ge 0. \label{H1_estimate_intenal}
    \end{align}
\end{lemma}
\begin{proof}
First of all, we note that the inequality \eqref{H1_estimate_intenal} holds when $ n = 0, 1, 2$
according to Lemma \ref{leelle3.1}. Then for $ n \ge 3 $, taking gradient on both sides of \eqref{rkii5} and squaring, we have
    \begin{align*}
	    \| \nabla u_{h}^{n+1} \|_{L^{2}(\Omega)}^{2} = \| \nabla u_{h}^{n} \|_{L^{2}(\Omega)}^{2} + 2 \tau_{n+1}\sum_{i=1}^{2}b_{i} (\nabla \dot{u}_{h}^{n,i},\nabla u_{h}^{n}) + \tau_{n+1}^{2}\sum_{i,j=1}^{2}b_{i}b_{j} (\nabla \dot{u}_{h}^{n,i}, \nabla \dot{u}_{h}^{n,j}).
    \end{align*}
Meanwhile, we recall \eqref{rkii4} and obtain
    \begin{align*}
        (\nabla u_{h}^{n}, \nabla \dot u_{h}^{n,i}) = (\nabla u_{h}^{n,i}, \nabla \dot u_{h}^{n,i}) - \tau_{n+1}\sum_{j=1}^{2} a_{ij} (\nabla \dot u_{h}^{n,j}, \nabla \dot u_{h}^{n,i})
    \end{align*}
Therefore,
\begin{align*}
	    \| \nabla u_{h}^{n+1} \|_{L^{2}(\Omega)}^{2} =& \| \nabla u_{h}^{n} \|_{L^{2}(\Omega)}^{2} + 2 \tau_{n+1}\sum_{i=1}^{2}b_{i} (\nabla u_{h}^{n,i},\nabla \dot u_{h}^{n,i}) - 2 \tau_{n+1}^{2}\sum_{i,j=1}^{2}a_{ij}b_{i} (\nabla \dot u_{h}^{n,i},\nabla \dot u_{h}^{n,j})\\
	&+ \tau_{n+1}^{2}\sum_{i,j=1}^{2}b_{i}b_{j} (\nabla \dot{u}_{h}^{n,i}, \nabla \dot{u}_{h}^{n,j})
    \end{align*}
Then we apply the algebraical stability of the scheme to obtain %since the symmetric matrix with entries $ d_{ij} = b_{i}a_{ij} + b_{j}a_{ji} - b_{i}b_{j} $ is positive semidefinite, we have
    \begin{align}\label{H1-energy-estimate-dot}
        \| \nabla u_{h}^{n+1} \|_{L^{2}(\Omega)}^{2} \le \| \nabla u_{h}^{n} \|_{L^{2}(\Omega)}^{2} + 2 \tau_{n+1}\sum_{i=1}^{2}b_{i} (\nabla u_{h}^{n,i},\nabla \dot u_{h}^{n,i}).
    \end{align}
    Testing \eqref{rkii3} with $ \dot u_{h}^{n,i} $, we have
    \begin{align}\label{gradu-graddotu-estimate}
        (\nabla u_{h}^{n,i}, \nabla \dot u_{h}^{n,i}) = - \| \dot u_{h}^{n,i} \|_{L^{2}(\Omega)}^{2}  - (P_{h}^{\text{RT}} \hat{I}_{h}u_{h}^{n,i}\cdot \nabla u_{h}^{n,i}, \dot u_{h}^{n,i}).
    \end{align}
    Substituting \eqref{gradu-graddotu-estimate} into \eqref{H1-energy-estimate-dot}, we have
    \begin{align*}
        \| \nabla u_{h}^{n+1} \|_{L^{2}(\Omega)}^{2} \le \| \nabla u_{h}^{n} \|_{L^{2}(\Omega)}^{2} - 2 \tau_{n+1}\sum_{i=1}^{2}b_{i}\| \dot u_{h}^{n,i} \|_{L^{2}(\Omega)}^{2} - 2 \tau_{n+1}\sum_{i=1}^{2}b_{i}(P_{h}^{\text{RT}}\hat{I}_{h} u_{h}^{n,i}\cdot \nabla u_{h}^{n,i}, \dot u_{h}^{n,i}).
    \end{align*}
    Since $ b_{i} > 0$ for each $ i = 1,2 $, by H\"older's inequality, we have
    \begin{equation}\label{H1-energy-middl3-nonlinear}
    \begin{aligned}
        \| \nabla u_{h}^{n+1} \|_{L^{2}(\Omega)}^{2} + & \tau_{n+1}\sum_{i=1}^{2}b_{i}\| \dot u_{h}^{n,i} \|_{L^{2}(\Omega)}^{2} \\
     &\qquad  \le \| \nabla u_{h}^{n} \|_{L^{2}(\Omega)}^{2} + C \tau_{n+1}\sum_{i=1}^{2}\| P_{h}^{\text{RT}} \hat{I}_{h}u_{h}^{n,i}\cdot \nabla u_{h}^{n,i} \|_{L^{2}(\Omega)}^{2}.
    \end{aligned}
    \end{equation}
    According to \cite[Lemma 3.1]{2DNS2022Li} and estimate \eqref{lp_estimate_RTFEM}, we have
    \begin{align}\label{numerical-nonlinear-l2}
        \| P_{h}^{\text{RT}} \hat{I}_{h}u_{h}^{n,i}\cdot \nabla u_{h}^{n,i} \|_{L^{2}(\Omega)} \le \|\hat{I}_{h} u_{h}^{n,i} \|_{L^{2}(\Omega)}^{1/2}\| \hat{I}_{h}u_{h}^{n,i} \|_{H^{1}(\Omega)}^{1/2} \| u_{h}^{n,i} \|_{H^{1}(\Omega)}^{1/2} \| A_{h}u_{h}^{n,i} \|_{L^{2}(\Omega)}^{1/2},
    \end{align}
    where we have used interpolation inequality. By \eqref{rkii3}, we have
    \begin{align}\label{numerical-laplacian-l2}
        \| A_{h}u_{h}^{n,i} \|_{L^{2}(\Omega)} \le \| \dot u_{h}^{n,i} \|_{L^{2}(\Omega)} + C \| P_{h}^{\text{RT}}\hat{I}_{h}u_{h}^{n,i}\cdot \nabla u_{h}^{n,i} \|_{L^{2}(\Omega)}.
    \end{align}
    Substituting \eqref{numerical-laplacian-l2} into \eqref{numerical-nonlinear-l2} and using Young's inequality, we have
    \begin{equation}\label{nonlinear-numer-final}
    \begin{aligned}
	    \| P_{h}^{\text{RT}}\hat{I}_{h} u_{h}^{n,i}\cdot \nabla u_{h}^{n,i} \|_{L^{2}(\Omega)} \le &  C\| \hat{I}_{h}u_{h}^{n,i} \|_{L^{2}(\Omega)}^{1/2}\| \hat{I}_{h}u_{h}^{n,i} \|_{H^{1}(\Omega)}^{1/2} \| u_{h}^{n,i} \|_{H^{1}(\Omega)}^{1/2}\| \dot u_{h}^{n,i} \|_{L^{2}(\Omega)}^{1/2}\\
	& + C \|\hat{I}_{h} u_{h}^{n,i} \|_{L^{2}(\Omega)}\| \hat{I}_{h}u_{h}^{n,i} \|_{H^{1}(\Omega)} \| u_{h}^{n,i} \|_{H^{1}(\Omega)}.
    \end{aligned}
    \end{equation}
Now we  substitute  this estimate into \eqref{H1-energy-middl3-nonlinear} and absorb $ \| \dot u_{h}^{n,i} \|_{L^{2}(\Omega)} $ on the right-hand side by using Young's inequality. Following from these steps and the definition of the extrapolation operator $ \hat{I}_{h} $, we obtain for $n\ge 3$
    \begin{equation*}%\label{H1-estimate-energy-uhni}
    \begin{aligned}
	  &\frac{1}{2}\tau_{n+1}\sum_{i=1}^{2}b_{i}\| \dot u_{h}^{n,i} \|_{L^{2}(\Omega)}^{2} + \| \nabla u_{h}^{n+1} \|_{L^{2}(\Omega)}^{2} - \| \nabla u_{h}^{n} \|_{L^{2}(\Omega)}^{2}\\	  \le &  C \tau_{n+1}\sum_{i=1}^{2} \|\hat{I}_{h} u_{h}^{n,i} \|_{L^{2}(\Omega)}^{2}\left(\|u_{h}^{n-1} \|_{H^{1}(\Omega)}^{2} + \| u_{h}^{n} \|_{H^{1}(\Omega)}^{2}\right)\| u_{h}^{n,i} \|_{H^{1}(\Omega)}^{2}.
    \end{aligned}
        \end{equation*}
 Due to the $ L^{2} $ boundedness of $ \hat{I}_h u_h^{n,i} $, we can Multiply $ t_{n+1} $ on both sides of the above estimate and obtain
			\begin{align}
				& \frac{1}{2}t_{n+1}\tau_{n+1} \sum_{i=1}^2 b_i \| \dot u_h^{n,i} \|_{L^{2}(\Omega)}^{2} + t_{n+1} \| \nabla u_h^{n+1} \|_{L^{2}(\Omega)}^{2} - t_n \| \nabla u_h^n \|_{L^{2}(\Omega)}^{2}\notag \\
				\le & \tau_{n+1}\| \nabla u_h^n \|_{L^{2}(\Omega)}^{2} + C \Big( t_{n-1} \| u_h^{n-1} \|_{H^{1}(\Omega)}^{2} + t_n \| u_h^n \|_{H^{1}(\Omega)}^{2}\Big) \tau_{n+1}\sum_{i=1}^2 \| u_h^{n,i} \|_{H^{1}(\Omega)}^{2} \notag \\
					& + C \Big( (\tau_n + \tau_{n+1}) \| u_h^{n-1} \|_{H^{1}(\Omega)}^{2} + \tau_{n+1}\| u_h^{n} \|_{H^{1}(\Omega)}^{2}\Big)\tau_{n+1} \sum_{i=1}^2 \| u_h^{n,i} \|_{H^{1}(\Omega)}^{2}. \label{tn1_H1_estimate_numerical}
			\end{align}
			From \eqref{numer-solution-l2-estimate}, we have that
			\begin{align}\label{l2H1_boundedness}
			    2\sum_{n=0}^m \tau_{n+1} \sum_{i=1}^2 b_i \| \nabla u_h^{n,i} \|_{L^{2}(\Omega)}^{2} \le \| u_h^0 \|_{L^{2}(\Omega)}^{2}.
			\end{align}
			Since $ \tau_{n-1}\sim \tau_n \sim \tau_{n+1} $ and $ t_3 \sim \tau_3 $, we can sum up \eqref{tn1_H1_estimate_numerical} with respect to $ n $ from $ 3 $ to $ m $ and obtain the following inequality in combination with \eqref{l2H1_boundedness}
			\begin{align*}
				 & \frac{1}{2}\sum_{n=3}^m t_{n+1}\tau_{n+1} \sum_{i=1}^{2} b_i \| \dot u_{h}^{n,i} \|_{L^{2}(\Omega)}^{2} + t_{m+1}\| \nabla u_h^{m+1} \|_{L^{2}(\Omega)}^{2} \\
				\le & C + C \sum_{n=3}^m \Big( t_{n-1} \| u_h^{n-1} \|_{H^{1}(\Omega)}^{2} + t_n \| u_h^n \|_{H^{1}(\Omega)}^{2}\Big) \tau_{n+1}\sum_{i=1}^2 \| u_h^{n,i} \|_{H^{1}(\Omega)}^{2}
			\end{align*}
			By using discrete Gronwall's inequality and \eqref{l2H1_boundedness}, we then obtain that
			\begin{align}\label{uhnidot_nablauhm_estimate}
			    \frac{1}{2}\sum_{n=3}^m t_{n+1}\tau_{n+1} \sum_{i=1}^{2} b_i \| \dot u_{h}^{n,i} \|_{L^{2}(\Omega)}^{2} + t_{m+1}\| \nabla u_h^{m+1} \|_{L^{2}(\Omega)}^{2} \le C.
			\end{align}

    Based on \eqref{energy3}, we have
    \begin{align}
\label{regu1}
\|\nabla u_h^{n,i}\|_{L^2(\Omega)}^2=&-(\dot{u}_h^{n,i},u_h^{n,i})=
(\dot{u}_h^{n,i},u_h^n+\tau_{n+1}\sum_{j=1}^{2}a_{ij}\dot{u}_h^{n,j})\notag\\
\le& \|\dot{u}_h^{n,i}\|_{H^{-1}(\Omega)}\|u_h^n\|_{H^{1}(\Omega)}
+C\tau_{n+1}\|\dot{u}_h^{n,i}\|_{L^{2}(\Omega)}\sum_{j=1}^{2}
\|\dot{u}_h^{n,j}\|_{L^{2}(\Omega)}.
\end{align}
It follows from \eqref{rkii3} and \eqref{lp_estimate_RTFEM} that
    \begin{align}
\label{regu22222}
\|\dot{u}_h^{n,i}\|_{H^{-1}(\Omega)} \le &C\| u_h^{n,i}\|_{H^{1}(\Omega)}+C\|P_{h}^{\text{RT}}\hat{I}_{h}u_h^{n,i}\otimes u_h^{n,i}\|_{L^{2}(\Omega)}\notag\\
\le &C\| u_h^{n,i}\|_{H^{1}(\Omega)}+C\|\hat{I}_{h}u_h^{n,i}\|_{L^{2}(\Omega)}^{1/2}\|\hat{I}_{h}u_h^{n,i}\|_{H^{1}(\Omega)}^{1/2}\|u_h^{n,i}\|_{L^{2}(\Omega)}^{1/2}\|u_h^{n,i}\|_{H^{1}(\Omega)}^{1/2}.
\end{align}
Substituting \eqref{regu22222} into \eqref{regu1}, it gives
    \begin{align*}
%\label{regu33}
&\|\nabla u_h^{n,i}\|_{L^2(\Omega)}^2\notag\\ \le &C\|u_h^n\|_{H^{1}(\Omega)}\| u_h^{n,i}\|_{H^{1}(\Omega)}+C\|u_h^n\|_{H^{1}(\Omega)}\|\hat{I}_{h}u_h^{n,i}\|_{H^{1}(\Omega)}^{1/2}\|u_h^{n,i}\|_{H^{1}(\Omega)}^{1/2}
+C\tau_{n+1}\sum_{j=1}^{2}
\|\dot{u}_h^{n,j}\|_{L^{2}(\Omega)}^2\notag\\
\le & Ct_{n}^{-1/2}\| u_h^{n,i}\|_{H^{1}(\Omega)}+Ct_{n}^{-3/4}\|u_h^{n,i}\|_{H^{1}(\Omega)}^{1/2}+
C\tau_{n+1}\sum_{j=1}^{2}
\|\dot{u}_h^{n,j}\|_{L^{2}(\Omega)}^2\notag\\
\le & \delta \| u_h^{n,i}\|_{H^{1}(\Omega)}^2+C t_{n+1}^{-1}+
C\tau_{n+1}\sum_{j=1}^{2}
\|\dot{u}_h^{n,j}\|_{L^{2}(\Omega)}^{2,}
\end{align*}
where $ \delta > 0 $ is a sufficiently small number so that the first term can be absorbed by the left-hand side.  Combining \eqref{uhnidot_nablauhm_estimate}, we obtain the desired estimate.
%    \begin{align}
%\label{regu2}
%&\sum_{i=1}^2\|\nabla u_h^{n,i}\|_{L^2(\Omega)}^2 \le C t_{n+1}^{-1}.
%\end{align}
%    Under the assumption $ u_{h}^{n,2} = u_{h}^{n+1} $, which gives the $ L^{2}(0,T;H^{1}(\Omega)^2) $ boundedness of $ u_{h}^{n} $. The proof is completed.
\end{proof}

\subsection{Error analysis}
The following lemma gives an \text{a priori} error bound for the time discretization.
\begin{lemma}\label{H_negative_estimate_error}
%Assume that the scheme \eqref{rkgg3}-\eqref{rk3} is both accurate and strictly accurate of order $2$.
Let $ u_h(t_{n+1}) $ be the solution to the semidiscrete scheme \eqref{FEM-num-solution1}. If the time stepsizes satisfy \eqref{time1} with a fixed constant $\alpha \in \left(\frac{3}{4},1\right)$ and $ u_h^{n+1} $ is the solution to the fully discrete scheme \eqref{eqn:app}, %then there is a fixed integer $ N^{*} $ independent of $ \tau $ and $ h $ such that when $ n \ge N^* $,
the error $ \eta^{n+1}:= u_h^{n+1}  - u_h(t_{n+1})  $ satisfies the following error bound
    \begin{align}%\label{error-h-ep-log}
	    \Big(\sum_{l=0}^{n}\tau_{l+1}\sum_{i=1}^{2}b_{i}\| \eta^{l,i} \|_{L^{2}(\Omega)}^{2}\Big)^{1/2} + \|\eta^{n+1}\|_{H^{-1}(\Omega)}\le
 Ct_{n+1}^{-3/2-\varepsilon}\tau_{n+1}^{2}.
    \end{align}
Here $ \varepsilon \in  (0, 2 \alpha - 3/2) $ could be arbitrarily small.
\end{lemma}
\begin{proof}
    Multiplying $ -A_{h}^{-1} $ on both sides of the third relation in \eqref{eqn:error-eq} and testing with $ \eta^{n+1} $, we obtain
    \begin{align*}
	    & \| \eta^{n+1} \|_{H^{-1}(\Omega)}^{2}- \| \eta^{n} \|_{H^{-1}(\Omega)}^{2}\\
	    =&   2\tau_{n+1}\sum_{i=1}^{2}b_{i}(-A_{h}^{-1}\dot \eta^{n,i}, \eta^{n}) + 2\Big(A_{h}^{-1}Q_{n+1}(\partial_{t}u_{h}),\eta^{n}\Big) \\
						       & + 2\tau_{n+1}\sum_{i=1}^{2}b_{i}(A_{h}^{-1}\mathcal G^{n,i}, \eta^{n}) - \tau_{n+1}^{2}\sum_{i,j=1}^{2}b_{i}b_{j}(A_{h}^{-1}\dot \eta^{n,i}, \dot\eta^{n,j})+ \| Q_{n+1}(\partial_{t}u_h) \|_{H^{-1}(\Omega)}^{2} \\
						       & - 2\tau_{n+1}^{2}\sum_{i,j=1}^{2}b_{i}b_{j}(-A_{h}^{-1}\dot \eta^{n,i},\mathcal G^{n,i})-2\tau_{n+1}\sum_{i=1}^{2}b_{i}\Big(-A_{h}^{-1}\dot \eta^{n,i}, Q_{n+1}(\partial_{t}u_{h})\Big) \\
						       & + \tau_{n+1}^{2}\sum_{i,j=1}^{2}b_{i}b_{j}(-A_{h}^{-1}\mathcal G^{n,i}, \mathcal G^{n,j}) + 2\tau_{n+1}\sum_{i=1}^{2}b_{i}\Big(-A_{h}^{-1}\mathcal G^{n,i}, Q_{n+1}(\partial_{t}u_{h})\Big).
    \end{align*}
Next  the second relation in \eqref{eqn:error-eq} and  the algebraical stability of the Gauss--Lobatto IIIC   lead to
%
%Similarly to the proof of Lemma \ref{leelle3.1}, from \eqref{error_equation_b} we know the expression of $ \eta^{n} $. We substitute the expression of $ \eta^{n} $ into $ (-A_{h}^{-1}\dot \eta^{n,i}, \eta^{n}) $. By utilizing the positive semidefinite of the symmetric matrix consisting of entries $ d_{ij} = b_{i}a_{ij}+ b_{j}a_{ji} - b_{i}b_{j} $, we obtain the following inequality:
\begin{align}
	    & \| \eta^{n+1} \|_{H^{-1}(\Omega)}^{2}- \| \eta^{n} \|_{H^{-1}(\Omega)}^{2} \notag\\
	    \le & 2\tau_{n+1}\sum_{i=1}^{2}b_{i}\Big(-A_{h}^{-1}\dot \eta^{n,i}, \eta^{n,i}+\tau_{n+1}\sum_{j=1}^{2}a_{ij}\mathcal G^{n,j}+Q_{n,i}(\partial_{t}u_{h})\Big)\notag\\
						   & - 2\tau_{n+1}\sum_{i=1}^{2}b_{i}(-A_{h}^{-1}\mathcal G^{n,i}, \eta^{n}) + 2\Big(A_{h}^{-1}Q_{n+1}(\partial_{t}u_{h}),\eta^{n}\Big) + \| Q_{n+1}(\partial_{t}u_h) \|_{H^{-1}(\Omega)}^{2} \notag\\
						       & - 2\tau_{n+1}^{2}\sum_{i,j=1}^{2}b_{i}b_{j}(-A_{h}^{-1}\dot \eta^{n,i},\mathcal G^{n,j})-2\tau_{n+1}\sum_{i=1}^{2}b_{i}\Big(-A_{h}^{-1}\dot \eta^{n,i}, Q_{n+1}(\partial_{t}u_{h})\Big)\notag \\
						       & + \tau_{n+1}^{2}\sum_{i,j=1}^{2}b_{i}b_{j}(-A_{h}^{-1}\mathcal G^{n,i}, \mathcal G^{n,j}) + 2\tau_{n+1}\sum_{i=1}^{2}b_{i}\Big(-A_{h}^{-1}\mathcal G^{n,i}, Q_{n+1}(\partial_{t}u_{h})\Big).\label{inequality_Hnegative_energy_estimate}
    \end{align}
    It follows from the first relation of \eqref{eqn:error-eq} that $ -A_{h}^{-1}\dot \eta^{n,i}= - \eta^{n,i} - A_{h}^{-1}T^{n,i} $. Then we obtain
    \begin{align}
	    & 2\tau_{n+1}\sum_{i=1}^{2}b_{i}\| \eta^{n,i} \|_{L^{2}(\Omega)}^{2}+\| \eta^{n+1} \|_{H^{-1}(\Omega)}^{2}-\| \eta^{n} \|_{H^{-1}(\Omega)}^{2} \notag\\
	    \le& 2\tau_{n+1}\sum_{i=1}^{2}b_{i}(-A_{h}^{-1}T^{n,i},\eta^{n,i})-  2\tau_{n+1}\sum_{i=1}^{2}b_{i}\Big(  \eta^{n,i} + A_{h}^{-1}T^{n,i}, \tau_{n+1}\sum_{j=1}^{2}a_{ij}\mathcal G^{n,j}+Q_{n,i}(\partial_{t}u_{h})\Big)\notag\\
						   & - 2\tau_{n+1}\sum_{i=1}^{2}b_{i}(-A_{h}^{-1}\mathcal G^{n,i}, \eta^{n}) + 2\Big(A_{h}^{-1}Q_{n+1}(\partial_{t}u_{h}),\eta^{n}\Big) + \| Q_{n+1}(\partial_{t}u_h) \|_{H^{-1}(\Omega)}^{2} \notag\\
						       & + 2\tau_{n+1}^{2}\sum_{i,j=1}^{2}b_{i}b_{j}(\eta^{n,i} + A_{h}^{-1}T^{n,i},\mathcal G^{n,j})+2\tau_{n+1}\sum_{i=1}^{2}b_{i}\Big(\eta^{n,i} + A_{h}^{-1}T^{n,i}, Q_{n+1}(\partial_{t}u_{h})\Big) \notag\\
						       & + \tau_{n+1}^{2}\sum_{i,j=1}^{2}b_{i}b_{j}(-A_{h}^{-1}\mathcal G^{n,i}, \mathcal G^{n,j}) + 2\tau_{n+1}\sum_{i=1}^{2}b_{i}\Big(-A_{h}^{-1}\mathcal G^{n,i}, Q_{n+1}(\partial_{t}u_{h})\Big).\label{mid_energy_estimate_H_negative}
    \end{align}
    We estimate the terms on the right-hand side of \eqref{mid_energy_estimate_H_negative} subsequently. For $ n \ge 2 $,
	the first term can be bounded by
	%{\color{red}need the $ W^{1,p} $ estimate of $ A_h u_h = f $ and $ W^{-1,p} $ stability of $ P_{X_h} $ and $ L^{p} $ stability of $ P_h^{\text{RT}} $}
    \begin{align}\label{T_ni_eta_ni_negative}
	    & 2\tau_{n+1}\sum_{i=1}^{2}b_{i}(-A_{h}^{-1}T^{n,i},\eta^{n,i})\notag\\
	    \le & C \tau_{n+1}\sum_{i=1}^{2}\| \eta^{n,i} \|_{W^{-1,4}(\Omega)}\| T^{n,i} \|_{W^{-1,4/3}(\Omega)}\notag\\
	    \le & C \tau_{n+1}\sum_{i=1}^{2}\| \eta^{n,i} \|_{W^{-1,4}(\Omega)}\Big(\| P_h^{\text{RT}}\hat{I}_{h}\eta^{n,i} \otimes u_{h}^{n,i} \|_{L^{4/3}(\Omega)}+\|P_h^{\text{RT}} \hat{I}_{h}u_{h}(t_{n,i})\otimes \eta^{n,i} \|_{L^{4/3}(\Omega)}\Big)\notag\\
	    \le & C \tau_{n+1}\sum_{i=1}^{2}\| \eta^{n,i} \|_{H^{-1}(\Omega)}^{1/2}\| \eta^{n,i} \|_{L^{2}(\Omega)}^{1/2}\Big(\| \hat{I}_{h}\eta^{n,i} \|_{L^{2}(\Omega)}\| u_{h}^{n,i} \|_{L^{4}(\Omega)} \notag\\
	    & \hskip3in + \|P_{h}^{\text{RT}} \hat{I}_{h}u_{h}(t_{n,i}) \|_{L^{4}(\Omega)}\| \eta^{n,i} \|_{L^{2}(\Omega)}\Big).
    \end{align}
By testing the second relation in \eqref{eqn:error-eq} with $-A_{h}^{-1}\eta^{n,i}$ and using the first relation in  \eqref{eqn:error-eq},
we obtain %an estimate for the second term on the right-hand side of equation \eqref{mid_energy_estimate_H_negative}.
    \begin{equation}\label{eta_ni_H_negative_estimate}
 \begin{aligned}
	    \| \eta^{n,i} \|_{H^{-1}(\Omega)}^{2} \le& C \| \eta^{n} \|_{H^{-1}(\Omega)}^{2}  + C \tau_{n+1}^{2} \sum_{j=1}^{2}\Big( \|T^{n,j}\|_{H^{-1}(\Omega)}^{2}
	+ \|\mathcal G^{n,j}\|_{H^{-1}(\Omega)}^{2}\Big) \\
	&+ C \tau_{n+1}\sum_{i=1}^{2}\| \eta^{n,j} \|_{L^{2}(\Omega)}^{2}+ C \|Q_{n,i}(\partial_{t}u_{h})\|_{H^{-1}(\Omega)}^{2}.
    \end{aligned}
        \end{equation}
    By substituting \eqref{eta_ni_H_negative_estimate} into \eqref{T_ni_eta_ni_negative} and utilizing H\"older's inequality together the imbedding $ H^{1/2}(\Omega) \hookrightarrow L^{4}(\Omega) $, we obtain
    \begin{align}\label{mid_estimate_nonlinear_error}
	& 2\tau_{n+1}\sum_{i=1}^{2}b_{i}(-A_{h}^{-1}T^{n,i},\eta^{n,i})\notag\\
	\le & \delta \tau_{n+1}\sum_{i=1}^{2}b_{i}\Big(\| \eta^{n,i} \|_{L^{2}(\Omega)}^{2} + \| \hat{I}_{h}\eta^{n,i} \|_{L^{2}(\Omega)}^{2}\Big) + C \tau_{n+1}^{3/2}\sum_{i=1}^{2}\| \eta^{n,i} \|_{L^{2}(\Omega)}^{2}\| u_{h}^{n,i} \|_{L^{4}(\Omega)}^{2}\notag\\
	    & + C_{\delta} \tau_{n+1}\sum_{i=1}^{2}\| \eta^{n} \|_{H^{-1}(\Omega)}^{2}\Big(\| u_{h}^{n,i} \|_{L^{4}(\Omega)}^{4}+\| P_{h}^{\text{RT}} \hat{I}_{h}u_{h}(t_{n,i}) \|_{L^{4}(\Omega)}^{4}\Big) \notag\\
	    & + C_{\delta} \tau_{n+1}^{5/4}\sum_{i=1}^{2}\| \eta^{n,i} \|_{L^{2}(\Omega)}^{2}\| P_{h}^{\text{RT}} \hat{I}_{h}u_{h}(t_{n,i}) \|_{L^{4}(\Omega)}+ C_{\delta} \tau_{n+1}\sum_{i=1}^{2}\| Q_{n,i}(\partial_{t}u_{h}) \|_{H^{-1}(\Omega)}^{2}\notag\\
	    & + C_{\delta} \tau_{n+1}^{3/2}\sum_{j=1}^{2}\Big( \| T^{n,j} \|_{H^{-1}(\Omega)}^{1/2} + \| \mathcal G^{n,j} \|_{H^{-1}(\Omega)}^{1/2}\Big)\notag\\
	    &\cdot\sum_{i=1}^{2}\| \eta^{n,i} \|_{L^{2}(\Omega)}^{1/2}\Big(\| \hat{I}_{h}\eta^{n,i} \|_{L^{2}(\Omega)}\| u_{h}^{n,i} \|_{L^{4}(\Omega)} + \| P_{h}^{\text{RT}} \hat{I}_{h}u_{h}(t_{n,i}) \|_{L^{4}(\Omega)}\| \eta^{n,i} \|_{L^{2}(\Omega)}\Big).
    \end{align}
%From Lemma \ref{discreteinter} and \eqref{discrete_H2_estimate_uh}, we can also deduce the estimate of $ \| u_{h}(t_{n,i}) \|_{L^{\infty}(\Omega)} $
%\begin{align}\label{Linf_estimate_uhtni}
%	\| u_{h}(t_{n,i}) \|_{L^{\infty}(\Omega)}\le C \| u_{h}(t_{n,i}) \|_{L^{2}(\Omega)}^{1/2} \| A_{h}u_{h}(t_{n,i}) \|_{L^{2}(\Omega)}^{1/2} \le C t_{n+1}^{-1/2},
%\end{align}
%where the first inequality follows from the discrete version of interpolation inequality, \textcolor{blue}{see Lemma \ref{discreteinter}.}{\color{red}given reference or prove it}.
By using Lemma \ref{regularity-numerical-u}, Lemma \ref{H1-estimate-numerical-solution}, the $ L^{\infty} $ estimate \eqref{linf_estimate_shalf} of $ u_h $, the $ L^{\infty} $ stability estimate \eqref{lp_estimate_RTFEM} and the interpolation inequality, we have % {\color{red}need the $ L^{\infty} $ stability of $ P_h^{\text{RT}} $} %the following estimates for $ \| T^{n,j} \|_{H^{-1}(\Omega)} $:
    \begin{align}
	   \| T^{n,j} \|_{H^{-1}(\Omega)}  \le &  \| P_h^{\text{RT}}\hat{I}_{h}\eta^{n,j}\otimes ( \eta^{n,j} + u_h (t_{n,j})) \|_{L^{2}(\Omega)} + \| P_h^{\text{RT}} \hat{I}_{h}u_{h}(t_{n,j})\otimes \eta^{n,j} \|_{L^{2}(\Omega)}\notag\\
	    \le & C\| \hat{I}_{h}\eta^{n,j} \|_{L^{2}(\Omega)}^{1/2}\| \hat{I}_{h}\eta^{n,j} \|_{H^{1}(\Omega)}^{1/2}\| \eta^{n,j} \|_{L^{2}(\Omega)}^{1/2}\| \eta^{n,j} \|_{H^{1}(\Omega)}^{1/2}\notag\\
			& + \| \hat{I}_{h}\eta^{n,j} \|_{L^{2}(\Omega)}\| u_{h}(t_{n,j}) \|_{L^{\infty}(\Omega)}\notag\\
			&+ C\| \eta^{n,j} \|_{L^{2}(\Omega)}\Big(\| \hat{I}_{h}u_{h}(t_{n,j}) \|_{L^{\infty}(\Omega)} + \| \hat{I}_h u_h (t_{n,j}) \|_{H^{1}(\Omega)} \Big)\notag\\
	    \le & C t_{n+1}^{-1/2}\| \hat{I}_{h}\eta^{n,j} \|_{L^{2}(\Omega)}^{1/2}\| \eta^{n,j} \|_{L^{2}(\Omega)}^{1/2} + C t_{n+1}^{-1/2}\Big(\| \hat{I}_{h}\eta^{n,j} \|_{L^{2}(\Omega)}+\| \eta^{n,j} \|_{L^{2}(\Omega)}\Big),\label{T_nj_H_negative_estimate}
    \end{align}
	By substituting \eqref{T_nj_H_negative_estimate} and \eqref{G_estimate_H_negative} into \eqref{mid_estimate_nonlinear_error},  and using Corollary \ref{11regularity-numerical-u} and estimate \eqref{eqn:kutta-err}, we obtain
    \begin{align}
	& 2\tau_{n+1}\sum_{i=1}^{2}b_{i}(-A_{h}^{-1}T^{n,i},\eta^{n,i})\notag\\
	    \le &  \delta \tau_{n+1}\sum_{i=1}^{2}b_{i}\Big(\| \eta^{n,i} \|_{L^{2}(\Omega)}^{2} + \| \hat{I}_{h}\eta^{n,i} \|_{L^{2}(\Omega)}^{2}\Big) + C_{\delta} \tau_{n+1}^{3/2}t_{n+1}^{-1/2}\sum_{i=1}^{2}\Big(\| \eta^{n,i} \|_{L^{2}(\Omega)}^{2}+\| \hat{I}_{h}\eta^{n,i} \|_{L^{2}(\Omega)}^{2}\Big)\notag\\
		& + C_{\delta} \tau_{n+1}\sum_{i=1}^{2}\| \eta^{n} \|_{H^{-1}(\Omega)}^{2}\Big(\| u_{h}^{n,i} \|_{H^{1}(\Omega)}^{2}+\| \hat{I}_{h}u_{h}(t_{n,i}) \|_{H^{1}(\Omega)}^{2}\Big) \notag\\
		& + C_{\delta} \tau_{n+1}^{5/4}t_{n+1}^{-1/4}\sum_{i=1}^{2}\| \eta^{n,i} \|_{L^{2}(\Omega)}^{2} + C_{\delta} \tau_{n+1}^{5}t_{n+1}^{-3} + C_{\delta} \tau_{n+1}^{7}t_{n+1}^{-6}.\label{Hnegative_estimate_nonlinear_eta}
    \end{align}
    %For the second term, we have
    %\begin{align}
%	& -  2\tau_{n+1}\sum_{i=1}^{2}b_{i}\Big(  \eta^{n,i} + A_{h}^{-1}T^{n,i}, \tau_{n+1}\sum_{j=1}^{2}a_{ij}\mathcal G^{n,j}+Q_{n,i}(\partial_{t}u_{h})\Big)\notag\\
%	\le &
%    \end{align}
    Substituting \eqref{Hnegative_estimate_nonlinear_eta} into \eqref{mid_energy_estimate_H_negative} together with estimates \eqref{eqn:kutta-err}, \eqref{eqn:kutta-err}, \eqref{G_estimate_H_negative}, \eqref{G_estimate_L2} and \eqref{T_nj_H_negative_estimate}, we obtain
    \begin{align}
        & 2\tau_{n+1}\sum_{i=1}^{2}b_{i}\| \eta^{n,i} \|_{L^{2}(\Omega)}^{2}+\| \eta^{n+1} \|_{H^{-1}(\Omega)}^{2}-\| \eta^{n} \|_{H^{-1}(\Omega)}^{2} \notag\\
 \le &  \delta \tau_{n+1}\sum_{i=1}^{2}b_{i}\Big(\| \eta^{n,i} \|_{L^{2}(\Omega)}^{2} + \| \hat{I}_{h}\eta^{n,i} \|_{L^{2}(\Omega)}^{2}\Big) + C_{\delta} \tau_{n+1}^{3/2}t_{n+1}^{-1/2}\sum_{i=1}^{2}\Big(\| \eta^{n,i} \|_{L^{2}(\Omega)}^{2}+\| \hat{I}_{h}\eta^{n,i} \|_{L^{2}(\Omega)}^{2}\Big)\notag\\
		& + C_{\delta} \tau_{n+1}\sum_{i=1}^{2}\| \eta^{n} \|_{H^{-1}(\Omega)}^{2}\Big(\| u_{h}^{n,i} \|_{H^{1}(\Omega)}^{2}+\| \hat{I}_{h}u_{h}(t_{n,i}) \|_{H^{1}(\Omega)}^{2}\Big)+ C_{\delta} \tau_{n+1}^{5/4}t_{n+1}^{-1/4}\sum_{i=1}^{2}\| \eta^{n,i} \|_{L^{2}(\Omega)}^{2} \notag\\
		& + C_{\delta} \tau_{n+1}^{3}t_{n+1}^{-5/2}\| \eta^{n} \|_{H^{-1}(\Omega)} + C \tau_{n+1}^{5}t_{n+1}^{-3} + C_{\delta} \tau_{n+1}^{6}t_{n+1}^{-5} + C \tau_{n+1}^{7}t_{n+1}^{-6}   \label{energy_estimate_H_negative_eta_1}\\
\le &  \delta\tau_{n+1}\sum_{i=1}^{2}b_{i}\Big(\| \eta^{n,i} \|_{L^{2}(\Omega)}^{2} + \| \hat{I}_{h}\eta^{n,i} \|_{L^{2}(\Omega)}^{2}\Big) + C_{\delta} \tau_{n+1}^{3/2}t_{n+1}^{-1/2}\sum_{i=1}^{2}\Big(\| \eta^{n,i} \|_{L^{2}(\Omega)}^{2}+\| \hat{I}_{h}\eta^{n,i} \|_{L^{2}(\Omega)}^{2}\Big)\notag\\
		& + C_{\delta} \tau_{n+1}\sum_{i=1}^{2}\| \eta^{n} \|_{H^{-1}(\Omega)}^{2}\Big(\| u_{h}^{n,i} \|_{H^{1}(\Omega)}^{2}+\| \hat{I}_{h}u_{h}(t_{n,i}) \|_{H^{1}(\Omega)}^{2}\Big)+ C_{\delta} \tau_{n+1}^{5/4}t_{n+1}^{-1/4}\sum_{i=1}^{2}\| \eta^{n,i} \|_{L^{2}(\Omega)}^{2} \notag\\
		& + C_{\delta} \tau_{n+1}t_{n+1}^{-1+2\varepsilon}\| \eta^{n} \|_{H^{-1}(\Omega)}^{2} + C_{\delta} \tau_{n+1}^{5}t_{n+1}^{-4-2\varepsilon},\label{energy_estimate_H_negative_eta}
    \end{align}
    where $ \varepsilon \in (0,2 \alpha - 3/2) $.
    When $ \delta $ is sufficiently small, $ n $ is larger than a fixed integer $ N^* $ such that $ \tau_{n+1}^{1/4}t_{n+1}^{-1/4} $ and $ \tau_{n+1}^{1/2}t_{n+1}^{-1/2} $ are sufficiently small. Then the corresponding terms on the right-hand side can be absorbed by the left-hand side. Summing up \eqref{energy_estimate_H_negative_eta} for $n \ge N^*$ and utilizing Gronwall's inequality give that
    \begin{align}
&\sum_{l=N^*}^{n}\tau_{l+1}\sum_{i=1}^{2}b_{i}\| \eta^{l,i} \|_{L^{2}(\Omega)}^{2}+\| \eta^{n+1} \|_{H^{-1}(\Omega)}^{2}\notag\\
	    \le & C \| \eta^{N^*} \|_{H^{-1}(\Omega)}^{2} + C \tau_{N^*}\sum_{i=1}^{2}\| \hat{I}_{h}\eta^{N^*,i} \|_{L^{2}(\Omega)}^{2} + C \sum_{l = N^*}^n \tau_{l+1}^5 t_{l+1}^{-4-2\varepsilon} . \label{estimate_negative_largeN_1}
	%\le &  C \| \eta^{N^*} \|_{H^{-1}(\Omega)}^{2}+ C \tau_{n+1}^{4}t_{n+1}^{-3-2\varepsilon} + C \tau^{1/(1-\alpha)},
    \end{align}
	Since $ \alpha \in (3 / 4, 1) $, we can deduce the following inequality by utilizing \eqref{time1}
	\begin{align}\label{taul1_tl1_sum}
	    \sum_{l=0}^n \tau_{l+1}^5 t_{l+1}^{-4-2\varepsilon} \le C \tau^4 \sum_{l = N^*}^n \tau_{l+1} t_{l+1}^{4\alpha -4 - 2\varepsilon} \le C \tau^4 t_{n+1}^{4\alpha - 3 - 2 \varepsilon} = C \tau_{n+1}^4 t_{n+1}^{-3-2\varepsilon}.
	\end{align}
	Substituting \eqref{taul1_tl1_sum} into \eqref{estimate_negative_largeN_1}, we obtain
	\begin{align}
	    &\sum_{l=N^*}^{n}\tau_{l+1}\sum_{i=1}^{2}b_{i}\| \eta^{l,i} \|_{L^{2}(\Omega)}^{2}+\| \eta^{n+1} \|_{H^{-1}(\Omega)}^{2}\notag\\
	    \le & C \| \eta^{N^*} \|_{H^{-1}(\Omega)}^{2} + C \tau_{N^*}\sum_{i=1}^{2}\| \hat{I}_{h}\eta^{N^*,i} \|_{L^{2}(\Omega)}^{2}  + C \tau_{n+1}^{4}t_{n+1}^{-3-2\varepsilon}. \label{estimate_negative_largeN}
	\end{align}

 To finalize the proof, we need to estimate the error $ \| \eta^n \|_{H^{-1}(\Omega)} $ for $ 1 \le n < N^* $. The inequality \eqref{energy_estimate_H_negative_eta_1} is valid for $ 0 \le n < N^* $. Summing up \eqref{energy_estimate_H_negative_eta_1} with respect to $ 0 \le n < N^* $, the following inequality is then followed from the $ L^{2}(\Omega) $ boundedness of $ \eta^{n,i} $ and $ \eta^{n} $
		\begin{align}
			& 2\sum_{l = 0}^{n} \tau_{l+1} \sum_{i=1}^{2} b_{i} \| \eta^{l,i} \|_{L^{2}(\Omega)}^{2} + \| \eta^{n+1} \|_{H^{-1}(\Omega)}^{2} \notag \\
			\le & C N^*\tau_{n+1} + C \sum_{l=0}^n \tau_{l+1}\sum_{i=1}^{2} \| \eta^l \|_{H^{-1}(\Omega)}^{2}\Big(\| u_h^{l,i} \|_{H^{1}(\Omega)}^{2} + \| \hat{I}_h u_h (t_{l,i}) \|_{H^{1}(\Omega)}^{2}\Big)\notag\\
				& + C \sum_{l=1}^{n} \| \eta^l \|_{H^{-1}(\Omega)} \quad \text{for}\quad 1 \le n < N^*.
		\end{align}
		By using \eqref{taul1_tl1_sum} and applying discrete Gronwall inequality, we obtain the result for $ 1 \le n < N^* $
		\begin{align}\label{estimate_negative_N*}
		   2 \sum_{l=0}^{n}\tau_{l+1}\sum_{i=1}^{2}b_{i}\| \eta^{l,i} \|_{L^{2}(\Omega)}^{2} + \| \eta^{n+1} \|_{H^{-1}(\Omega)}^2 \le C \tau_{n+1} ,
		\end{align}
		where the constant $ C $ is dependent on the fixed constant $ N^* $, but is independent on $ n $.
By using the property $ (2) $ specified at the beginning of section \ref{section_3_1} and choosing $ M_0 = N^* $, we have that
		\begin{align}\label{estimate_tau_smallN}
		    \tau_1 \sim \tau_2 \sim \cdots  \sim \tau_{N^*} \sim \tau^{\frac{1}{1-\alpha}},
		\end{align}
		when $ \alpha \in (\frac{3}{4},1) $ and the equivalence depends on $ N^* $. Then for $ 1 \le n \le N^* $, $ \tau_n \le t_n \le N^* \tau_n $, which implies $ t_n \sim \tau_n $. Therefore, we have
		\begin{align}\label{taun1leCtau_n1}
		    \tau_{n+1} \sim \tau_{n+1}^4 t_{n+1}^{-3 } \le C \tau^4 \tau_{n+1}^{-3 - 2\varepsilon } \quad \text{for}\quad 0 \le n \le N^* - 1.
		\end{align}
		Substituting \eqref{taun1leCtau_n1} into \eqref{estimate_negative_N*}, we obtain the desired result for $ 1 \le n < N^* $ such that
		\begin{align}\label{etali_0nN_estimate_negative}
		    2 \sum_{l=0}^{n}\tau_{l+1}\sum_{i=1}^{2}b_{i}\| \eta^{l,i} \|_{L^{2}(\Omega)}^{2} + \| \eta^{n+1} \|_{H^{-1}(\Omega)}^2 \le C \tau_{n+1}^4 t_{n+1}^{-3 - 2\varepsilon}.
		\end{align}

		Then we can substitute \eqref{estimate_negative_N*} into \eqref{estimate_negative_largeN} and obtain the following inequality for $ n \ge N^* $ by using the $ L^{2} $ boundedness of $ \eta^{n,i} $ and $ \eta^n $
		\begin{align}\label{estimate_Hminus1_largeN}
		    \sum_{l=N^*}^{n}\tau_{l+1}\sum_{i=1}^{2}b_{i}\| \eta^{l,i} \|_{L^{2}(\Omega)}^{2}+\| \eta^{n+1} \|_{H^{-1}(\Omega)}^{2} \le C \tau_{N^*} + C \tau_{n+1}^4 t_{n+1}^{-3-2\varepsilon}.
		\end{align}
		By adding up \eqref{estimate_Hminus1_largeN} and the inequality \eqref{estimate_negative_N*}, we have
		\begin{align}\label{taul1bietali_etan1_tauN}
		    \sum_{l=0}^{n}\tau_{l+1}\sum_{i=1}^{2}b_{i}\| \eta^{l,i} \|_{L^{2}(\Omega)}^{2}+\| \eta^{n+1} \|_{H^{-1}(\Omega)}^{2} \le C \tau_{N^*} + C \tau_{n+1}^4 t_{n+1}^{-3-2\varepsilon}, \quad \text{for} \quad  n \ge N^*.
		\end{align}
		Again, by using the equivalence \eqref{estimate_tau_smallN}, we can derive that
		\begin{align}\label{tau_N_estimate}
		    \tau_{N^*} \sim \tau^{\frac{1}{1-\alpha}} =  \tau^4 \tau^{\frac{4\alpha-3}{1-\alpha}} \sim \tau^4 \tau_1^{4\alpha-3} \le  \tau^4 t_{n+1}^{4\alpha-3} \sim \tau_{n+1}^4 t_{n+1}^{-3} \le C \tau_{n+1}^4 t_{n+1}^{-3 - 2\varepsilon}.
		\end{align}
		 Combining \eqref{etali_0nN_estimate_negative} and \eqref{taul1bietali_etan1_tauN}, \eqref{tau_N_estimate}, we complete the proof.
\end{proof}

Using the $ H^{-1}(\Omega) $ estimate of the errors proved in Lemma  \ref{H_negative_estimate_error}, we can derive the $ L^{2}(\Omega) $ error bound, which is present in the following theorem.
\begin{theorem}
\label{thfull}
    Under the same conditions of $ \alpha $ and $ \varepsilon $ as Lemma  \ref{H_negative_estimate_error}, the error  $ \eta^{n+1}:= u_h^{n+1}  - u_h(t_{n+1}) $ satisfies the following error bound
    \begin{align}
        \| \eta^{n+1} \|_{L^{2}(\Omega)} \le C t_{n+1}^{-2-\varepsilon}\tau_{n+1}^{2}.\label{L2_error_estimate}
    \end{align}
\end{theorem}
\begin{proof}
Let $N^*$ be the fixed integer defined in  Lemma  \ref{H_negative_estimate_error}. Then for
$ n < N^* $, \eqref{L2_error_estimate} follows directly from the $ L^{2}(\Omega) $ boundedness of $ u_{h}^{n+1} $ and $ u_h (t_{n+1}) $.
And it suffices to prove the desired result for $ n \ge N^* $.

Squaring the third relation \eqref{eqn:error-eq}  on both sides, we obtain
    \begin{align}
	    \| \eta^{n+1} \|_{L^{2}(\Omega)}^{2} = & \| \eta^{n} \|_{L^{2}(\Omega)}^{2} + 2\tau_{n+1}\sum_{i=1}^{2}b_{i}(\dot \eta^{n,i}, \eta^{n}) - 2\tau_{n+1}\sum_{i=1}^2 b_i (\eta^n, \mathcal G^{n,i}) - 2\Big(\eta^n, Q_{n+1}(\partial_t u_h)\Big) \notag\\
	    & + \tau_{n+1}^{2}\sum_{i,j=1}^{2}b_{i}b_{j}(\dot \eta^{n,i}, \dot \eta^{n,j}) - 2\tau_{n+1}^{2}\sum_{i,j=1}^{2}b_{i}b_{j}(\dot \eta^{n,i}, \mathcal G^{n,j}) \notag\\
	    & - 2\tau_{n+1}\sum_{i=1}^{2}b_{i} \Big(\dot \eta^{n,i}, Q_{n+1}(\partial_{t}u_{h})\Big) + \tau_{n+1}^{2} \sum_{i,j=1}^{2}b_{i}b_{j}(\mathcal G^{n,i}, \mathcal G^{n,j}) \notag\\
	    & + 2\tau_{n+1}\sum_{i=1}^{2}b_{i}\Big(\mathcal G^{n,i}, Q_{n+1}(\partial_{t}u_{h})\Big) + \| Q_{n+1}(\partial_{t}u_{h}) \|_{L^{2}(\Omega)}^{2}.
    \end{align}
    Similarly to the deduction of \eqref{inequality_Hnegative_energy_estimate}, by representing $ \eta^{n} $ using the second relation in \eqref{eqn:error-eq} and utilizing the algebraical stability of Gauss--Lobatto IIIC, we obtain
    \begin{align}
        \| \eta^{n+1} \|_{L^{2}(\Omega)}^{2} \le & \| \eta^{n} \|_{L^{2}(\Omega)}^{2} + 2\tau_{n+1}\sum_{i=1}^{2}b_{i}\Big(\dot \eta^{n,i}, \eta^{n,i} + \tau_{n+1}\sum_{j=1}^{2}a_{ij}\mathcal G^{n,j}+Q_{n,i}(\partial_{t}u_h)\Big) \notag\\
					       & - 2\tau_{n+1}\sum_{i=1}^2 b_i (\eta^n, \mathcal G^{n,i}) - 2\Big(\eta^n, Q_{n+1}(\partial_t u_h)\Big) + \tau_{n+1}^{2} \sum_{i,j=1}^{2}b_{i}b_{j}(\mathcal G^{n,i}, \mathcal G^{n,j}) \notag\\
	    & - 2\tau_{n+1}\sum_{i=1}^{2}b_{i} \Big(\dot\eta^{n,i}, Q_{n+1}(\partial_{t}u_{h})\Big) + \| Q_{n+1}(\partial_{t}u_{h}) \|_{L^{2}(\Omega)}^{2} \notag\\
	    & + 2\tau_{n+1}\sum_{i=1}^{2}b_{i}\Big(\mathcal G^{n,i}, Q_{n+1}(\partial_{t}u_{h})\Big)- 2\tau_{n+1}^{2}\sum_{i,j=1}^{2}b_{i}b_{j}(\dot \eta^{n,i}, \mathcal G^{n,j}) .
    \end{align}
    By the first relation in \eqref{eqn:error-eq}, we have
    \begin{equation} \label{L2_energy_estimate_error}
    \begin{aligned}
	& 2\tau_{n+1}\sum_{i=1}^{2}b_{i}\| \nabla\eta^{n,i} \|_{L^{2}(\Omega)}^{2} + \| \eta^{n+1} \|_{L^{2}(\Omega)}^{2} - \|\eta^{n} \|_{L^{2}(\Omega)}^{2} \\
	\le & C \tau_{n+1}^{2}\sum_{i,j=1}^{2}\| \dot \eta^{n,i} \|_{L^{2}(\Omega)}\| \mathcal G^{n,j} \|_{L^{2}(\Omega)} + \| Q_{n+1}(\partial_{t}u_{h}) \|_{L^{2}(\Omega)}^{2}\\
		& \;+ C  \tau_{n+1}\sum_{i=1}^{2}\|T^{n,i}\|_{H^{-1}(\Omega)}\| \eta^{n,i}\|_{H^{1}(\Omega)} + C \tau_{n+1}\sum_{i=1}^{2}\| \eta^{n} \|_{L^{2}(\Omega)} \| \mathcal G^{n,i} \|_{L^{2}(\Omega)}\\
		& \;+ C\|\eta^n\|_{L^{2}(\Omega)}\| Q_{n+1}(\partial_t u_h)\|_{L^{2}(\Omega)} + C\tau_{n+1}^{2} \sum_{i=1}^{2}\|\mathcal G^{n,i}\|_{L^{2}(\Omega)}^{2} \\
		& \;+C  \tau_{n+1}\sum_{i=1}^{2} \Big(\| \eta^{n,i}\|_{H^{1}(\Omega)}+ \| T^{n,i} \|_{H^{-1}(\Omega)}\Big)\Big(\| Q_{n,i}(\partial_{t}u_h)\|_{H^{1}(\Omega)} + \| Q_{n+1}(\partial_{t}u_{h})\|_{H^{1}(\Omega)}\Big) %. \label{L2_energy_estimate_error}
    \end{aligned}
    \end{equation}
    Suppose $ \mathcal O^{-1} = (r_{ij}) $. By the second relation in \eqref{eqn:error-eq}, %\eqref{eqn:error_equation_b},
    we have
    \begin{align*}
	    \dot \eta^{n,i} = \tau_{n+1}^{-1}\sum_{j=1}^{2}r_{ij}(\eta^{n,j} - \eta^{n})+\mathcal G^{n,i} + \tau_{n+1}^{-1}\sum_{j=1}^{2}r_{ij}Q_{n,j}(\partial_{t}u_{h}).
    \end{align*}
    Hence, we can derive the estimate of $ \| \dot \eta^{n,i} \|_{L^{2}(\Omega)} $ such that
    \begin{align*}
        \| \dot \eta^{n,i} \|_{L^{2}(\Omega)} \le C \tau_{n+1}^{-1}\sum_{j=1}^{2}\Big(\| \eta^{n,j} \|_{L^{2}(\Omega)} +\| \eta^{n} \|_{L^{2}(\Omega)} + \| Q_{n,j}(\partial_{t}u_{h}) \|_{L^{2}(\Omega)}\Big) + \| \mathcal G^{n,i} \|_{L^{2}(\Omega)}.
    \end{align*}
    %By testing \eqref{error_equation_b} with $ \eta^{n,i} $, we have
    %\begin{align}
%	    \| \eta^{n,i} \|_{L^{2}(\Omega)}^{2} \le & C \| \eta^{n} \|_{H^{-1}(\Omega)} \| \eta^{n,i} \|_{H^{1}(\Omega)} + C \tau_{n+1} \sum_{j=1}^{2}\| \dot \eta^{n,j} \|_{H^{-1}(\Omega)}\| \eta^{n,i} \|_{H^{1}(\Omega)} \notag\\
%						     & + \tau_{n+1} \sum_{j=1}^{2}\| \mathcal G^{n,j} \|_{H^{-1}(\Omega)} \| \eta^{n,i} \|_{H^{1}(\Omega)} + \| Q_{n,i}(\partial_{t}u_{h}) \|_{H^{-1}(\Omega)} \| \eta^{n,i} \|_{H^{1}(\Omega)} \notag\\
%	\le & \delta  \sum_{j=1}^{2}b_{j}\| \eta^{n,j} \|_{H^{1}(\Omega)}^{2} + C \delta^{-1}\| \eta^{n} \|_{H^{-1}(\Omega)}^{2}  + C \tau_{n+1}^{4}t_{n+1}^{-3} \notag\\
%	    & + C \delta^{-1}\tau_{n+1}^{2}\sum_{j=1}^{2}\| T^{n,j} \|_{H^{-1}(\Omega)}^{2}\label{L2_estimate_error_ni}
%    \end{align}
Substituting this estimate into \eqref{L2_energy_estimate_error}, we have
\begin{align}
     &2 \tau_{n+1}\sum_{i=1}^{2}b_{i}\| \nabla\eta^{n,i} \|_{L^{2}(\Omega)}^{2} + \| \eta^{n+1} \|_{L^{2}(\Omega)}^{2} - \|\eta^{n} \|_{L^{2}(\Omega)}^{2} \notag\\
	\le & \delta \tau_{n+1}\sum_{i=1}^{2}b_{i}\| \eta^{n,i} \|_{H^{1}(\Omega)}^{2} + C \tau_{n+1} t_{n+1}^{-1} \Big( \| \eta^{n} \|_{L^{2}(\Omega)}^{2} + \sum_{i=1}^{2}\| \eta^{n,i} \|_{L^{2}(\Omega)}^{2}
 \Big)  \notag\\
					       & + C\| Q_{n+1}( \partial_{t}u_{h}) \|_{L^{2}(\Omega)}^2 + C \delta^{-1}\tau_{n+1}\sum_{i=1}^{2} \Big(\| Q_{n,i}(\partial_{t}u_h)\|_{H^{1}(\Omega)} + \| Q_{n+1}(\partial_{t}u_{h})\|_{H^{1}(\Omega)}\Big)^{2} \notag\\
					       & + C \tau_{n+1}t_{n+1}\sum_{i=1}^{2}\| \mathcal G^{n,i} \|_{L^{2}(\Omega)}^{2}+ C \delta^{-1}\tau_{n+1}\sum_{i=1}^{2}\|T^{n,i}\|_{H^{-1}(\Omega)}^{2} .
\end{align}
By using estimates \eqref{T_nj_H_negative_estimate}, \eqref{eqn:kutta-err},  \eqref{G_estimate_H_negative} and \eqref{G_estimate_L2}, when $ \delta $ is sufficiently small, we obtain
\begin{align}
    & \tau_{n+1}\sum_{i=1}^{2}b_{i}\| \nabla\eta^{n,i} \|_{L^{2}(\Omega)}^{2} + \| \eta^{n+1} \|_{L^{2}(\Omega)}^{2} - \|\eta^{n} \|_{L^{2}(\Omega)}^{2}\notag\\
	\le & C \tau_{n+1}t_{n+1}^{-1}\Big(\| \eta^{n} \|_{L^{2}(\Omega)}^{2} + \sum_{i=1}^{2}\| \eta^{n,i} \|_{L^{2}(\Omega)}^{2} + \sum_{i=1}^{2}\| \hat{I}_{h}\eta^{n,i} \|_{L^{2}(\Omega)}^{2}\Big) + C \tau_{n+1}^{5}t_{n+1}^{-5}.\label{L2_energy_estimate_error_eta_middle}
\end{align}
 Multiplying $ t_{n+1} $ on both sides of \eqref{L2_energy_estimate_error_eta_middle}, we have
\begin{align*}
	& t_{n+1} \tau_{n+1} \sum_{i=1}^2 b_i \| \nabla \eta^{n,i} \|_{L^{2}(\Omega)}^{2} + t_{n+1} \| \eta^{n+1} \|_{L^{2}(\Omega)}^{2} - t_n \| \eta^n \|_{L^{2}(\Omega)}^{2} \\
	\le & \tau_{n+1}\| \eta^n \|_{L^{2}(\Omega)}^{2} + C \tau_{n+1} \Big(\| \eta^n \|_{L^{2}(\Omega)}^{2} + \sum_{i=1}^2 \| \eta^{n,i} \|_{L^{2}(\Omega)}^{2} + \sum_{i=1}^2 \| \hat{I}_h \eta^{n,i} \|_{L^{2}(\Omega)}^{2}\Big) + C \tau_{n+1}^5 t_{n+1}^{-4}.
\end{align*}
By summing up the above inequality with respect to $ n $ from $ 0 $ to $ m $ and following the proof similarly to \eqref{taul1_tl1_sum}, we have
\begin{align*}
	& \sum_{n=0}^m t_{n+1}\tau_{n+1} \sum_{i=1}^2 b_i \| \nabla \eta^{n,i} \|_{L^{2}(\Omega)}^{2} + t_{m+1}\| \eta^{m+1} \|_{L^{2}(\Omega)}^{2} \\
	\le & C \sum_{n= 0}^m \tau_{n+1}\Big(\| \eta^n \|_{L^{2}(\Omega)}^{2} + \sum_{i=1}^2 \| \eta^{n,i} \|_{L^{2}(\Omega)}^{2} + \sum_{i=1}^2 \| \hat{I}_h \eta^{n,i} \|_{L^{2}(\Omega)}^{2} \Big) + C \tau_{m+1}^4 t_{m+1}^{-3}.
\end{align*}
Then the desired result follows from Lemma  \ref{H_negative_estimate_error}.
\end{proof}

\section{Numerical examples}%
\label{sec:numerics}
In this section we present numerical examples to support the theoretical results in Theorem \ref{thsemi} and Theorem \ref{thfull}. All examples concern the incompressible NS problem in \eqref{system-1}.
%\begin{equation}\label{system-22}
%    \left\{\begin{array}{rclll}
%        \partial_t u + u\cdot \nabla u - \nu\Delta u + \nabla p & \!\!\!\!= & \!\!\!\!0 & \mbox{in }& \Omega\times (0,T],\\
%	\nabla \cdot u  &\!\!\!\! = & \!\!\! \! 0 & \mbox{in} & \Omega \times (0,T],\\
%        u & \!\!\!\! = & \!\!\!\! 0 & \mbox{on} & \partial\Omega\times(0,T],\\
%        u & \!\!\!\! = & \!\!\!\! u_0 & \mbox{on} & \Omega\times\{0\}.
%    \end{array}\right.
%\end{equation}

\begin{Example}[The merging of two co-rotating vortices]\upshape
	In this example, we investigate the simulation of the merging of two co-rotating Lamb-Oseen vortices within a two-dimensional domain $ \Omega = (-\pi,\pi)\times (-\pi,\pi) $. The initial value of the standard Lamb-Oseen vortex \cite{ferreira2005,Josserand2007,mao2012,orlandi2007} is inherently a function in $ L^{p}(\Omega)^{2} $ for $ p < 2 $. To ensure that the initial value belongs to $ L^{2}(\Omega)^{2} $ but not to $ H^{\varepsilon}(\Omega)^{2} $, we make a slight modification to the data by selecting the initial value $ u_{0} = u_{1} + u_{2} $ and
    \begin{align*}
	    u_{1} = \Big(-\frac{ y\Gamma}{2\pi r_{1}^{2-\varepsilon}}, \frac{ (x+0.5)\Gamma}{2\pi r_{1}^{2-\varepsilon}}\Big),\quad
	    u_{2} = \Big(-\frac{ y\Gamma}{2 \pi r_{2}^{2-\varepsilon}}, \frac{ (x-0.5)\Gamma}{2\pi r_{2}^{2-\varepsilon}} \Big),
    \end{align*}
    with $ r_{1} = \sqrt{(x+0.5)^{2} + y^{2}} $, $ r_{2} = \sqrt{(x - 0.5)^{2} + y^{2}} $, and $ \varepsilon = 0.1 $. Here, $ \Gamma $ denotes the circulation, set to $ 2\pi $ for this test. The viscosity $ \nu $ is chosen as $ 0.1 $. We choose the domain $ \Omega $ so large that we may assume that $ u $ satisfies $ 0 $ Dirichlet boundary condition. %Since the trace of the solution $ u $ for this practical problem is generally unknown, we adopt the Neumann-type boundary condition:
    %\begin{align*}
	%    \nu \frac{\partial u}{\partial n} - p n = 0.
    %\end{align*}
%The theoretical results in Theorem \ref{thsemi} and Theorem \ref{thfull} do not cover this example as we consider the Neumann-type boundary condition here instead of the no-slip boundary condition.

We test temporal convergence at $ T = 0.1 $ using graded stepsizes \eqref{time1} with $ \alpha = 0$ (uniform) and $\alpha = 0.76$. The reference solution $u^N_{h,\rm ref}$ is computed with $\tau= 1/1024$. Temporal errors $\|u^N_h-u^N_{h,\rm ref}\|_{L^2(\Omega)}$ in Figure \ref{figure_vort} (a) for $\tau = 1/32, 1/64, 1 / 128,1/256$ (spatial errors negligible for sufficiently small $ h $) show second-order convergence for $\alpha = 0.76$ (consistent with Theorem \ref{thfull}) but irregular convergence for $\alpha = 0$, justifying the necessity of graded stepsizes in \eqref{time1}.

%We test the convergence rates at time $ T = 0.1 $. We solve the NS equations by the proposed method using the graded stepsizes in \eqref{time1}, with $ \alpha = 0 $ (the uniform stepsizes), and $\alpha=0.76$. The temporal discretization errors $\|u^N_h-u^N_{h,\rm ref}\|_{L^2(\Omega)}$ and corresponding convergence rates are presented in Figure \ref{figure_vort} (a) for stepsizes $ \tau = 1/ 32, 1/64, 1 / 128, 1 / 256 $ with a sufficiently small mesh size that ensures the errors from spatial discretization are negligible. The reference solution $u^N_{h,\rm ref}$ is computed with $\tau= 1/1024$. We observe second-order convergence in time for the case $ \alpha = 0.76 $, which aligns with the results proved in Theorem \ref{thfull}. In the case of $ \alpha = 0 $, the convergence order is irregular, which shows the necessity of using the graded stepsizes in \eqref{time1}.

In Figure \ref{figure_vort} (b), we present the spatial discretization errors $\|u^N_h-u^N_{h,\rm ref}\|_{L^2(\Omega)}$ and convergence rates for mesh sizes $ h = \pi / 8, \pi/16, \pi / 32, \pi /64 $ with a sufficiently small temporal stepsize that ensures the errors from temporal discretization are negligible.
The reference solution $u^N_{h,\rm ref}$ is chosen to be the numerical solution with mesh size $h= \pi/128$. We use $ P_{2} $--$ P_1 $ Taylor--Hood elements and observe that the convergence in space is second order. This aligns with the theoretical result proved in Theorem \ref{thsemi} and shows the sharpness of the convergence rate in space.

\begin{figure}[htp]
    \centering
    \subfigure[$ L^2 $ error of $ u $ from temporal discretization]{\includegraphics[trim = .1cm .1cm .1cm .1cm, clip=true,width=0.4\textwidth,height=0.35\textwidth]{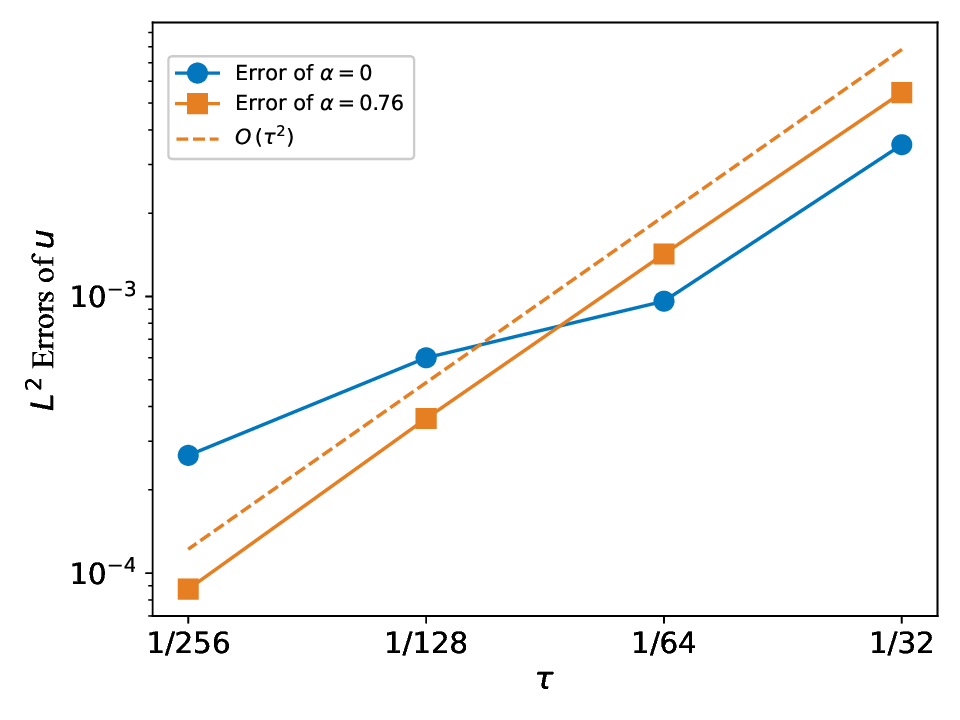}}
\qquad\quad
    \subfigure[$ L^2 $ error of $ u $ from spatial discretization]{\includegraphics[trim = .1cm .1cm .1cm .1cm, clip=true,width=0.4\textwidth,height=0.35\textwidth]{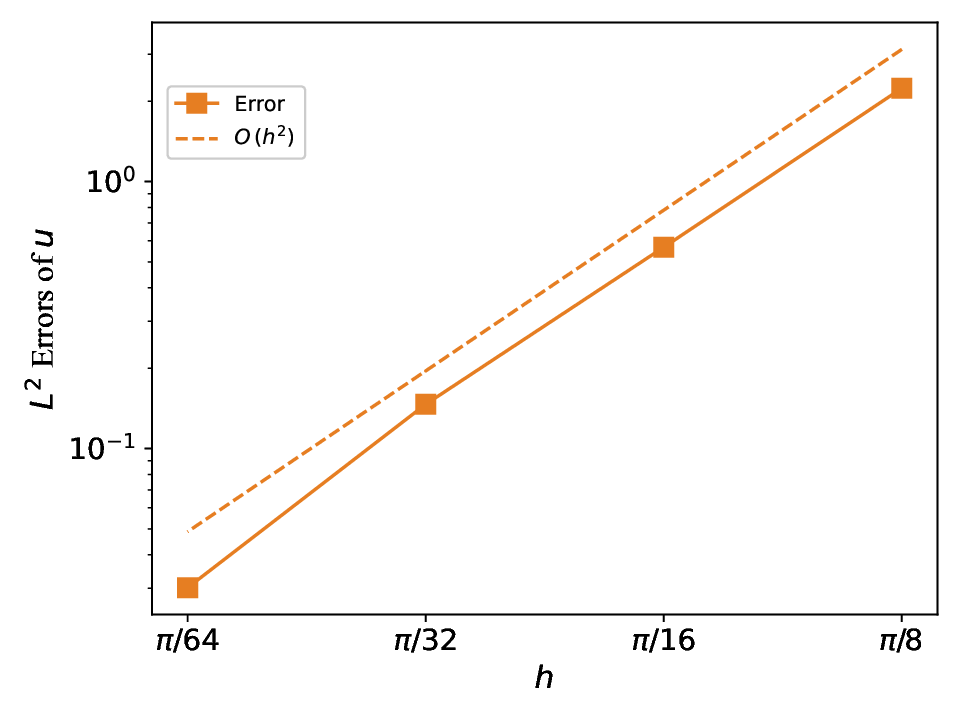}}
    \caption{$ L^{2} $ errors of $ u $
}\label{figure_vort}
\end{figure}

The evolution of the velocity $ u $ for the co-rotating vortices is illustrated at various time instances, specifically at $ t = 0.1, 0.3, 0.5, 0.7, 1.0, 2.0 $. These visualizations are depicted in Figure \ref{figure1} (a)--(f) with mesh size $ h = \pi/50 $ and time stepsize $ \tau = 0.005 $. The parameter $ \alpha $ is  chosen as $ 0.76 $. The numerical simulation indicates a gradual merging of the two co-rotating vortices over time. Notably, at $ t = 2.0 $, the vortices have completely merged into a single vortex, as shown in Figure \ref{figure1} (f).
\begin{figure}[htp]
    \centering
    \subfigure[The vorticity at time $ t = 0.01 $]{\includegraphics[trim = .1cm .1cm .1cm .1cm, clip=true,width=0.38\textwidth,height=0.31\textwidth]{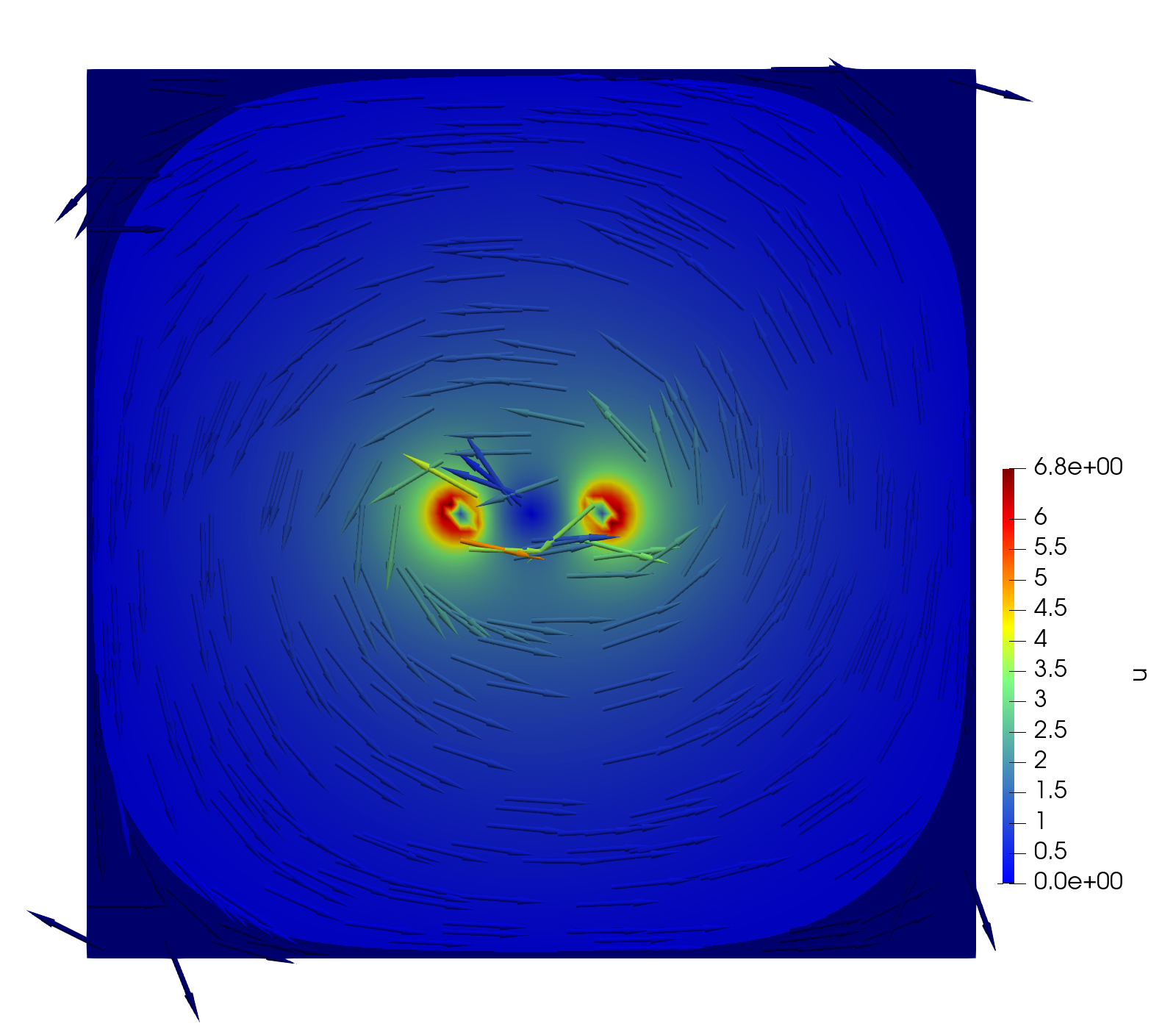}}
    \subfigure[The vorticity at time $ t = 0.1 $]{\includegraphics[trim = .1cm .1cm .1cm .1cm, clip=true,width=0.38\textwidth,height=0.31\textwidth]{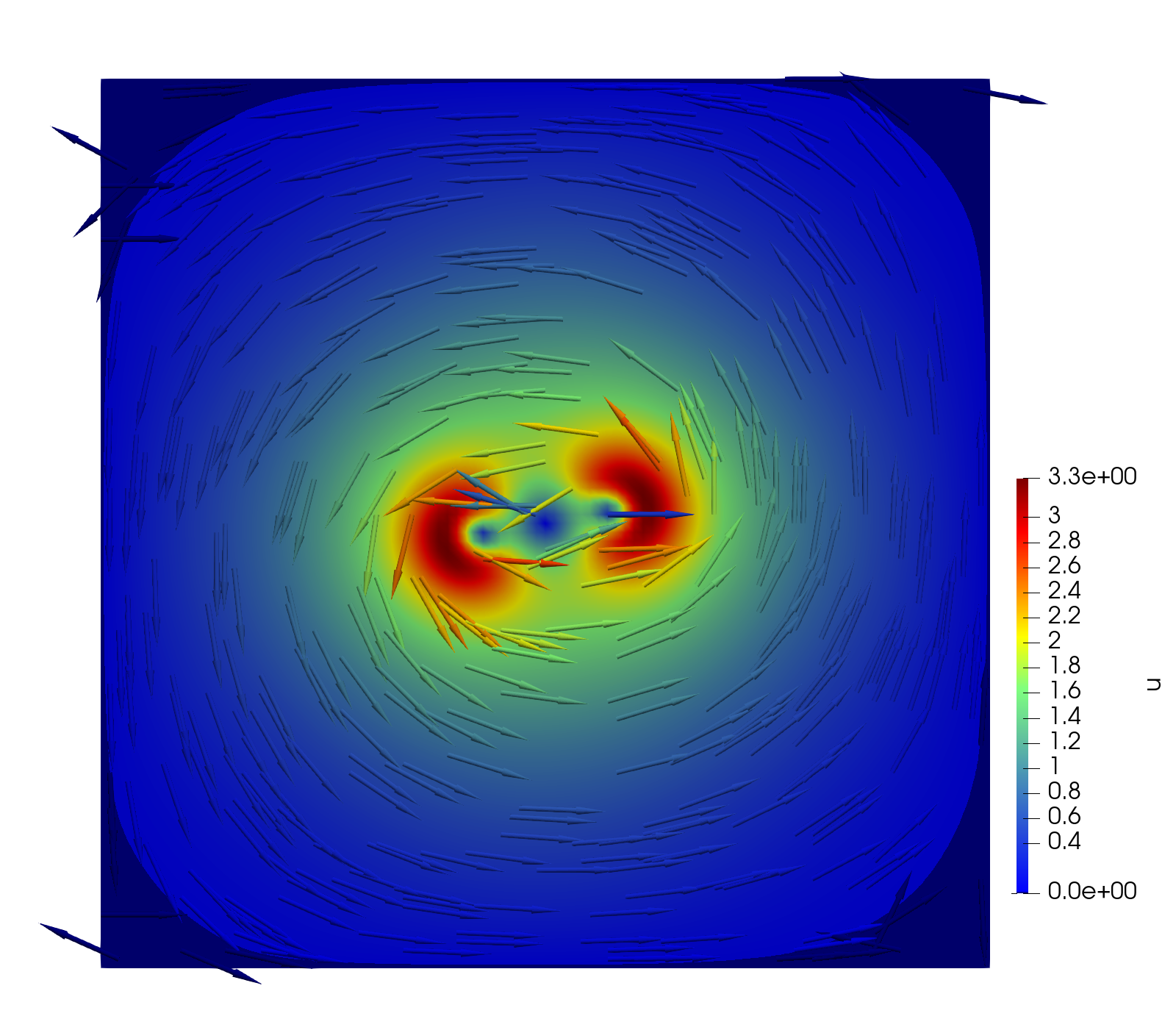}}
    \subfigure[The vorticity at time $ t = 0.3 $]{\includegraphics[trim = .1cm .1cm .1cm .1cm, clip=true,width=0.38\textwidth,height=0.31\textwidth]{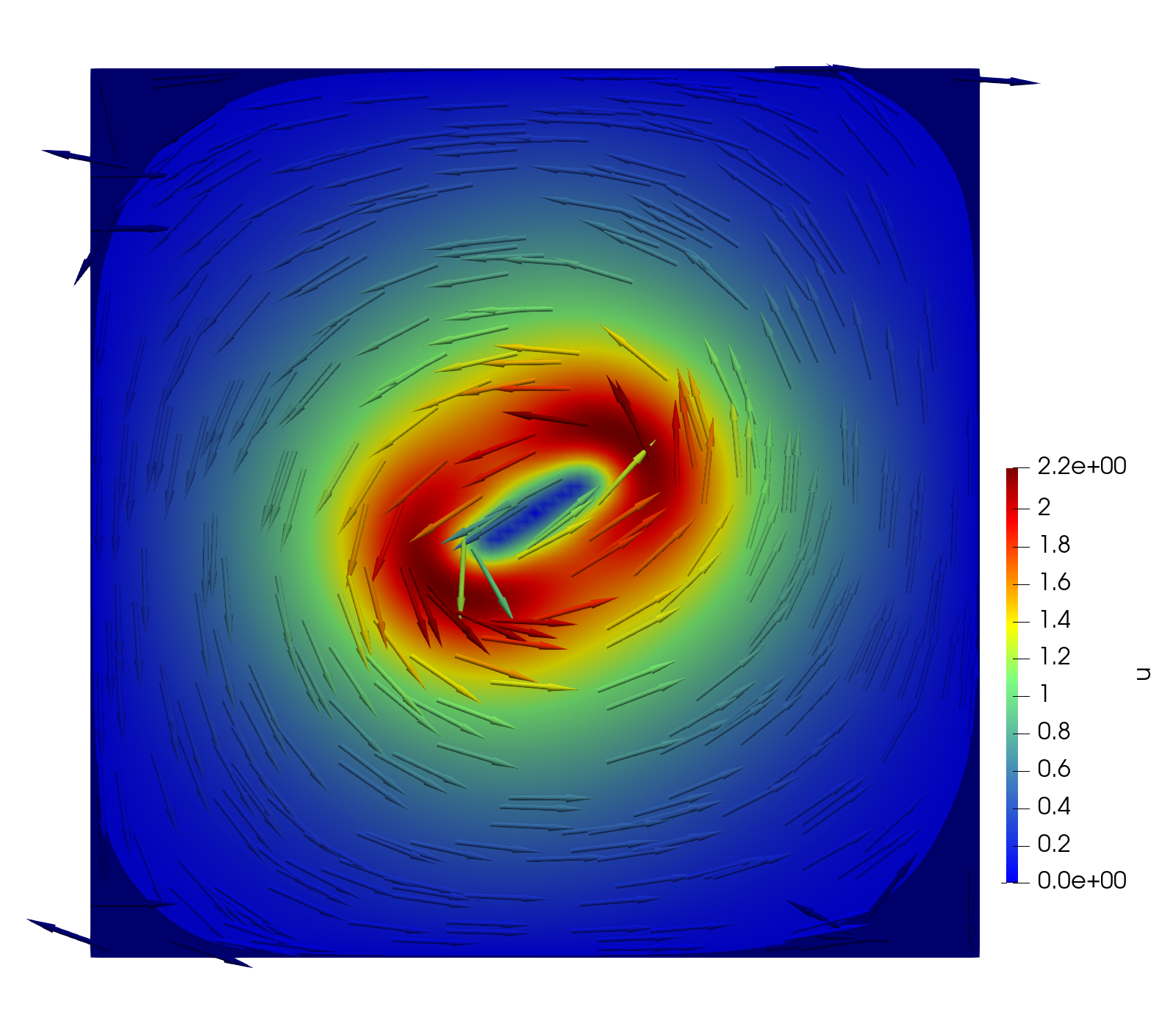}}
    \subfigure[The vorticity at time $ t = 0.5 $]{\includegraphics[trim = .1cm .1cm .1cm .1cm, clip=true,width=0.38\textwidth,height=0.31\textwidth]{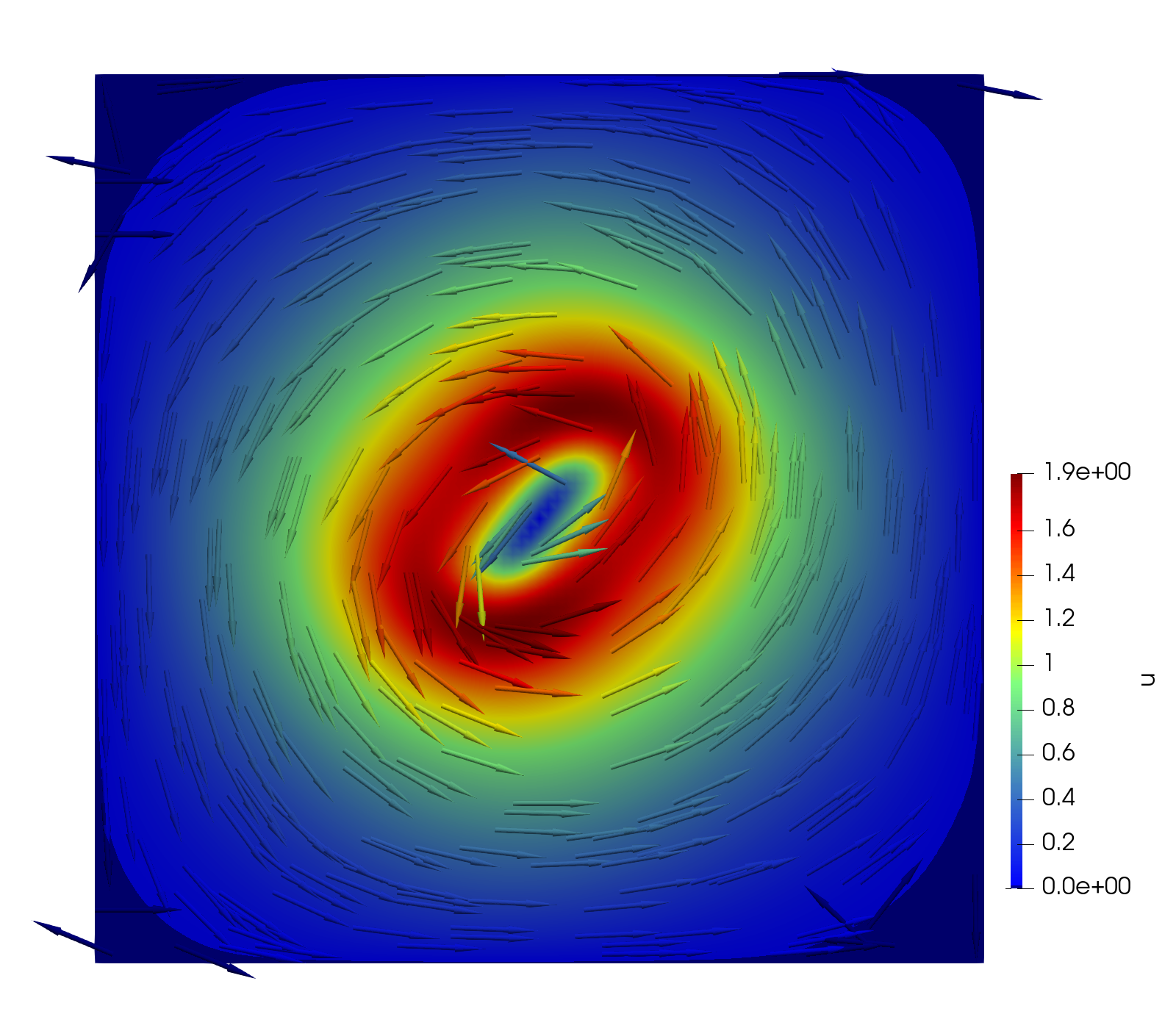}}
    \subfigure[The vorticity at time $ t = 1.0 $]{\includegraphics[trim = .1cm .1cm .1cm .1cm, clip=true,width=0.38\textwidth,height=0.31\textwidth]{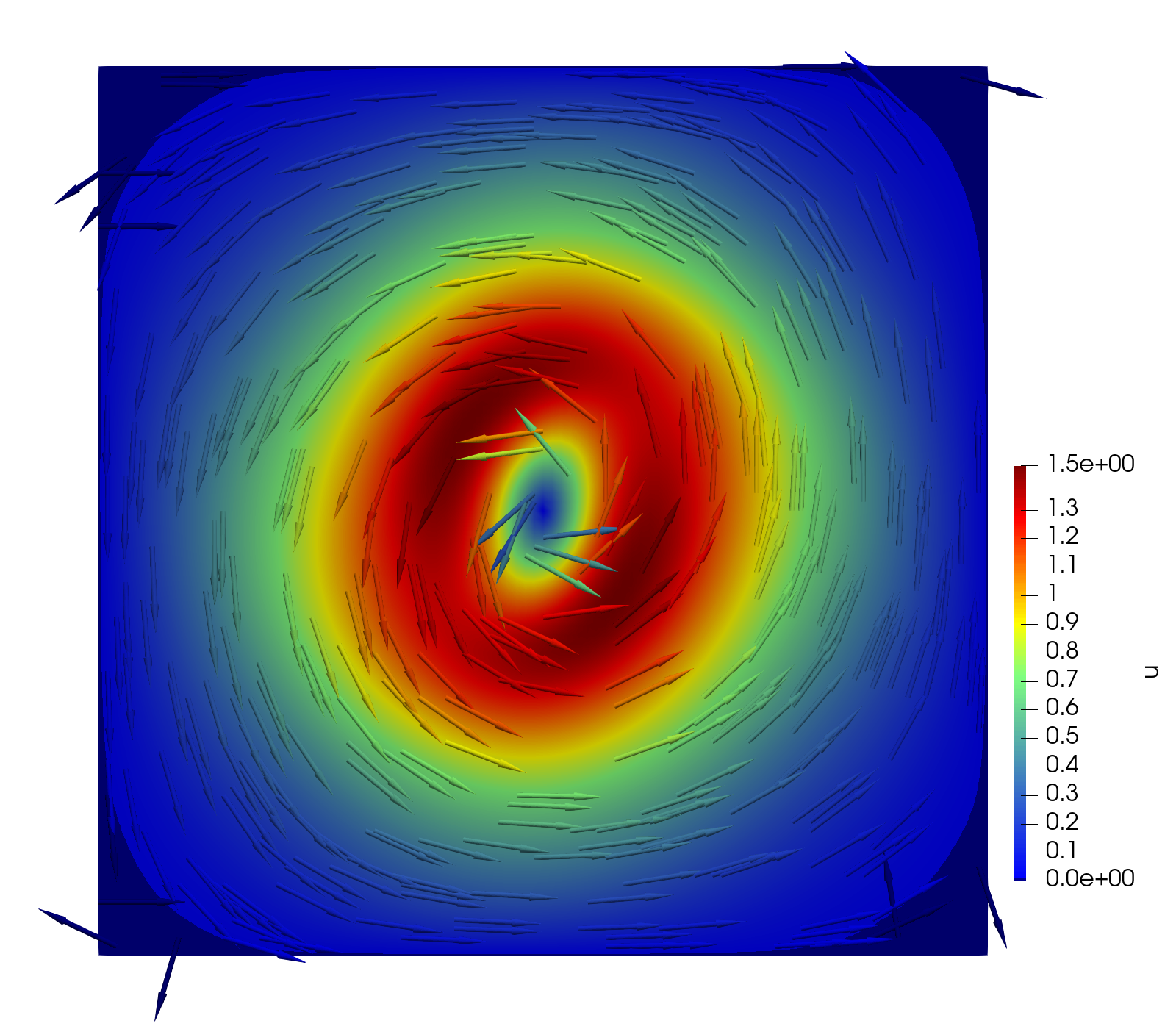}}
    \subfigure[The vorticity at time $ t = 2.0 $]{\includegraphics[trim = .1cm .1cm .1cm .1cm, clip=true,width=0.38\textwidth,height=0.31\textwidth]{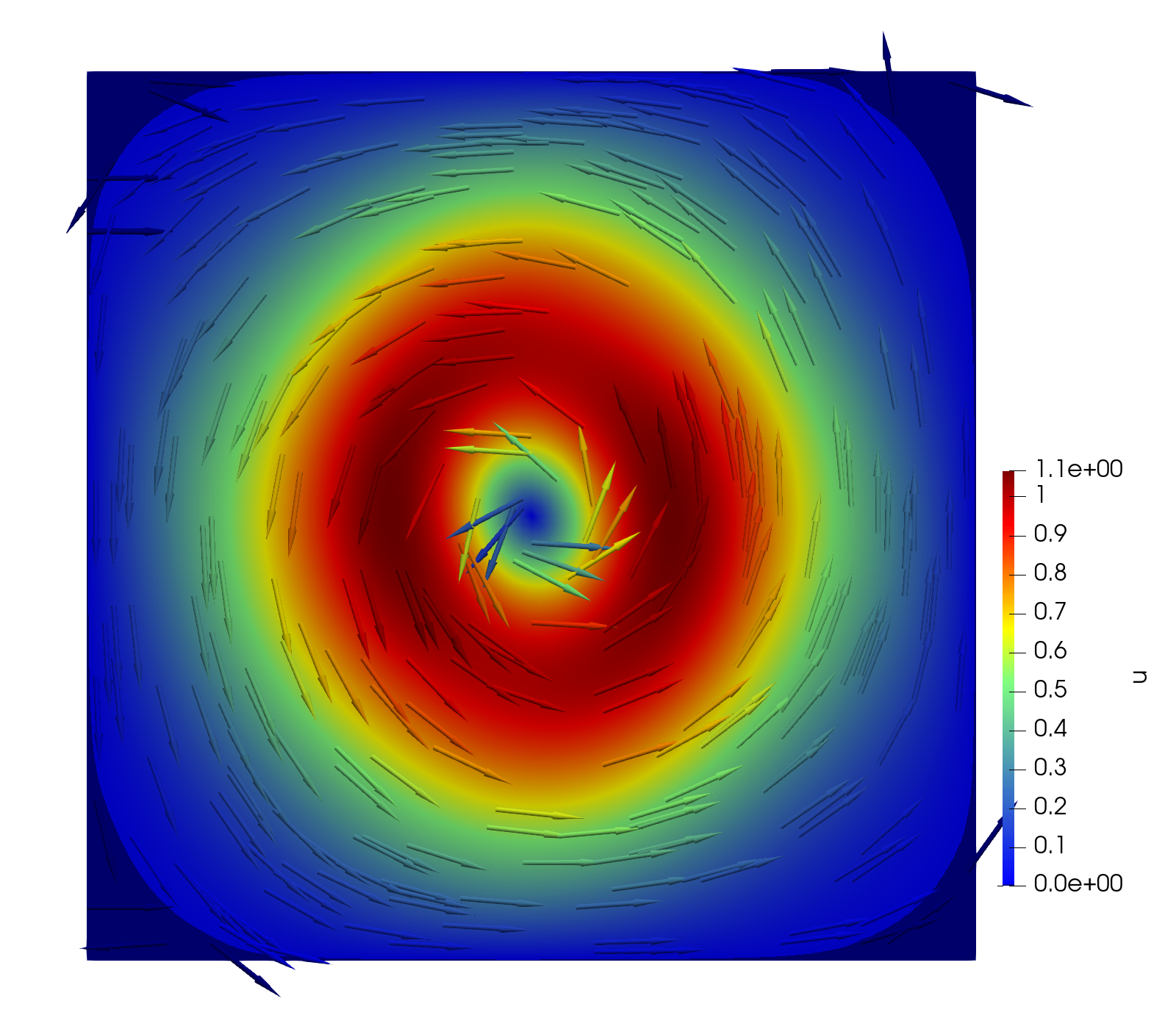}}
\caption{Isocontours of the velocity $ u $
}\label{figure1}
\end{figure}
\end{Example}

\begin{Example}[Piecewise constant initial value]\upshape
In this example, we present numerical simulation of the Navier--Stokes equations with a piecewise constant initial value in $ \Omega = (- \pi, \pi)\times (-\pi,\pi) $. The viscosity $ \nu $ is chosen to be $ 0.1 $. The initial value $ u_0 $ takes value $ (10, 0) $ when $ y >0 $, and $ (-10,0) $ when $ y < 0 $. This initial value is in $\dot L^2(\Omega)\cap H^{\frac{1}{2}-\varepsilon}(\Omega)^2$ for any $\varepsilon\in(0,\frac12)$ but not in $H^{\frac{1}{2}}(\Omega)^2 $.

We test temporal convergence at $ T = 1 $ using graded stepsizes \eqref{time1} with $\alpha = 0.76$. The reference solution $u^N_{h,\rm ref}$ is computed with $\tau= 1/1024$. Temporal errors $\|u^N_h-u^N_{h,\rm ref}\|_{L^2(\Omega)}$ in Figure \ref{figure_Rie} (a) for $\tau = 1/32, 1/64, 1 / 128,1/256$ (spatial errors negligible for sufficiently small $ h $) show second-order convergence for $\alpha = 0.76$, which is consistent with Theorem \ref{thfull}.

%We test the convergence rates at time $ T = 1 $. We solve the problem using the graded stepsizes in \eqref{time1}, with $ \alpha = 0 $ (the uniform stepsizes), and $\alpha=0.76$. The temporal discretization errors $\|u^N_h-u^N_{h,\rm ref}\|_{L^2(\Omega)}$ and corresponding convergence rates are presented in Figure \ref{figure_Rie} (a) for stepsizes $ \tau = 1/ 32, 1/64, 1 / 128, 1 / 256 $ with a sufficiently small mesh size that ensures the errors from spatial discretization are negligible. The reference solution $u^N_{h,\rm ref}$ is computed with $\tau= 1/1024$. We observe second-order convergence in time for the case $ \alpha = 0.76 $, which aligns with the results proved in Theorem \ref{thfull}. In the case of $ \alpha = 0 $, the convergence order is irregular, which shows the necessity of using the graded stepsizes in \eqref{time1}.

In Figure \ref{figure_Rie} (b), we present the spatial discretization errors $\|u^N_h-u^N_{h,\rm ref}\|_{L^2(\Omega)}$ and convergence rates for mesh sizes $ h = 2 \pi / 15, 2 \pi/30, 2 \pi / 60, 2\pi /120 $ with a sufficiently small temporal stepsize that ensures the errors from temporal discretization are negligible.
The reference solution $u^N_{h,\rm ref}$ is chosen to be the numerical solution with mesh size $h= 2 \pi/240$. We use $ P_{2} $--$ P_1 $ Taylor--Hood elements and observe that the convergence in space is second order. This aligns with the theoretical result proved in Theorem \ref{thsemi} and shows the sharpness of the convergence rate in space.

\begin{figure}[htp]
    \centering
    \subfigure[$ L^2 $ error of $ u $ from temporal discretization]{\includegraphics[trim = .1cm .1cm .1cm .1cm, clip=true,width=0.4\textwidth,height=0.35\textwidth]{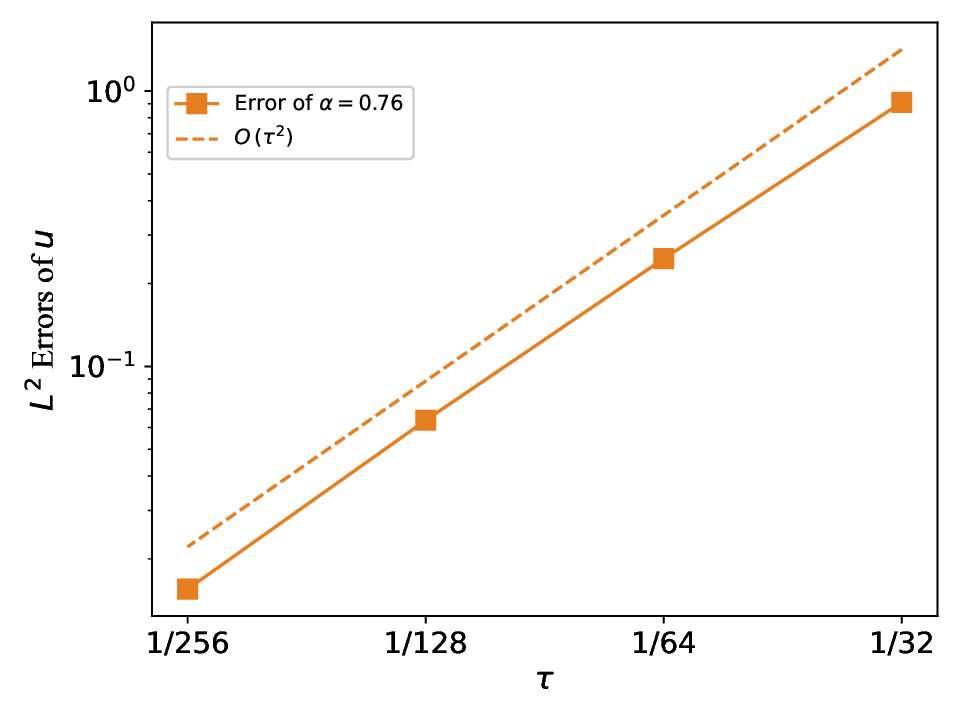}}
\qquad\quad
    \subfigure[$ L^2 $ error of $ u $ from spatial discretization]{\includegraphics[trim = .1cm .1cm .1cm .1cm, clip=true,width=0.4\textwidth,height=0.35\textwidth]{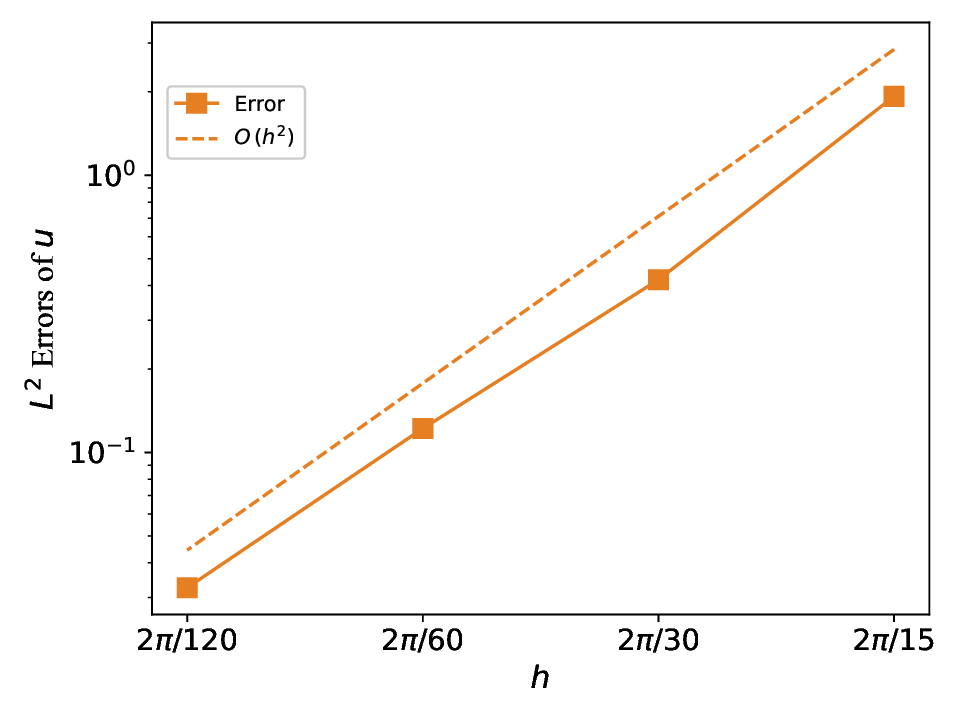}}
    \caption{$ L^{2} $ errors of $ u $
}\label{figure_Rie}
\end{figure}

%The theoretical results in Theorem \ref{thsemi} and Theorem \ref{thfull}, as well as the numerical results in Example \ref{ex4.1}, indicate the convergence of numerical solutions for this example.

The evolution of the velocity field $u$ computed by the proposed method is illustrated at various time instances, specifically at $ t = 0,0.02,0.1,0.5,1.0,2.0 $. These visualizations are depicted in Figure \ref{figure2} with mesh size $ h = 0.06 $ and time stepsize $\tau = 0.01$. The parameter $ \alpha $ is chosen to be $ 0.76 $. Notably, the discontinuous initial velocity field gradually becomes smooth as time evolves.

\begin{figure}[htp]
    \centering
    \subfigure[The velocity at time $ t = 0 $]{\includegraphics[trim = .1cm .1cm .1cm .1cm, clip=true,width=0.38\textwidth,height=0.31\textwidth]{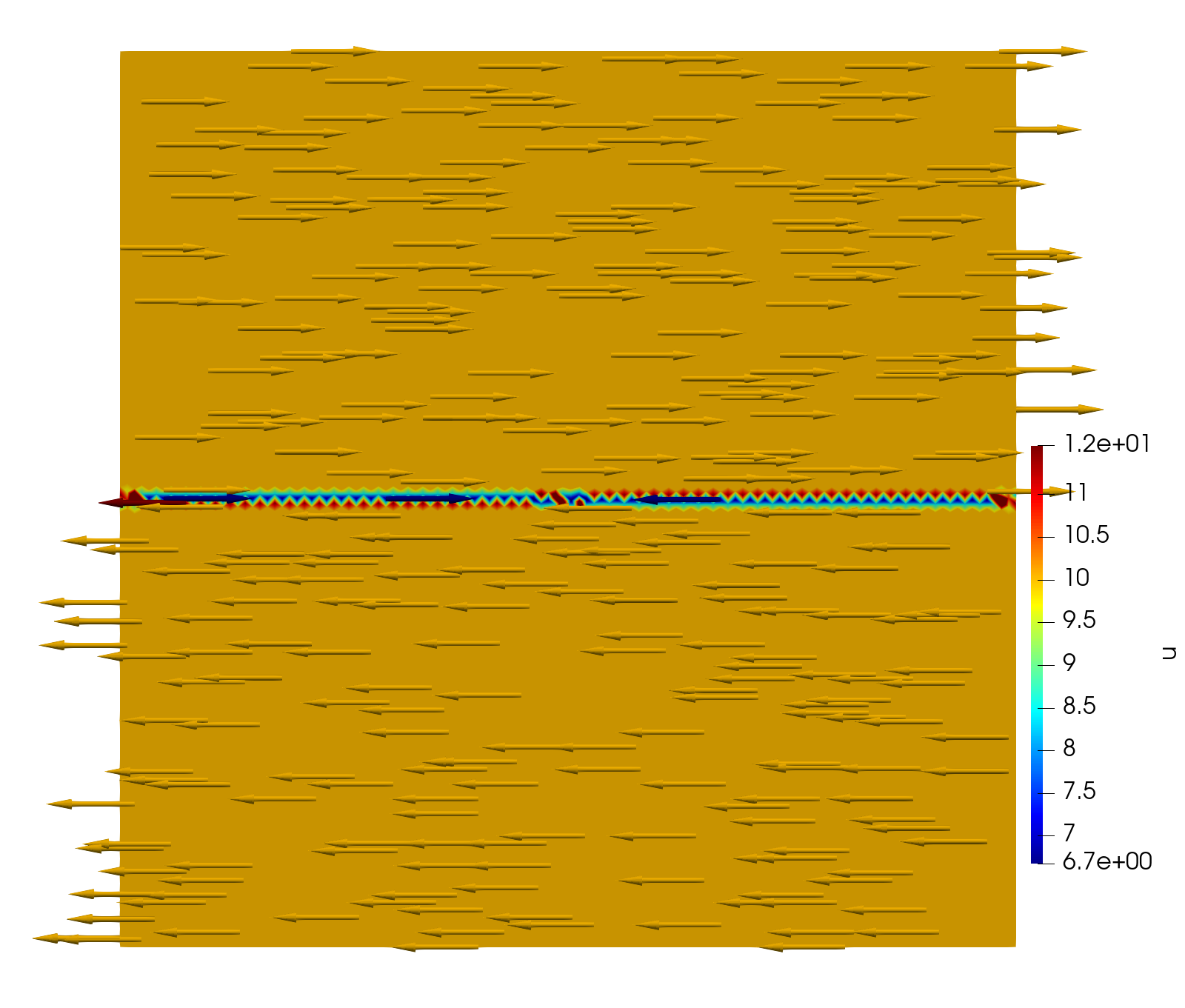}}
    \subfigure[The velocity at time $ t = 0.02 $]{\includegraphics[trim = .1cm .1cm .1cm .1cm, clip=true,width=0.38\textwidth,height=0.31\textwidth]{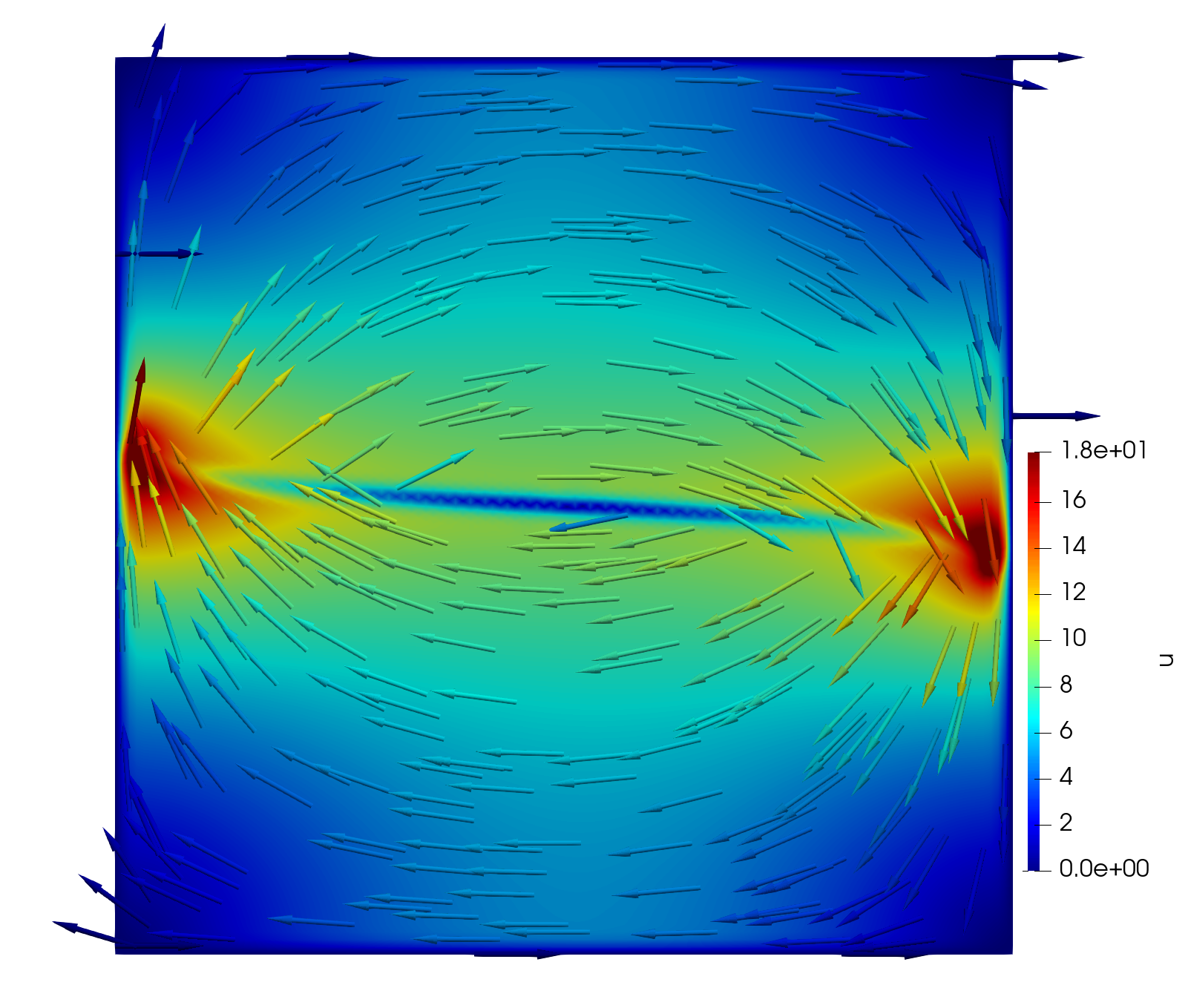}}
    \subfigure[The velocity at time $ t = 0.1 $]{\includegraphics[trim = .1cm .1cm .1cm .1cm, clip=true,width=0.38\textwidth,height=0.31\textwidth]{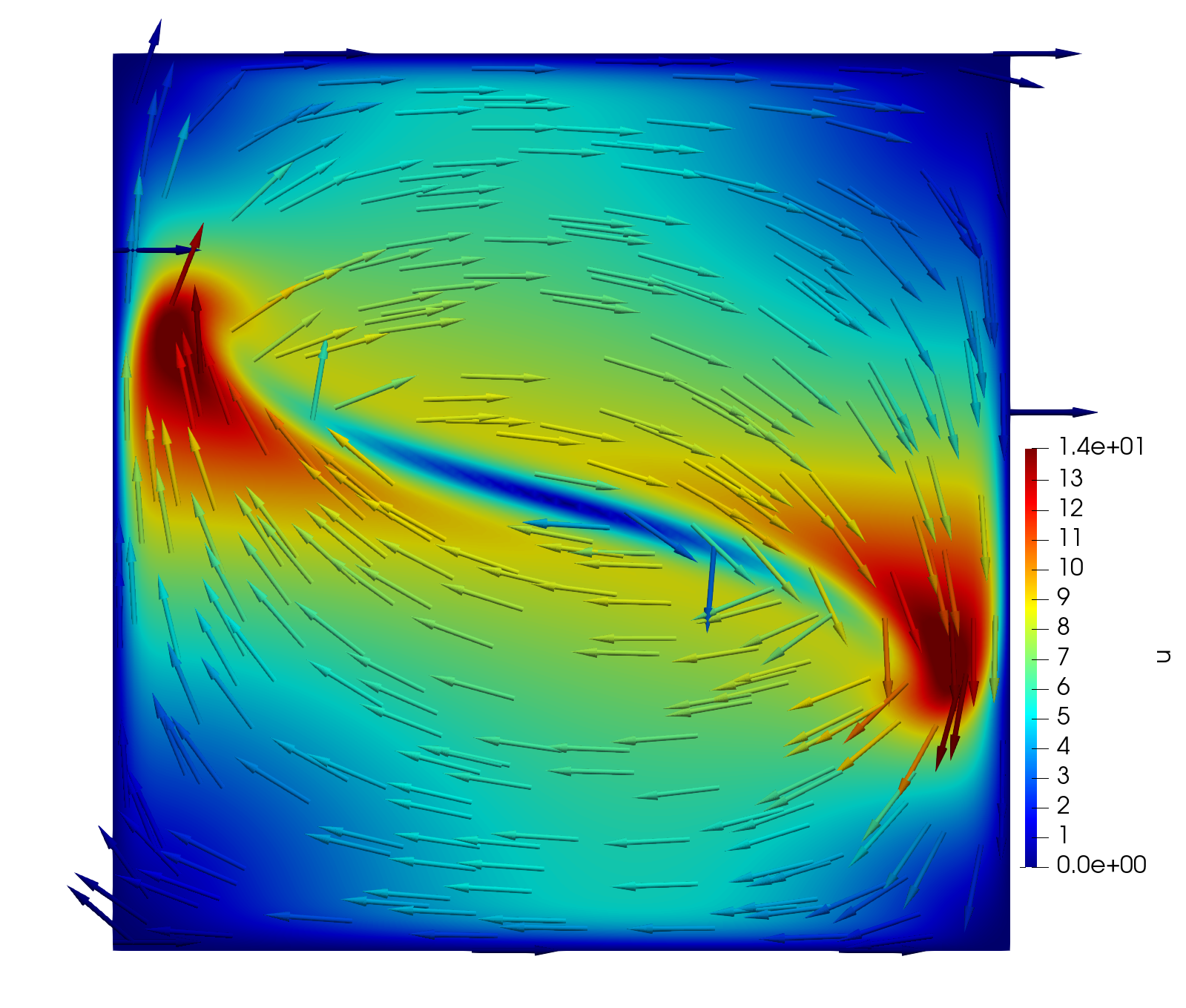}}
    \subfigure[The velocity at time $ t = 0.5 $]{\includegraphics[trim = .1cm .1cm .1cm .1cm, clip=true,width=0.38\textwidth,height=0.31\textwidth]{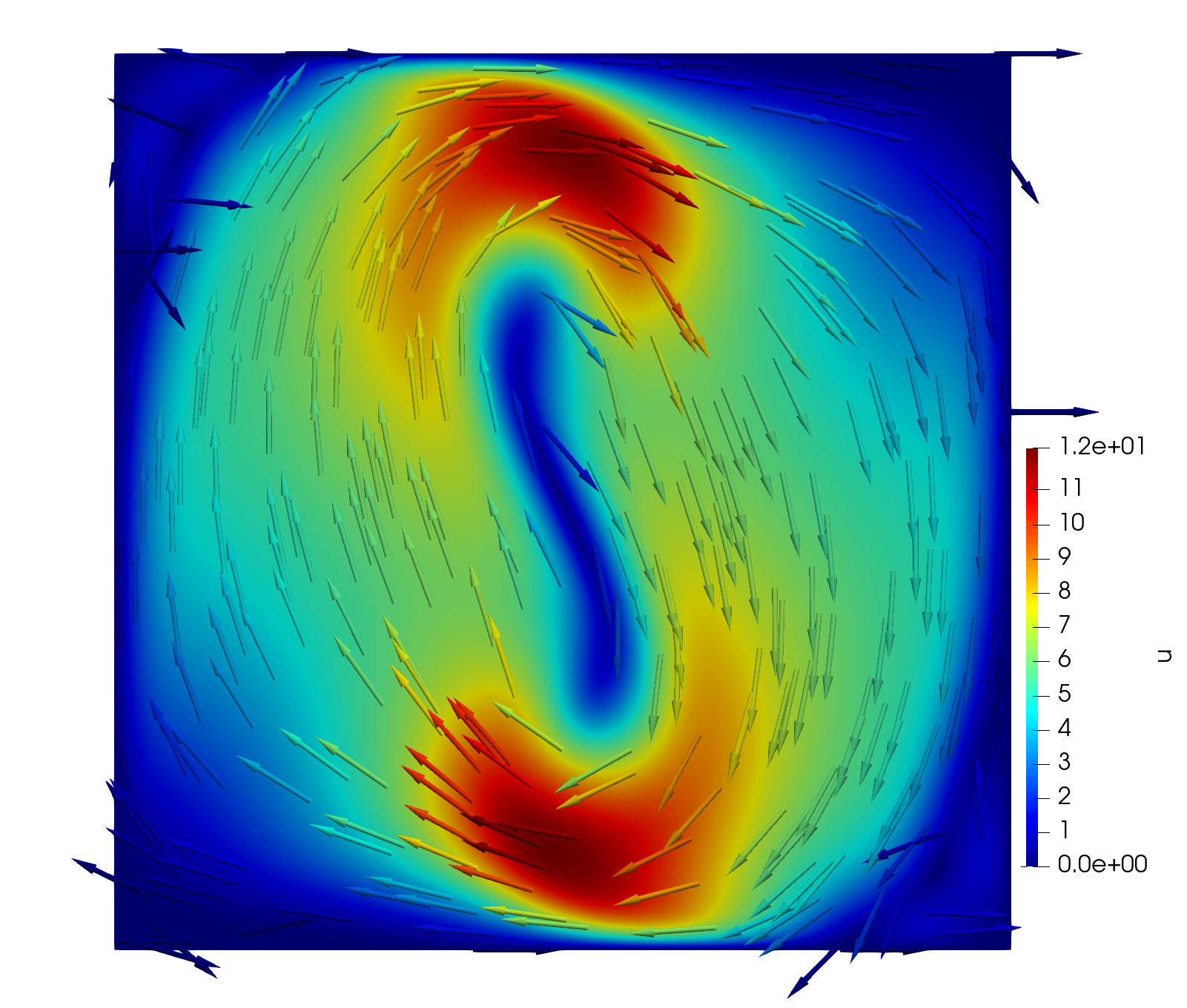}}
    \subfigure[The velocity at time $ t = 1.0 $]{\includegraphics[trim = .1cm .1cm .1cm .1cm, clip=true,width=0.38\textwidth,height=0.31\textwidth]{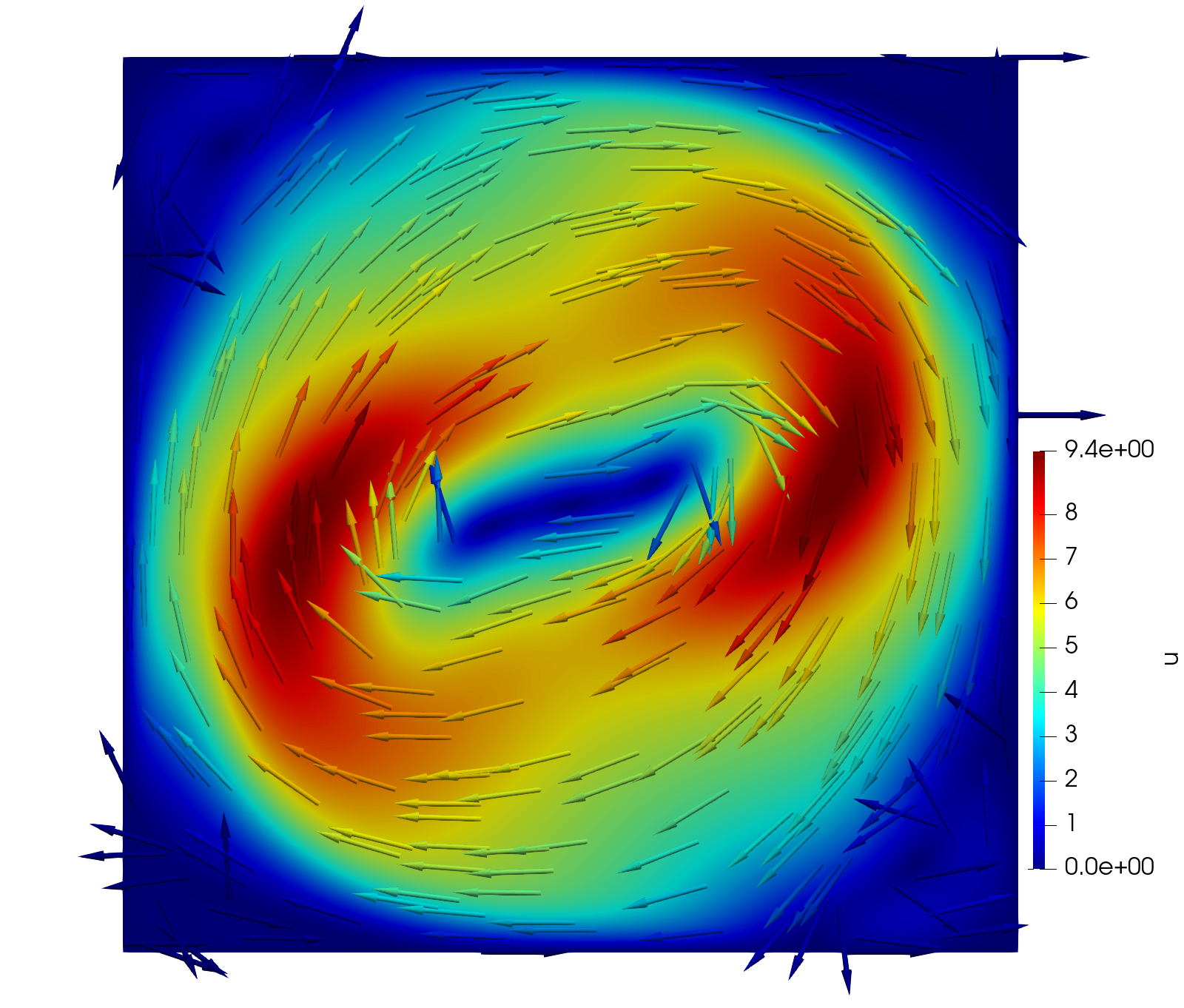}}
    \subfigure[The velocity at time $ t = 2.0 $]{\includegraphics[trim = .1cm .1cm .1cm .1cm, clip=true,width=0.38\textwidth,height=0.31\textwidth]{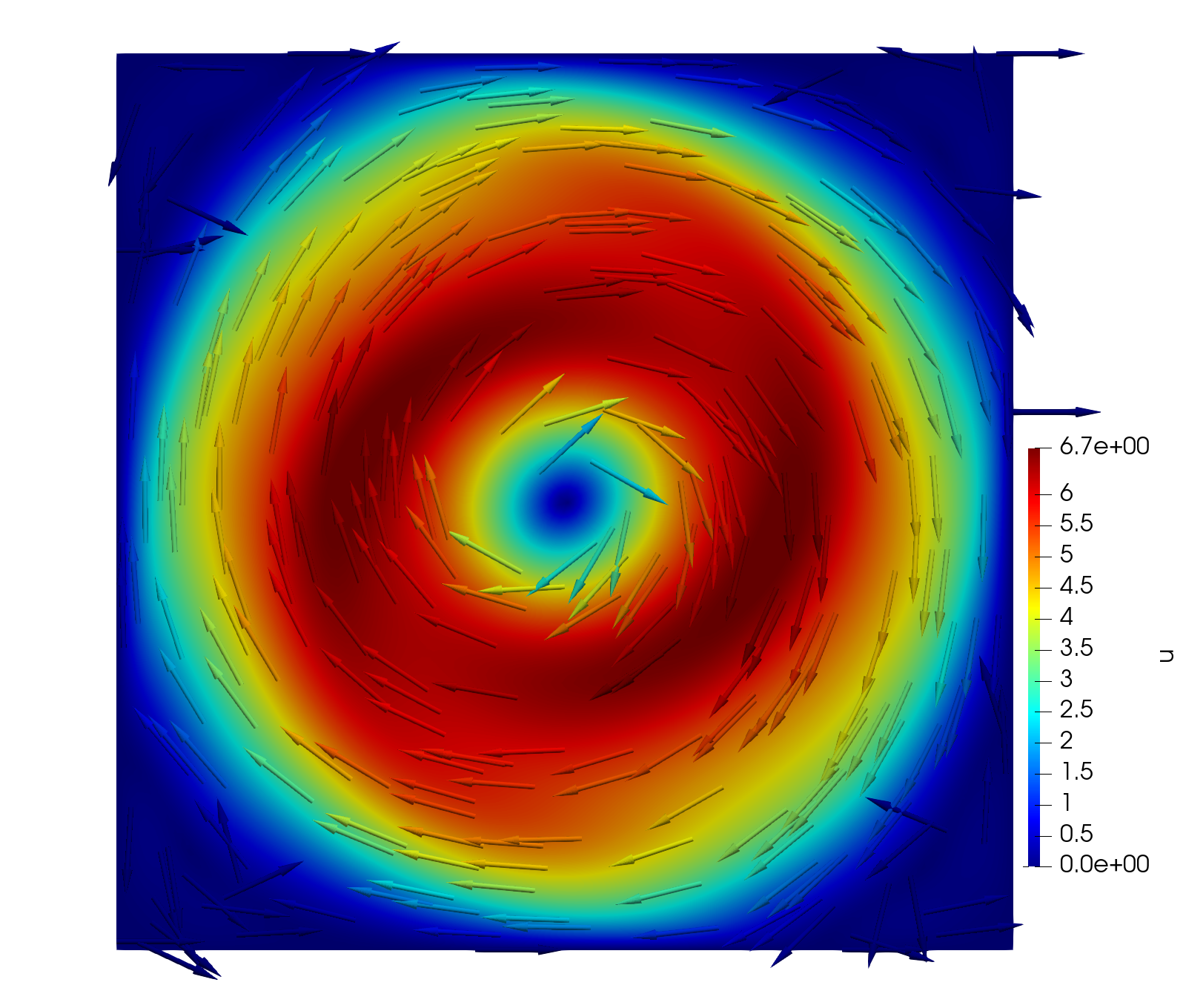}}
\caption{Isocontours of the velocity $ u $
}\label{figure2}
\end{figure}

\end{Example}

\section{Conclusion}%
\label{conclusion}
In this work, we have studied numerical treatment for  the two-dimensional Navier-Stokes equations with $L^2$ initial data. To date, the best convergence results obtained for fully discrete schemes are limited to first-order accuracy in both time and space, which are suboptimal in space and considered low-order in time.
We have proposed a fully discrete scheme that utilizes the finite element method for spatial discretization and a implicit-explicit Runge--Kutta method in conjunction with graded time meshes. By employing discrete semigroup techniques, sharp regularity estimates, negative norm estimates and the $L^2$ projection onto the divergence-free Raviart--Thomas element space, we have demonstrated that the proposed scheme attains second-order convergence in both space and time.  The argument presented in this paper could be further extended to higher-order implicit-explicit Runge--Kutta schemes. The numerical results are consistent with the theoretical analysis and demonstrate the sharpness of convergence order.

\section*{Acknowledgements}
The work of B. Li is supported in part by the Hong Kong Research Grants Council (GRF Project No. 15306123) and an internal grant of The Hong Kong Polytechnic University (Project ID: P0045404). The work of H. Zhang is supported by the National Natural Science Foundation of China (Project Nos. 12120101001,12371447,12171284), the Natural Science Foundation of Shandong Province (Project Nos. ZR2021ZD03), and the Hong Kong Research Grants Council (GRF Project No. 15301321). The work of Z. Zhou is supported by National Natural Science Foundation of China (12422117),  the Hong Kong Research Grants Council (GRF Project No. 15303122) and an internal grant of Hong Kong Polytechnic University (Project ID: P0038888).

\bibliographystyle{abbrv}

\end{document}